\documentclass[11pt]{article}

\usepackage{hyperref}
\usepackage[T1]{fontenc}
\usepackage[latin1]{inputenc}
\usepackage{amsmath,amsthm,amssymb}
\usepackage{mathabx}
\usepackage{xcolor}
\usepackage{authblk}
\usepackage[margin=1in]{geometry}
\usepackage{eufrak}
\usepackage{mathrsfs}
\usepackage{enumitem} 
\usepackage{fancybox}
\usepackage{framed}

\providecommand{\U}[1]{\protect\rule{.1in}{.1in}}

\DeclareRobustCommand{\rchi}{{\mathpalette\irchi\relax}}
\newcommand{\irchi}[2]{\raisebox{\depth}{$#1\chi$}} 

\newtheorem{prop}{Proposition}[section]
\newtheorem{cor}[prop]{Corollary}

\newtheorem{rmk}[prop]{Remark}

\newtheorem{lem}[prop]{Lemma}

\newtheorem{theo}[prop]{Theorem}

\newcommand{\vertiii}[1]{{\left\vert\kern-0.25ex\left\vert\kern-0.25ex\left\vert #1
    \right\vert\kern-0.25ex\right\vert\kern-0.25ex\right\vert}}

\def\tr{\mbox{\rm Tr}}

\newcommand{ \DD }{\mathbb{D}}
\newcommand{\EE}{\mathbb{E}}

\newcommand{\II}{\mathbb{I}}
\newcommand{\JJ}{\mathbb{J}}

\newcommand{\LL}{\mathbb{L}}

\newcommand{\PP}{\mathbb{P}}

\newcommand{\RR}{\mathbb{R}}

\newcommand{\TT}{\mathbb{T}}

\newcommand{\Na}{ {\cal N }}

\newcommand{\Ia}{ {\cal I }}

\newcommand{\Ma}{ {\cal M }}
\newcommand{\Ta}{ {\cal T}}

\newcommand{\Pa}{ {\cal P }}

\newcommand{\Wa}{ {\cal W }}

\newcommand{\point}{\mbox{\LARGE .}}

\newcommand{\cqfd}{\hfill\blbx \\}
\def\blbx{\hbox{\vrule height 5pt width 5pt depth 0pt}\medskip}

\def \PP{\mathbb{P}}
\def \RR{\mathbb{R}}
\def \SS{\mathbb{S}}
\def \EE{\mathbb{E}}
\def \EE{\mathbb{E}}

\def \LL{\mathbb{L}}

\def \WW{\mathbb{W}}

\numberwithin{equation}{section}

\makeatletter
\DeclareRobustCommand\frownotimes{\mathbin{\mathpalette\frown@otimes\relax}}
\newcommand{\frown@otimes}[2]{
  \vbox{
    \ialign{##\cr
      \hidewidth$\m@th#1{}_\frown$\kern-\scriptspace\hidewidth\cr
      \noalign{\nointerlineskip\kern-.1pt}
      $\m@th#1\otimes$\cr
    }
  }
}
\makeatother

\newcommand\quotient[2]{
        \mathchoice
            {
                \text{\raise0ex\hbox{$#1$}/\lower0ex\hbox{$#2$}}%
            }
            {
                #1\,/\,#2
            }
            {
                #1\,/\,#2
            }
            {
                #1\,/\,#2
            }
    }

\begin{document}

\title{Backward It\^o-Ventzell and stochastic interpolation formulae}
\author[$1$]{P. Del Moral\thanks{P. Del Moral was supported in part from the Chair Stress Test, RISK Management and Financial Steering, led by the French Ecole polytechnique and its Foundation and sponsored by BNP Paribas, and by the ANR Quamprocs on quantitative analysis of metastable processes.\\
Authors declaration of interests: none}}
\author[$2$]{S. S. Singh}
\affil[$1$]{{\small INRIA, Bordeaux Research Center \& CMAP, Polytechnique Palaiseau, France}}
\affil[$2$]{{\small Department of Engineering, University of Cambridge, United Kingdom.}}
\date{}

\maketitle

\begin{abstract}
We present a novel backward It\^o-Ventzell formula and an extension of the 
Alekseev-Gr\"obner interpolating formula to stochastic flows. We also present some natural spectral conditions that yield direct and simple proofs of time uniform estimates of the difference between the two stochastic flows when their drift and diffusion functions are not the same, yielding what seems to be the first results of this type for this class of  anticipative models.
  We illustrate the impact of these results in the context of diffusion perturbation theory, interacting diffusions and discrete time approximations.\\



\end{abstract}

\section{Introduction}

 Let  $b_t(x)$ be a vector-valued function from  $\RR^{d}$ into $\RR^d$ and $\sigma_{t}(x)=[\sigma_{t,1}(x),\ldots,\sigma_{t,r}(x)]$ be a matrix-valued function from  $\RR^{d}$ into $\RR^{d\times r}$, for some parameters $d,r\geq 1$. Both functions will be assumed to be differentiable.
 Let  $W_t$  be an $r$-dimensional Brownian motion and  denote by $\Wa_{s,t}$ the $\sigma$-field generated by the increments $(W_u-W_v)$ of the Brownian motion, with $u,v\in [s,t]$. 
 
   For any time horizon $s\geq 0$ we denote by
  $X_{s,t}(x)$  the stochastic flow defined for any $t\in [s,\infty[$ and any starting point $X_{s,s}(x)=x\in \RR^d$ by the stochastic differential equation
    \begin{equation}\label{diff-def}
    d X_{s,t}(x)=b_t\left(X_{s,t}(x)\right)~dt+\sigma_{t}\left(X_{s,t}(x)\right)~
    dW_t
    \end{equation}
  We assume that $ x\mapsto b_t(x)$ and $ x\mapsto \sigma_t(x)$ have continuous and uniformly bounded derivatives up to the third order. This condition is clearly met for linear Gaussian models as well as for the geometric Brownian motion. This condition ensures that the stochastic flow $
 x\mapsto X_{s,t}(x)$ is a twice differentiable function of the initialisation $x$. In addition, 
  all absolute moments of the flow and the ones of  its first and second order derivatives exists for any time horizon. 
  As it is well known, dynamical systems and hence stochastic models involving drift functions with quadratic growth require additional regularity  conditions to ensure non explosion of the solution in finite time. 
  It is also implicitly assumed that all functions  $(b_t,\sigma_t)$ are smooth functions w.r.t. the time parameter. The present article develop several constructive and stochastic analysis tools including Bismut-Elworthy-Li formulae, stochastic semigroup perturbation formulae, extended two-sided stochastic integration, Malliavin calculus, gradient and Hessian semigroup processes estimates.   We are also looking for useful quantitative and  time uniform  estimates which are valid under  a single set of easily checked conditions that only depend on the parameters of the model. Various techniques presented in the article and many results can be separately and readily extended to more general models with weaker and abstract 
  custom assumptions that depend on the different quantities to handle.

 Let $\overline{X}_{s,t}(x)$ be the stochastic flow associated with a stochastic differential equation
 defined as (\ref{diff-def}) by replacing $(b_t,\sigma_t)$ by some drift and diffusion functions $(\overline{b}_t,\overline{\sigma}_t)$ with the same regularity properties. Constant diffusion functions $(\sigma_t,\overline{\sigma}_t)$ are defined by
\begin{equation}\label{constant-sigma}
\sigma_t(x)=\Sigma_t\quad \mbox{\rm and}\quad \overline{\sigma}_t(x)=\overline{\Sigma}_t\quad \mbox{\rm for some matrices  $\Sigma_t$ and $\overline{\Sigma}_t$.}\quad 
\end{equation}
In this context, we will assume that $\Sigma_t$ and $\overline{\Sigma}_t$ are uniformly bounded w.r.t. the time horizon.

  The  Markov transition semigroups associated with the flows $X_{s,t}(x)$ and
  $\overline{X}_{s,t}(x)$  are defined for any measurable function $f$
on $\RR^d$ by the formula
$$
 P_{s,t}(f)(x):=\EE\left(f(X_{s,t}(x))\right)\quad \mbox{\rm and}\quad  \overline{P}_{s,t}(f)(x):=\EE\left(f(\overline{X}_{s,t}(x))\right)
$$
 
In this paper we derive equations for the differences $(X_{s,t}-\overline{X}_{s,t})$ and $(P_{s,t}-\overline{P}_{s,t})$
 in terms of the difference of their corresponding drifts and diffusion functions, 
 \begin{equation}\label{diff-functions}
  \Delta a_t:=a_t-\overline{a}_t\qquad
 \Delta b_t:=b_t-\overline{b}_t \quad \mbox{\rm and}\quad    \Delta \sigma_t=\sigma_t-\overline{\sigma}_t
 \end{equation}
 where  $a_t(x):=\sigma_{t}(x)~\sigma_{t}(x)^{\prime}$ and  
 $\overline{a}_t(x):=\overline{\sigma}_t(x)\,\overline{\sigma}_t^{\prime}(x)$. 
 In some applications the functions $ \overline{b}_t = b_t   - \Delta b_t$  and $ \overline{\sigma}_t = \sigma_t - \Delta \sigma_t$ can be interpreted as a local perturbation of the drift and the diffusion  of the stochastic flow ${X}_{s,t}$.
  
  We also address the problem of finding time-uniform estimates for the difference between the stochastic flows $X_{s,t}$ and $\overline{X}_{s,t}$ and their corresponding Markov transition kernels $P_{s,t}$ and $\overline{P}_{s,t}$. 
  
  These important questions arise in a variety of domains including stochastic perturbation theory as well as in the stability and the qualitative theory of  stochastic systems. Classical analytic estimates on the difference between the stochastic flows driven by different drift and diffusion functions are often much too large for most diffusion processes of practical interest. In some instances none of the diffusion flows are stable. In this context, any local perturbation of the stochastic model propagates so that any global error estimate eventually tends to $\infty$ as the time horizon $t\rightarrow\infty$. 

  Whenever one of the stochastic flows is stable, classical perturbation bounds combining Lipschitz type inequalities 
  with Gronwall lemma~\cite{bellman-2,gronwall} yield exceedingly pessimistic global estimates that grows exponentially fast w.r.t. the time horizon.
Notice that
an exponential type estimate of the form $e^{\lambda t}$ for some parameter $\lambda>0$ and some time horizon $t$ s.t. $\lambda\,t\geq 199$ would induce an error bound larger than the estimated number $10^{86}$ of elementary particles of matters in the visible universe. As mentioned in~\cite{iserles} in the context of Euler scheme type approximations of deterministic dynamical systems, one may encounter situations where $\lambda=10^8$ and $t=10^2$ and the resulting exponential bounds are clearly impractical from a numerical perspective. \\

 The statement of the main results of the article  are presented  in section~\ref{sec-smr-intro}:
 \begin{enumerate}
 \item[i.] Section~\ref{sec-1-biw-intro} presents a novel generalized backward It\^o-Ventzell formula (cf. theorem~\ref{biv}).  The It\^o-Ventzell is a very important formula, arguably as useful as the It\^o's change of variable, but surprisingly  the backward It\^o-Ventzell presented in this work has never been studied before. Theorem~\ref{biv} can be seen as a new generalized backward version of the generalized It\^o-Ventzell formula presented in~\cite{ocone-pardoux}.
\item[ii.] In section~\ref{sec-sfi-intro} we apply the backward It\^o-Ventzell formula to derive a forward-backward stochastic perturbation formula that expresses  the difference between the stochastic flows $X_{s,t}$ and $\overline{X}_{s,t}$ in terms of first and second order derivatives  of the flows, which we call the tangent and Hessian processes respectively, with respect to the space parameter (cf. theorem~\ref{theo-al-gr}).
\item[iii.] Section~\ref{sec-sfi-intro} also provides a novel forward-backward  It\^o type differential formula  for interpolating stochastic diffusion flows (cf. the change of variable formula (\ref{ref-interp-du})). 

\item[iv.] In the beginning of section~\ref{sec-sfi-intro} we present a discrete time approach based on the pivotal interpolating telescoping sum formula (\ref{ref-telescoping-sum}). This  interpolating stochastic semigroup technique can be seen as an extension to stochastic flows  of the stochastic perturbation analysis developed in~\cite{dm-2000,d-2004,dm-g-99,guionnet} and in~\cite{mp-var-18,mp-var-19,bishop-19} in the context of discrete time models, matrix and nonlinear interacting processes (see also~\cite{mp-dualtiy,mp-var-19}). For a more thorough discussion on these models, we refer to section~\ref{sec-comments}.
This approach allows to derive a stochastic interpolation formula (\ref{Alekseev-grobner}) with a fluctuation term (\ref{def-S-sk}) defined by an extended two-sided stochastic integral.  

 \item[v.] Section~\ref{sec-uewrtt}  presents some natural spectral conditions on the gradients of $b_t(x), \sigma_t(x), \overline{b}_t(x)$ and  $\overline{\sigma}_t(x)$ that allows us to derive in a direct way a series of realistic uniform estimates 
with respect to the time horizon.
\end{enumerate}

The rest of the article is organized as follows: 

Section~\ref{var-sec} provides some basic tools associated with the first and second variational equations associated with a diffusion flow.
We also present some quantitative estimates of the tangent and the Hessian processes.  For a more thorough discussion on stochastic flows and their differentiability properties we refer to~\cite{carverhill,kunita,norris}.

Section~\ref{proof-theo-al-gr} is mainly concerned with the forward-backward stochastic interpolation formula (\ref{Alekseev-grobner}) stated in theorem~\ref{theo-al-gr}. 
Two approaches are presented: The first one discussed in 
section~\ref{two-sided} is based on an extension of the two-sided stochastic calculus  introduced by Pardoux and Protter in~\cite{pardoux-protter} to stochastic interpolation flows. The second one discussed in section~\ref{ref-rig}  is based on the generalized backward It\^o-Ventzell formula. This section also discusses a multivariate Skorohod-Alekseev-Gr\"obner formula. Apart from more complex and sophisticated tensor notation, the quantitative stochastic analysis of these
multivariate formulae follows the same arguments as the ones used in the proof of theorem~\ref{theo-al-gr-2}. Thus, we have chosen to concentrate this introduction on  stochastic flows.  

Some extensions of the stochastic interpolation formula
(\ref{Alekseev-grobner})
are discussed in section~\ref{sec-extensions}.

Section~\ref{sk-section} is dedicated to the analysis of the Skorohod fluctuation process introduced in (\ref{def-S-sk}). 

Section~\ref{lem-var-proof} is dedicated to  the analysis of an extended version of two-sided stochastic integrals and a generalized backward It\^o-Ventzell formula.
 
Section~\ref{sec-illustrations} presents some illustrations of the  forward-backward  interpolation formulae discussed in the present article in the context of diffusion perturbation theory, interacting diffusions and discrete time approximations.

The technical proofs of some results are housed in the appendix.

 \subsection{Statement of some main results}\label{sec-smr-intro}

 \subsubsection{A backward It\^o-Ventzell formula}\label{sec-1-biw-intro}
We represent the gradient of a real valued function of several variables as a column vector while the gradient and the Hessian of a (column) vector valued function as tensors of type $(1,1)$ and $(2,1)$, see for instance (\ref{grad-def}) and (\ref{Hessian-def}); in more layman terms a $(1,1)$ tensor is a matrix while the $(2,1)$ tensor can be visualized as a ``row of matrices'' $[A_1,\ldots, A_n]$ where the entries $A_i$ are matrices of a common dimension. We also use the tensor product and the transpose operator defined in (\ref{tensor-notation}), see also (\ref{tensor-product-2-2}).

We denote by $D_t$ the Malliavin  
derivative  from some dense domain  $\DD_{2,1}\subset\LL_2(\Omega)$ into the space $\LL_2(\Omega\times\RR_+;\RR^r)$.  For multivariate $d$-column vector random variables $F$ with entries $F^j$, we use the same rules as for the gradient and $D_t F$ is the $(r,p)$-matrix with entries $(D_t F)_{i,j}:=D_t^iF^j$. For  $(p\times q)$-matrices $F$ with entries $F^j_{k}$ we let $D_tF$ be the tensor with entries 
$
(D_tF)_{i,j,k}=D^i _tF^j_{k}
$. 

For a more thorough discussion on Malliavin derivatives and   Skorohod integration we refer to section~\ref{sec-malliavin}.

Let $F$ be some function from $\RR^p$ into $\RR^q$, and let $y\in\RR^p$ be some given state, for some $p,q\geq 1$.  Suppose we are given a forward $p$-dimensional continuous semi-martingale $Y_{s,t}$ and a backward random  field $F_{s,t}$ from $\RR^p$ into $\RR^q$ with a column-vector type canonical representation of the following form:
  \begin{equation}\label{backward-random-fields}
 \left\{ 
 \begin{array}{rcl}
\displaystyle Y_{s,t}&=&\displaystyle y+\int_s^t~B_{s,u}~du+\int_s^t~\Sigma_{s,u}~dW_u\\
 \displaystyle  F_{s,t}(x)&=&\displaystyle F(x)+\int_s^t~G_{u,t}(x)~du+\int_s^t~H_{u,t}(x)~dW_u
   \end{array}\right.
\end{equation}
for some  $\Wa_{s,t}$-adapted functions
$B_{s,t},G_{s,t},H_{s,t},\Sigma_{s,t}$ with appropriate dimensions and satisfying
 the following conditions:\\

{\it
$(H_1)$: The functions $F_{s,t}$, $G_{u,t}$ and $H_{u,t}$ as well as $\nabla H_{u,t}$, $\nabla^2F_{u,t}$ and the derivatives  $ D_v\nabla F_{u,t}$ and $D_vH_{u,t}$  are  continuous w.r.t. the state and the time variables for any given $\omega\in\Omega$. 

$(H_2)$ The function $G_{u,t},\nabla H_{u,t}, \nabla^2 F_{u,t}$, and the  derivatives
$D_vH_{u,t},D_v\nabla F_{u,t}$ have at most polynomial growth w.r.t. the state variable, uniformly with respect to $\omega\in \Omega$.

$(H_3)$ The processes   $B_{s,u},\Sigma_{s,u}$ as well as $D_v\Sigma_{s,u}$ are continuous and have moments of any order. }\\

In this notation, the first main result of this article is the following theorem.
  \begin{theo}\label{biv}
Assume conditions $(H_i)_{i=1,2,3}$ are satisfied. In this situation, for any $s\leq u\leq v\leq t$ we have the generalized backward It\^o-Ventzell formula 
 \begin{equation}\label{backward-ito-v}
   \begin{array}{l}
   \displaystyle F_{v,t}(Y_{s,v})-F_{u,t}(Y_{s,u})   =\int_u^v(  \nabla  F_{r,t}(Y_{s,r})^{\prime}~B_{s,r} +\frac{1}{2}~\nabla^2  F_{r,t}(Y_{s,r})^{\prime}~\Sigma_{s,r}\Sigma_{s,r}^{\prime}-G_{r,t}(Y_{s,r}))~dr\\
\\
\hskip5cm+   \displaystyle\int_u^v~\left(
      \nabla F_{r,t}(Y_{s,r})^{\prime}~\Sigma_{s,r}-H_{r,t}(Y_{s,r})\right)~dW_r
      \end{array}
\end{equation}
The stochastic anticipating integral in the r.h.s. of~\ref{backward-ito-v} is understood as a Skorohod stochastic integral.
  \end{theo}

 The above theorem can be seen as  the backward version of the generalized It\^o-Ventzell formula presented in~\cite{ocone-pardoux,pardoux-90}.   The proof of the above theorem is provided in section~\ref{biv-proof} (see theorem~\ref{theo-iv}).

 Conventional forward and backward It\^o stochastic integrals are particular instances of the two-sided stochastic integrals introduced by Pardoux and Protter in~\cite{pardoux-protter}. The terminology "  two-sided " coined by the authors in~\cite{pardoux-protter} comes from the fact that the integrand of the Skorohod integral depend on the past as well as on the future of the history generated by the Brownian motion.
 
 The stochastic anticipating integral in the r.h.s. of (\ref{backward-ito-v}) involves a backward random field and a forward semimartingale, thus it is tempting to interpret this integral as a two sided integral. Unfortunately, this class of integrands are not considered in the construction of the two-sided stochastic integrals defined in~\cite{pardoux-protter}. In section~\ref{two-sided} and section~\ref{sec-extended-2-sided} we shall present an extended version of the two-sided stochastic integrals introduced in~\cite{pardoux-protter} that applies to integrands defined as a compositions of backward and forward stochastic flows. This extended 
 version applies to backward stochastic flows but it doesn't encapsulate more general backward random fields. 
We believe more general extensions of the two-sided integrals can be developed but it is out of the scope of this article to develop a  theory on generalized two-sided stochastic integrals. We finally mention that all two-sided stochastic integrals discussed in this article are particular instances of Skorohod integrals

\subsubsection{A stochastic flow interpolation formula}\label{sec-sfi-intro}

The diffusion flow (\ref{diff-def}) is defined
in term of a column vector with twice continuously differentiable entries.
For  $h\simeq 0$ we use 
 the backward approximation:
  \begin{equation}\label{Alekseev-grobner-intro-0}
 \begin{array}{l}
 X_{s,t}(x)- X_{s-h ,t}(x)
 =X_{s,t}(x)-(X_{s,t}\circ X_{s-h ,s})(x)\\
 \\
 \simeq X_{s,t}(x)-X_{s,t}\left(x+b_s(x)~h +\sigma_s(x)~\left(W_{s}-W_{s-h }\right)\right)\\
 \\
\displaystyle \simeq -\left[\left( \nabla X_{s,t}(x)^{\prime}~b_s(x)+\frac{1}{2}~\nabla^2 X_{s,t}(x)^{\prime}~a_s(x)\right)~h + \nabla X_{s,t}(x)^{\prime}\sigma_s(x)~(W_{s}-W_{s-h})\right]
\end{array}
 \end{equation}
 In the above display, $X_{s,t}\circ X_{s-h ,s}$ stands for the composition of the mappings $X_{s,t}$ and $X_{s-h ,s}$.
 
The above approximations are  rigorously justified  in section~\ref{two-sided} and lead to the backward stochastic flow evolution equation:
\begin{equation}\label{backward-synthetic}
\begin{array}{l}
d_s X_{s,t}(x)
=-\left[\left( \nabla X_{s,t}(x)^{\prime}~b_s(x)+\frac{1}{2}~\nabla ^2 X_{s,t}(x)^{\prime}~a_s(x)\right)~ds+\nabla X_{s,t}(x)^{\prime}\sigma_s(x)~dW_s\right]
\end{array}
\end{equation}
In the above display, $d_s X^i_{s,t}(x)$ represents the change in $X^i_{s,t}(x)$ w.r.t. the variable $s$.

 In the same vein,   for any $s< u< t$  we have the interpolating semigroup decompositions
$$
\begin{array}{l}
X_{u+h,t}\circ \overline{X}_{s,u+h}-X_{u,t}\circ \overline{X}_{s,u}\\
\\
=(X_{u+h,t}-X_{u,t})\circ \overline{X}_{s,u}+\left(X_{u+h,t}\circ \overline{X}_{s,u+h}-X_{u+h,t}\circ\overline{X}_{s,u}\right)
\end{array}$$
as well as the forward approximations
  \begin{equation}\label{Alekseev-grobner-intro-ref-2}
\begin{array}{l}
X_{u+h,t}\left(\,\overline{X}_{s,u}(x)+\left(\overline{X}_{s,u+h}(x)-\overline{X}_{s,u}(x)\right)\right)-X_{u+h,t}(\overline{X}_{s,u}(x))\\
\\
\displaystyle \simeq\left(\nabla X_{u+h,t}\right)(\overline{X}_{s,u}(x))^{\prime}~(\overline{X}_{s,u+h}(x)-\overline{X}_{s,u}(x))
+\frac{1}{2}\,\left(\nabla ^2X_{u+h,t}\right)(\overline{X}_{s,u}(x))^{\prime}~\overline{a}_u(\overline{X}_{s,u}(x))~h
\end{array}
\end{equation}
The above approximations are  rigorously justified  in section~\ref{two-sided} and lead to the forward-backward stochastic  interpolation equation
  \begin{equation}\label{ref-interp-du}
 \begin{array}{l}
d_u\left(X_{u,t}\circ \overline{X}_{s,u}\right)(x)\\
\\
\displaystyle=\left(d_uX_{u,t}\right)(\overline{X}_{s,u}(x))+\left(\nabla X_{u,t}\right)(\overline{X}_{s,u}(x))^{\prime}~d_u\overline{X}_{s,u}(x)
+\frac{1}{2}\,\left(\nabla ^2X_{u,t}\right)(\overline{X}_{s,u}(x))^{\prime}~\overline{a}_u(\overline{X}_{s,u}(x))~du
 \end{array}
 \end{equation}

 The discrete time version of the forward-backward stochastic formula in the above display reduces to the  telescoping sum formula (\ref{ref-telescoping-sum}) and the second order Taylor expansions discussed in section~\ref{two-sided}. We already mention that  (\ref{ref-telescoping-sum}) can be interpreted as a discrete time version of the  Alekseev-Gr\"obner lemma~\cite{Alekseev,grobner}.    The terminology {\em forward-backward} comes from the  forward and backward nature of (\ref{ref-interp-du}) and the telescoping sum formula (\ref{ref-telescoping-sum}).
 
Also notice that (\ref{backward-synthetic}) can also be deduced formally from (\ref{ref-interp-du}) by replacing $\overline{X}_{s,u}$ by the stochastic flow ${X}_{s,u}$ in  (\ref{ref-interp-du}), and then letting $s=u$. 
   
 This yields the following interpolation theorem.
\begin{theo}\label{theo-al-gr}
We have the forward-backward stochastic interpolation formula
  \begin{equation}\label{Alekseev-grobner}
  X_{s,t}(x)-\overline{X}_{s,t}(x)=T_{s,t}(\Delta a,\Delta b)(x)+S_{s,t}(\Delta \sigma)(x)
 \end{equation}
with the stochastic process 
  \begin{equation}\label{def-T-st}
 \begin{array}{l}
\displaystyle 
T_{s,t}(\Delta a,\Delta b)(x)\\
\\
\displaystyle := \int_s^t\left[\left(\nabla X_{u,t}\right)(\overline{X}_{s,u}(x))^{\prime}~\Delta b_u (\overline{X}_{s,u}(x))+\frac{1}{2}~\left(\nabla ^2X_{u,t}\right)(\overline{X}_{s,u}(x))^{\prime}~\Delta a_u(\overline{X}_{s,u}(x))\right]~du
\end{array}
 \end{equation}
and  the  fluctuation term given by the Skorohod stochastic integral
  \begin{equation}\label{def-S-sk}
 S_{s,t}(\Delta \sigma)(x):=\int_s^t~\left(\nabla X_{u,t}\right)(\overline{X}_{s,u}(x))^{\prime}~\Delta\sigma_u(\overline{X}_{s,u}(x))~dW_u
 \end{equation}
The fluctuation term in the above display can also be seen as the extended two-sided stochastic integral
defined in (\ref{sk-integral}) (see also proposition~\ref{k-prop}).
 \end{theo}
 
 These interpolation formulae combine the backward evolution (\ref{backward-synthetic}) with the conventional forward evolution of the perturbed flow.
 
The proof of the interpolation formula (\ref{Alekseev-grobner})   is provided in section~\ref{proof-theo-al-gr}.

We will present two different approaches: 
The first one presented in section~\ref{two-sided} is rather elementary and very intuitive. It combines the conventional It\^o-type discrete time approximations of stochastic integrals
discussed above with the two-sided stochastic integration calculus introduced in~\cite{pardoux-protter}. Using this approximation technique the
fluctuation term is defined by the extended two-sided stochastic integral
defined in (\ref{sk-integral}).
 In this interpretation, the equation  (\ref{Alekseev-grobner}) can be seen as an extended version of the It\^o-type change rule formula stated in theorem 6.1 in the article~\cite{pardoux-protter} to the interpolating flow
 \begin{equation}\label{interpolating-flow}
 Z^{s,t}~:~ u\in [s,t]~\mapsto~ Z^{s,t}_u:=X_{u,t}\circ \overline{X}_{s,u}\quad\Longrightarrow\quad Z^{s,t}_{s}-Z^{s,t}_{t}= X_{s,t}-\overline{X}_{s,t}
 \end{equation}
  Roughly speaking, the increments of the interpolating path
 are decomposed into two parts:
 
  One comes from the backward increments of the flow $u\mapsto X_{u,t}$ given the past values of the stochastic flow $\overline{X}_{s,u}$. 
  The other one comes from the conventional It\^o increments of $u\mapsto \overline{X}_{s,u}$ given the future values of the stochastic flow $X_{u,t}$.
  
The second approach discussed in section~\ref{ref-rig} is based on the generalized backward It\^o-Ventzell formula stated in theorem~\ref{biv}.
More precisely we also recover  (\ref{Alekseev-grobner}) from (\ref{backward-ito-v}) by choosing
\begin{eqnarray*}
(F_{s,t}(x),Y_{s,t}(y))&=&(X_{s,t}(x),\overline{X}_{s,t}(y))\qquad
(B_{s,t},\Sigma_{s,t})=\left(\overline{b}_t\left(\overline{X}_{s,t}(x)\right),\overline{\sigma}_{t}\left(\overline{X}_{s,t}(x)\right)\right)\\
 G _{u,t}(x)&=&\nabla  F_{u,t}(x)^{\prime}~b_u(x) +\frac{1}{2}~\nabla^2  F_{u,t}(x)^{\prime}~a_{u}(x)\quad \mbox{\rm and}\quad
H _{u,t}(x)=\nabla F_{u,t}(x)^{\prime}~\sigma_u(x)
\end{eqnarray*}
and letting $(u,v)=(s,t)$ in (\ref{backward-ito-v}).  The regularity conditions on the drift and the diffusion function ensure that conditions $(H_i)_i$ with $i=1,2,3$ stated in section~\ref{sec-1-biw-intro} are satisfied.

We emphasize that the backward diffusion flow discussed in (\ref{backward-synthetic}) and (\ref{ref-backward-flow}) is essential to apply theorem~\ref{biv}. Section~\ref{ref-rig} also provides a multivariate version of (\ref{Alekseev-grobner}).


The interpolation formula (\ref{Alekseev-grobner}) with a fluctuation term given 
by the Skorohod stochastic integral (\ref{def-S-sk}) can be seen as a Alekseev-Gr\"obner formula of Skorohod type.  

In this context, the integrability of the fluctuation term and any quantitative type
estimates require a refined analysis of the Malliavin derivatives of the integrand.   Under our regularity 
conditions  the stochastic flows $X_{s,t}(x)$ and $\overline{X}_{s,t}(x)$ are Holder-continuous w.r.t. the time parameters 
  as well as twice differentiable w.r.t. the space variables, with almost sure uniformly bounded first and second order derivatives. In addition, for any $n\geq 1$
  all the $n$-absolute moments of the stochastic flows 
  are finite with at most linear growth w.r.t. the initial values. 
  These properties ensure that  the Skorohod stochastic integral (\ref{def-S-sk}) is well defined and they allow to derive several quantitative estimates.  Section~\ref{sk-section} provides a refined of the fluctuation term; see for instance theorem~\ref{theo-quantitative-sko}.

When $\sigma_t=0$ the flow $X_{s,t}(x)$ is deterministic so that  the Skorohod fluctuation term (\ref{def-S-sk}) reduces to the traditional It\^o stochastic integral. In this context, 
quantitative estimates of the fluctuation term are obtained combining  Burkholder-Davis-Gundy inequalities with the generalized Minkowski inequality.
The resulting interpolation formula (\ref{Alekseev-grobner}) can be seen as a Alekseev-Gr\"obner formula of It\^o-type.  

To distinguish these two classes of models, the
 interpolation formulae (\ref{Alekseev-grobner}) associated with the case $\sigma_t=0$
will be called an It\^o-Alekseev-Gr\"obner  formula; the one associated with the case $\Delta\sigma_t\not=0$ will be called a Skorohod-Alekseev-Gr\"obner  formula.

\subsubsection{Uniform estimates w.r.t. the time horizon}\label{sec-uewrtt}

The final objective of this article is to derive uniform estimates w.r.t. the time parameter.
Our methodology is mainly based on two different types of regularity conditions to be defined and discussed in detail in section~\ref{regularity-sec}:

 $\bullet$ The first is a technical condition
that ensures that the 
$n$-absolute moments of the flows $ X_{s,t}$ and $\overline{X}_{s,t}$ are uniformly bounded w.r.t. the time horizon; we call this condition  $(M)_{n}$.   

 $\bullet$  The second is a spectral condition on the gradient of the drift and diffusion matrices of the stochastic flows, which we call condition $(T)_{n}$.   Without going into details, we state one usual case of interest: for constant diffusion functions (\ref{constant-sigma})
 the spectral condition $(T)_{n}$ is met for any $n\geq 2$  as soon as the following log-norm
conditions are met
\begin{equation}\label{T2-intro}
\nabla b_t+(\nabla b_t)^{\prime}\leq -2\lambda~I\quad \mbox{\rm and}\quad \nabla \overline{b}_t+(\nabla  \overline{b}_t)^{\prime}\leq -2\overline{\lambda}~I
\quad\mbox{\rm
for some $\lambda\wedge\overline{\lambda}>0,$}
\end{equation}

To motivate the above condition consider a linear drift function of the form $b_t(x)=B_t~x$
and $\sigma=0$. In this case the tangent process $\nabla X_{s,t}(x)$ satisfies a time-varying deterministic linear dynamical system 
$$
\partial_t  \,\nabla X_{s,t}(x)=\nabla X_{s,t}(x)~ B_t^{\prime}
$$
The asymptotic behavior of this process  cannot be characterized by the statistical properties of the spectral abscissa of the matrices $B_t$. Indeed, unstable semigroups  associated with time-varying (deterministic) matrices $B_t$ with negative eigenvalues are exemplified in~\cite{coppel1978stability,wu1974note}. Conversely, stable semigroups with $B_t$ having positive eigenvalues are given by Wu in~\cite{wu1974note}. 
 In contrast, the uniform log-norm condition (\ref{T2-intro}) provides a readily verifiable condition.

 To describe with some precision the second main result of the  article, we need to introduce some additional terminology.   
 When there is no ambiguity, we denote by $\Vert\point\Vert$ any (equivalent) norm on some finite dimensional vector space.
 For some multivariate function $f_t(x)$, for $(t,x)\in [0,\infty)\times\RR^d$, 
let $\Vert f(x)\Vert:=\sup_{t}\Vert f_t(x)\Vert$ and the uniform norm be $\Vert f\Vert:=\sup_{t,x}\Vert f_t(x)\Vert$. 
For any $n\geq 1$ we also set
\begin{equation}\label{ref-vertiii}
\vertiii{f(x)}_n:=\sup_{s\geq 0}\sup_{t\geq s}\EE\left(\Vert  f_t (\overline{X}_{s,t}(x))\Vert^{n}\right)^{1/n}
\end{equation}
We denote by $\kappa_n$ and $\kappa_{\delta,n}$   some constants that depend on some parameters $n$ and $(\delta,n)$ but  do not depend  on the time horizon, nor on the space variable.

In this notation, the second main result of the article takes basically the following form.

\begin{theo}\label{theo-al-gr-2}
Assume conditions $(M)_{2n/\delta}$ and  $(T)_{2n/(1-\delta)}$ are satisfied for some parameters $n\geq 2$ and $\delta\in ]0,1[$. In this situation, we have the time-uniform estimates
  \begin{equation}\label{intro-inq-1}
   \begin{array}{l}
\displaystyle
  \EE\left[\Vert   X_{s,t}(x)-\overline{X}_{s,t}(x)\Vert^{n}\right]^{1/n}\\
  \\
\displaystyle\leq  \kappa_{\delta,n}~\left(\vertiii{\Delta a(x)}_{2n/(1+\delta)}+\vertiii{\Delta b(x)}_{2n/(1+\delta)}
+\vertiii{\Delta\sigma(x)}_{2n/\delta}~(1\vee\Vert x\Vert)\right)\end{array} 
 \end{equation}
For constant diffusion functions (\ref{constant-sigma}),  the estimate simplifies to 
  \begin{equation}\label{intro-inq-2}
 \displaystyle  (\ref{T2-intro})\Longrightarrow \forall n\geq 2\quad\EE\left[\Vert   X_{s,t}(x)-\overline{X}_{s,t}(x)\Vert^{n}\right]^{1/n}
\leq \kappa_n~\left(\vertiii{\Delta b(x)}_{n}+\Vert \Sigma-\overline{\Sigma}\Vert\right)    
 \end{equation}
 \end{theo}

The estimates (\ref{intro-inq-1}) 
come from (\ref{intro-inq-0}) and (\ref{intro-inq-s}).    A more detailed proof  is provided in the appendix, on page~\pageref{intro-inq-1-proof}.  
The estimates (\ref{intro-inq-2})
are direct consequences of (\ref{intro-inq-0-0}) and (\ref{intro-inq-s-2-2}).

When $\sigma_t=\overline{\sigma}_t$ the Skorohod  term is indeed absent and (\ref{Alekseev-grobner}) reduces to
  \begin{equation}\label{Alekseev-grobner-sigma-d}
 X_{s,t}(x)-\overline{X}_{s,t}(x)=\int_s^t~\left(\nabla X_{u,t}\right)(\overline{X}_{s,u}(x))^{\prime}~\Delta b_u (\overline{X}_{s,u}(x))~du
 \end{equation}
We recover the interpolation formula for nonlinear stochastic flows presented in section 3.1 in the article~\cite{mp-var-18}.
In this context the analysis of $\LL_n$-errors will proceed via two-step procedure. In section \ref{tangent-sec} we will  derive the exponential bound 
 $$\sup_{x} \EE(\Vert \left(\nabla X_{u,t}\right)(x) \Vert_2^n)^{1/n}\leq \kappa_n \exp(-\lambda(n)~(t-u))\quad\mbox{\rm for some} \quad \lambda(n)>0$$
 Using the Minkowski integral inequality in (\ref{Alekseev-grobner-sigma-d}) yields 
\begin{eqnarray*}
 \EE \left [\Vert  X_{s,t}(x)-\overline{X}_{s,t}(x) \Vert^n \right]^{1/n} 
  &\leq &
 \int_s^t~ \EE \left [\Vert \left(\nabla X_{u,t}\right)(\overline{X}_{s,u}(x)) \Vert^n\times\Vert \Delta b_u (\overline{X}_{s,u}(x))\Vert^n \right]^{1/n}~du.
 \end{eqnarray*}
A further conditioning argument and the above exponential bound on the tangent process yields   
   \begin{equation*}
\EE \left [\Vert  X_{s,t}(x)-\overline{X}_{s,t}(x) \Vert^n \right]^{1/n} \leq \kappa_n~\int_s^t  \exp(-\lambda(n)~(t-u))du~~  
\sup_{s\leq u} \EE [\Vert \Delta b_u (\overline{X}_{s,u}(x))\Vert^n]^{1/n}.
 \end{equation*} Replacing the term outside the time integral with $\vertiii{\Delta b (x)}_n$ yields the stated result in (\ref{intro-inq-1}) excluding the terms representing the difference in the diffusions.

We illustrate  one use of theorem \ref{theo-al-gr}  in the context of analyzing the error in discretising the diffusion $ X_{s,t}(x)$ for some initial time point $s\geq 0$. 
Let $h>0$ denote the discretisation interval size and for any $ t\in [s+kh,s+(k+1)h[$ let
$$
  d X_{s,t}^h(x)=Y^h_{s,t}(x)~dt+  \Sigma~  dW_t\quad \mbox{\rm with}\quad Y^h_{s,t}(x):=b\left(X^h_{s,s+kh}(x)\right)
$$
for a fixed diffusion matrix $\sigma_t(x)=\Sigma$. Here $X_{s,t}^h(x)$ is the discretisation of $X_{s,t}(x)$ with resolution $h$. Note that that the drift at time $t$ is not a function of the instantaneous value of $X_{s,t}^h(x)$, at time $t$, but rather the value it took at the largest discrete time-point before $t$. In section \ref{sec-extensions} we discuss how the formula in (\ref{Alekseev-grobner}) also applies in this context and establish that 
$$
X^h_{s,t}(x)-X_{s,t}(x)=\int_s^t~\left(\nabla X_{u,t}\right)(X^h_{s,u}(x))^{\prime}~
~\left[Y^h_{s,u}(x)-b(X^h_{s,u}(x))\right]~
du.
$$ This comparison result when combined with the regularity assumptions (\ref{eq:lem_discretize}) yields the moment bound below. 
\begin{prop}\label{lem:discretize}
Assume that 
\begin{equation}\label{eq:lem_discretize}
\nabla b+(\nabla b)^{\prime}\leq -2\lambda~I\qquad
\Vert \nabla b\Vert:=\sup_{x}{\Vert \nabla b(x)\Vert}<\infty \quad \mbox{\rm and}\quad 
\langle x,b(x)\rangle\leq -\beta~\Vert x\Vert^2 
\end{equation}
for some $\lambda>0$, $\beta>0.$ In this situation, for any $n\geq 1$ we have the uniform estimates
 $$
\EE\left(\Vert X^h_{s,t}(x)-X_{s,t}(x)\Vert^n\right)^{1/n}\leq
\Vert \nabla b\Vert~\left(\left[\Vert  b(0)\Vert+\widehat{m}_n(x)~\Vert  \nabla b\Vert \right]~h+\sigma~\sqrt{h}\right)
/\lambda
$$ where $\widehat{m}_n(x) \leq \kappa_n~(1+\Vert x\Vert)$. 
\end{prop}  

Proposition~\ref{lem:discretize} is proved in section \ref{subsec:lemdiscretizeproof}. To apply proposition~\ref{lem:discretize} to a Langevin diffusion with a convex potential $U(x)$, the drift would be $b_t(x)=-\nabla U(x)$ and the corresponding assumptions on $U(x)$ are typical.

\subsection{Comments and comparisons with existing literature}\label{sec-comments}

The interpolation formula (\ref{Alekseev-grobner}) can be interpreted as an extension of Alekseev-Gr\"obner lemma~\cite{Alekseev,grobner,jentzen} as well as an extended version of the variation-of-constant and related
Gronwall type lemma~\cite{bellman-2,gronwall} to diffusion processes. In this connection we underline that the forward-backward formula (\ref{Alekseev-grobner})  differs from the stochastic Gronwall lemma presented in~\cite{scheutzow} based on particular classes of stochastic linear inequalities that doesn't involve Skorohod type integrals. 

The forward-backward interpolation formula (\ref{Alekseev-grobner}) can also be seen as an extension of theorem 6.1 in~\cite{pardoux-protter} on two-sided stochastic integrals to diffusion flows.  This interpolation formula can also be interpreted as a backward version of the generalized It\^o-Ventzell formula presented in~\cite{ocone-pardoux} (see also theorem 3.2.11 in~\cite{nualart}).

Stochastic interpolation formulae of the form (\ref{Alekseev-grobner}) and their discrete time version discussed in (\ref{ref-telescoping-sum}) are not really new.
To describe their origins, it is worth to mention that the stochastic perturbations may come from auxiliary random sources, uncertainty propagations, as well as time discretization schemes and mean field type particle fluctuations.

The pivotal interpolating telescoping sum formula (\ref{ref-telescoping-sum}) and the second order forward-backward perturbation semigroup methodology discussed in the present article can also be found in  chapter 7 in~\cite{d-2004} for discrete time models as well as in the series of articles~\cite{dm-g-99,guionnet,dm-2000} published at the beginning of the 2000s, see also chapter 10 in~\cite{d-2013}.  In this context, the random perturbations come from the fluctuations of a genetic type particle interpretation of nonlinear Feynman-Kac semigroups.

The more recent articles~\cite{bishop-stab,bishop-18,bishop-19} also provide a series of backward-forward interpolation formulae of the same form as (\ref{Alekseev-grobner}) for stochastic matrix Riccati diffusion flows arising in data assimilation theory (cf. for instance theorem 1.3 in~\cite{bishop-19} as well as section 2.2 in~\cite{bishop-18} and the proof of theorem 2.3 in~\cite{bishop-stab}). In this context, the random perturbations come from the fluctuations of a mean field particle interpretation of a class of nonlinear diffusions equipped with an interacting sample covariance matrix functional.

We underline that the It\^o-Alekseev-Gr\"obner formula (4.6)  discussed in~\cite{bishop-19} 
is an extension of the interpolation formula (\ref{Alekseev-grobner}) to stochastic diffusion flows in matrix spaces. In this context the unperturbed model is given by the 
flow of a deterministic matrix Riccati differential equation and the random perturbations are described by matrix-valued diffusion martingales. The corresponding It\^o-Alekseev-Gr\"obner formulae  can be seen as a matrix version of theorem 1.2 in the present article when $\sigma=0$.
 These stochastic interpolation formulae were used in~\cite{bishop-19} to quantify the fluctuation of the stochastic flow around the limiting deterministic Riccati equation, at any order. We will briefly discuss the analog of these Taylor type expansions in section~\ref{sec-perturbation} in the context of Euclidian diffusions.

The forward-backward perturbation methodology discussed in the present article has also been used in~\cite{mp-var-18,mp-var-19} in the context of nonlinear diffusions and their mean field type interacting particle interpretations, see for instance section 2.3 in~\cite{mp-var-19}.
 In this context, the random perturbations come from the fluctuations of a mean field particle interpretation of a class of nonlinear diffusions. 
The extended version of the It\^o-Alekseev-Gr\"obner formula (\ref{Alekseev-grobner-sigma-d}) to nonlinear diffusions is also discussed in section 3.1 in the article~\cite{mp-var-18}. In this situation, the time varying drift and diffusion functions of the stochastic flows 
 depend on some possibly different nonlinear measure valued semigroups which may start from two possibly different initial distributions. For a more thorough discussion on this class of nonlinear diffusions, we refer to the  It\^o-Alekseev-Gr\"obner formula (3.2) and corollary 3.2 in the article~\cite{mp-var-18}. These  It\^o-Alekseev-Gr\"obner formulae  correspond to theorem 1.2 in the present article when $\sigma=0$.
 
 The interpolating stochastic semigroup techniques discussed in the present article are also applied to mean field particle systems and deterministic nonlinear measure valued semigroups. In this context, the process $X_{s,t}$ is given a deterministic measure-valued process and $\overline{X}_{s,t}$ represents the evolution of the particle density profiles associated with an approximating mean field particle interpretation of $X_{s,t}$. 
 For instance, the article ~\cite{mp-dualtiy} is concerned with interacting jumps models on path spaces, the second article~\cite{mp-var-19} discusses the propagation of chaos properties of mean field type interacting diffusions. The stochastic interpolation formulae discussed in ~\cite{mp-dualtiy,mp-var-19} correspond to the case (\ref{Alekseev-grobner}) with $\sigma=0$ and or $\overline{\sigma}\not=\sigma$  (see for instance the interpolation formula (3.5), theorem 2.6, theorem 2.7 and the interpolating telescoping sum in section 1.2 in~\cite{mp-var-19})

 In the series of articles discussed above, as in  (\ref{ref-interp-du}) the central common idea  is to analyse the evolution of the interpolating process (\ref{interpolating-flow}) between a given process $X_{s,t}$ and some  stochastic flow $\overline{X}_{s,t}$ with an extra level of randomness. In discrete time settings, the differential interpolation formula (\ref{ref-interp-du}) can also recasted  in terms of a telescoping sum of the same form as (\ref{ref-telescoping-sum}) combined with a second order Taylor expansion reflecting the differences between a stochastic semigroup and its perturbations, see for instance  chapter 7 in~\cite{d-2004}. 
 
 In most of the application domains discussed above, this second order stochastic perturbation methodology has been developed to quantify uniformly w.r.t. the time horizon the propagations of some stochastic perturbations entering in {\em some deterministic and stable reference or unperturbed  process.} In the context of Euclidian diffusions, this corresponds to the situation where the diffusion function $\sigma=0$ (the case $\overline{\sigma}=0$ can be treated by symmetry arguments). 
 The It\^o-Alekseev-Gr\"obner type formulae  discussed in section 3.1 in the article~\cite{mp-var-18} correspond to theorem 1.2 in the present article when $\sigma=\overline{\sigma}$.

The present article can be seen as a natural extension of the second order perturbation methodology developed in the above referenced articles to diffusion type perturbed processes when  $\sigma\not=\overline{\sigma}$. 
 
 To the best of our knowledge, the first article considering the case $\sigma\not=\overline{\sigma}$ with $\sigma\not=0$ and $\overline{\sigma}\not=0$ is the independent work of Hudde-Hutzenthaler-Jentzen-Mazzonetto~\cite{hudde}. In this article,
 the authors discuss an It\^o-Alekseev-Gr\"obner formula for abstract diffusion perturbation models of the form (\ref{X-Y-Z}). 
 Here again, as in the list of referenced articles discussed above, the common central idea is to use discrete time approximations and combine  the pivotal interpolating telescoping sum formulae  (\ref{ref-telescoping-sum}) with a second order Taylor expansion. Besides this fact and in contrast with our analysis, the fluctuation term (\ref{def-S-sk}) discussed in~\cite{hudde} cannot be interpreted in terms of the extended two-sided stochastic integral
defined in (\ref{sk-integral}) (see also proposition~\ref{k-prop}) but only in terms of a Skorohod stochastic integral.
 The study~\cite{hudde} is also based on a series of particularly chosen and custom regularity conditions. 
  For instance, the authors assume that the abstract diffusion perturbation models are chosen so that the Skorohod fluctuation term exists 
  without providing any quantitative type estimate. This work is also not connected to 
the  two-sided stochastic integration calculus developed by  Pardoux and Protter in~\cite{pardoux-protter} nor to any type of backward It\^o-Ventzell formula.
  
  We feel that our approach is more direct and intuitive as it relies on an extended version
  of It\^o's change rule formula (\ref{ref-interp-du})  to interpolating stochastic flows. It also allows to interpret the fluctuation term (\ref{def-S-sk}) as an extended two-sided stochastic integral.
  
  In section~\ref{sk-section} in the present article, we will also see that any quantitative analysis requires to estimate the absolute moments of the Malliavin derivatives of the stochastic integrands of the Brownian motion arising in  the Skorohod fluctuation term. In our framework, these Malliavin derivatives  depend on the gradient of both of the diffusion functions $(\sigma,\overline{\sigma})$ as well as on the tangent process of the perturbed diffusion flow. The quantitative analysis  developed in~\ref{sk-section} can be extended without difficulties to abstract diffusion perturbation models satisfying appropriate differentiability and integrability conditions.
  
The article~\cite{hudde} also presents an application to  tamed Euler type 
discrete time approximations of a  stochastic van-der-Pol process introduced in~\cite{tim}, 
 simplifying the analysis provided in an earlier work~\cite{hutz-14}. In this situation, we underline that the Skorohod fluctuation term is null so that  the resulting Alekseev-Gr\"obner type formula resumes to the simple and elementary case discussed in (\ref{Alekseev-grobner-sigma-d}) and in the article~\cite{mp-var-18}.
 As expected for this class of "unstable processes",
the authors recast a series of $\LL_2$-estimates discussed in~\cite{hutz-14} into a series of estimates that grow exponentially fast with respect to the time horizon.

 In contrast with the present work, the above article doesn't discuss any quantitative uniform estimates w.r.t. the time horizon. The analysis presented in~\cite{hudde} is mainly concerned with the proof of a Skorohod-Alekseev-Gr\"obner type formula  for abstract diffusion perturbation models and it doesn't apply to derive any type of estimates to  general diffusion perturbation models without adding regularity conditions.
 
Besides its elegance the forward-backward interpolation formula (\ref{Alekseev-grobner}) is clearly of rather poor mathematical and numerical interest without a better understanding
of the variational processes and the Skorohod fluctuation term (\ref{def-S-sk}). 
A crucial problem is to avoid exceedingly pessimistic exponential estimates that grow exponentially fast w.r.t. the time horizon. 

One advantage of the second order perturbation methodology developed in the present article is that it takes advantage of the stability properties of the 
tangent and the Hessian flow in the estimation of Skorohod fluctuation term and this sharpen analysis of the difference between stochastic flows. 
Our main contribution  is to develop a refined analysis of these variational processes and the Skorohod fluctuation terms. We also deduce several uniform perturbation propagation  estimates with respect to the time horizon, yielding what seems to be the first results of this type for this class of models. 

The forward-backward stochastic interpolation formula (\ref{Alekseev-grobner}) can also be extended to more general classes of stochastic flows on abstract state spaces. For instance the recent article~\cite{jentzen} provides a deterministic first order version of (\ref{Alekseev-grobner}) on abstract Banach spaces. The stochastic perturbation analysis developed in the series of articles~\cite{mp-dualtiy,mp-var-19,bishop-stab,bishop-18,bishop-19,dm-g-99,guionnet,dm-2000} and the books~\cite{d-2004,d-2013} is applied to matrix-valued diffusions and measure valued processes, including mean field type interacting diffusions and Feynman-Kac type interacting jumps models.

The stability properties of these abstract models discussed above depend on the problem at hand.
To focus on the main ideas without clouding the article with unnecessary technical details and sophisticated mathematical tools based on abstract ad hoc regularity conditions we have chosen to concentrate the article on diffusion flows on Euclidian spaces with simple and easily checked regularity conditions.

\section{Preliminary results}

\subsection{Some basic notation}\label{notation-sec}
With a slight abuse of notation, we denote by $I$ the identity $(d\times d)$-matrix, for any $d\geq 1$.     We also denote by $\Vert\point\Vert$ any (equivalent) norm on a finite dimensional vector space over $\RR$.  All vectors are column vectors by default.

We introduce some matrix notation needed from the onset. 

We denote by $\tr(A)$, $\Vert A\Vert_{2}:=\lambda_{\tiny max}(AA^{\prime})^{1/2}=\lambda_{\tiny max}(A^{\prime}A)^{1/2}$, resp. $\Vert A\Vert_{F}=\tr(AA^{\prime})^{1/2}$ and $\rho(A)=\lambda_{\tiny max}((A+A^{\prime})/2)$ the trace, the spectral norm, the Frobenius norm, and the logarithmic norm of some matrix $A$.  $A^{\prime}$ is the transpose of $A$ and $\lambda_{\tiny max}(\point)$ the largest eigenvalue. The spectral norm is sub-multiplicative or $\Vert A B\Vert_{2}\leq \Vert A\Vert_{2} \Vert B\Vert_{2}$ and compatible with the Euclidean norm for vectors, by that we mean for a vector $x$ we have  $\Vert A x \Vert \leq \Vert A\Vert_{2} \Vert x\Vert$.

Let $[n]$ be the set of $n$ multiple indexes $i=(i_1,\ldots,i_n)\in \Ia^n$ over some finite set $\Ia$.
We denote by $(A_{i,j})_{(i,j)\in [p]\times [q]}$ the  entries of a $(p,q)$-tensor $A$ with index set $\Ia$ for $[p]$ and $\mathcal{J}$ for $[q]$. For the sake of brevity, the index sets will be implicitly defined through the context.

For a given $(p_1,q)$-tensor $A$ and  a given $(q,p_2)$ tensor $B$, $AB$ and $B^{\prime}$ is a $(p_1,p_2)$-tensor resp. a $(p_2,q)$-tensor with entries
given by
\begin{equation}\label{tensor-notation}
\forall (i,j)\in [p_1]\times [p_2]\qquad
(AB)_{i,j}=\sum_{k\in [q]}A_{i,k}~B_{k,j}\quad \mbox{\rm and}\quad B_{j,k}^{\prime}:=B_{k,j}.
\end{equation}
The symmetric part $A_{\tiny sym}$ of a $(p,p)$-tensor is the $(p,p)$-tensor $A_{\tiny sym}$ with entries
 $$\forall (i,j)\in [p]\times [p]\qquad(A_{\tiny sym})_{i,j}=(A_{i,j}+A_{j,i})/2$$
We consider the Frobenius inner product given for any $(p,q)$-tensors $A$ and $B$ by
$$
\langle A,B\rangle_F=\tr(AB^{\prime})=\sum_{i}(AB^{\prime})_{i,i}\quad \mbox{\rm and the norm}\quad \Vert A\Vert_F=\sqrt{\tr(AA^{\prime})}
$$
For any $(p,q)$-tensors $A$ and $B$  we also check the Cauchy-Schwartz inequality
$$
\langle A,B\rangle_F^2\leq \Vert A\Vert_F~\Vert B\Vert_F\quad\mbox{\rm and}\quad  \Vert A\Vert_{2}\leq\Vert A\Vert_F\leq  \mbox{\rm Card}(\Ia)^p~ \Vert A\Vert_{2}
\quad\mbox{\rm with}\quad
\Vert A\Vert_2:=\lambda_{\tiny max}(AA^{\prime})^{1/2}
$$
For any tensors $A,B$ with appropriate dimensions we have the inequality
$$
\Vert AB\Vert_F\leq \Vert A\Vert_F~\Vert B\Vert_F
$$
Given some tensor valued function $T:(t,x)\mapsto T_t(x)$ we also set
$$
\Vert T\Vert_{F}:=\sup_{t,x}\Vert T_t(x)\Vert_F\qquad \Vert T\Vert_{2}:=\sup_{t,x}\Vert T_t(x)\Vert_2
\qquad\mbox{\rm and}\qquad \Vert T\Vert:=\sup_{t,x}\Vert T_t(x)\Vert
$$
    Given some smooth function $h(x)$ from $\RR^{p}$ into $\RR^{q}$
we denote by 
\begin{equation}\label{grad-def}
\nabla h=\left[\nabla h^1,\ldots,\nabla h^q\right]\quad \mbox{\rm with}\quad\nabla h^i=\left[\begin{array}{c}
\partial_{x_1}h^i\\
\vdots\\
\partial_{x_p}h^i
\end{array}\right]
\end{equation} 
the gradient $(p,q)$-matrix associated with the column vector-valued
 function $h=(h^i)_{1\leq i\leq q}$. Building on this notation: let $b:\RR^n \rightarrow \RR^p$ and let the mapping $x \rightarrow G(x) = h(b(x))$. Then $\nabla G(x) = \nabla b(x) \times \nabla h (b(x))$. Let
\begin{equation}\label{Hessian-def}
\nabla^2 h=\left[\nabla^2 h^1,\ldots,\nabla^2 h^q\right]\quad \mbox{\rm with}\quad\nabla^2 h^i=\left[\begin{array}{ccc}
\partial_{x_1,x_1}h^i&\ldots&\partial_{x_1,x_p}h^i\\
\vdots&\ldots&\vdots\\
\partial_{x_p,x_1}h^i&\ldots&\partial_{x_p,x_p}h^i\
\end{array}\right]
\end{equation} 
The Hessian $H=\nabla^2 h$ associated with the
 function $h=(h^i)_{1\leq i\leq q}$ is a $(2,1)$-tensor where $H_{(i,j),k}=(\nabla^2 h^k)_{i,j}=\partial_{x_i,x_j}h^k$.  In this notation we can compactly represent the second order term of the Taylor expansion of the the vector valued function $h$. For a vector $y=(y_1,\ldots,y_p)'$
 \[
 \left [
 \begin{array}{c}
 y^{\prime}~\nabla^2 h^1(x)~y \\
 \vdots \\
 y^{\prime}~\nabla^2 h^q(x)~y \\ 
 \end{array}
 \right ] = \nabla^2 h(x)^{\prime}~yy^{\prime}
 \]
 where we have regarded the matrix $yy'$ as the $(2,1)$-tensor $Y$ with $Y_{(i,j),1}=y_iy_j$. 
 
 In the same vein, in terms of the  tensor product  (\ref{tensor-notation}), for any pair of column vector-valued
 function $h=(h^k)_{1\leq k\leq q}$ and $b=(b^i)_{1\leq i\leq p}$ and any matrix function $a=(a^{i,j})_{1\leq i,j\leq p}$  from $\RR^{p}$ into $\RR^{q}$, for any parameter $1\leq k\leq q$ we also have
 $$
 \begin{array}{rcl}
 \displaystyle\left( \nabla h(x)^{\prime}~b(x)\right)^k&=& \displaystyle\sum_{1\leq i\leq p}( \nabla h(x))^{\prime}_{k,i}~b^i(x)=\sum_{1\leq j\leq p}~\partial_{x_i} h^k(x)~b^i(x)=\langle \nabla h^k(x),b(x)\rangle\\
 \\
 \displaystyle\left(\nabla^2 h(x)^{\prime}~a(x)\right)^k&=& \displaystyle\sum_{1\leq i,j\leq p}( \nabla^2 h(x))^{\prime}_{k,(i,j)}~a^{i,j}(x)\\
 &&\\
 &=& \displaystyle\sum_{1\leq i,j\leq p}~ \partial_{x_i,x_j} h^k(x)~a^{i,j}(x)=\langle \nabla^2h^k(x),a(x)\rangle_{F}
\end{array}
 $$
 In a more compact form, the above formula takes the form
  \begin{equation}\label{tensor-product-2-2}
  \nabla h(x)^{\prime}~b(x)=\left[\begin{array}{c}
\langle\nabla h^1(x),b(x)\rangle\\
\vdots\\
\langle\nabla h^q(x),b(x)\rangle
\end{array}\right]\quad\mbox{\rm and}\quad
  \nabla^2 h(x)^{\prime}~a(x)=\left[\begin{array}{c}
\langle\nabla^2h^1(x),a(x)\rangle_F\\
\vdots\\
\langle\nabla^2h^q(x),a(x)\rangle_F
\end{array}\right]
\end{equation} 
  
  For any $n\geq 1$ we let $\Pa_n(\RR^d)$ be the convex set of probability measures $\mu_1,\mu_2$ on $\RR^d$ with absolute $n$-th moment
and equipped  with the Wasserstein distance  of order $n$ denoted by
$$
\WW_n(\mu_1,\mu_2):=\inf\EE(\Vert X_1-X_2\Vert^n)^{1/n}
$$
 In the above display the infimum is taken over all pair or random variables $(X_1,X_2)$ with marginal distributions $(\mu_1,\mu_2)$.
   The stochastic  transition semigroups associated with the flows $X_{s,t}(x)$ and $\overline{X}_{s,t}(x)$ are defined for any measurable function $f$
on $\RR^d$ by the formulae
$$
\PP_{s,t}(f)(x):=f(X_{s,t}(x))\quad\mbox{\rm and}\quad \overline{\PP}_{s,t}(f)(x):=f(\overline{X}_{s,t}(x))
$$

Given some column vector-valued 
 function $f=(f^i)_{1\leq i\leq p}$, let $\PP_{s,t}(f)$ and  $P_{s,t}(f)$ denote the column vector-valued
 functions with entries $\PP_{s,t}(f^i)$ and $P_{s,t}(f^i)$. Building on the tensor notation,  let $\PP_{s,t}(\nabla f)$ and $\PP_{s,t}(\nabla^2 f)$ 
respectively denote  the $(1,1)$ and $(2,1)$-tensor valued functions with entries
 $$
 \PP_{s,t}(\nabla f)(x)_{i,k}:=
 \PP_{s,t}(\partial_{x_i} f^k)(x)\quad\mbox{\rm and}\quad \PP_{s,t}(\nabla^2 f)(x)_{(i,j),k} :=\PP_{s,t}(\partial_{x_i,x_j} f^k)(x)
 $$
We also consider the random $(2,1)$ and $(2,2)$-tensors 
 given by
\begin{eqnarray*}
\nabla^2\, X_{s,t}(x)_{(i,j),k}&=&\partial_{x_i,x_j}X^{k}_{s,t}(x)=\left[\nabla^2\, X_{s,t}(x)\right]^{\prime}_{k,(i,j)}\\
\left[\nabla X_{s,t}(x)\otimes \nabla X_{s,t}(x)\right]_{(i,j),(k,l)}&=&\nabla X_{s,t}(x)_{i,k}\nabla X_{s,t}(x)_{j,l}=\left[\nabla X_{s,t}(x)\otimes \nabla X_{s,t}(x)\right]^{\prime}_{(k,l),(i,j)}
\end{eqnarray*}

  Throughout the rest of the article,  unless otherwise stated $\kappa,\kappa_{\epsilon},\kappa_n,\kappa_{n,\epsilon}$ denote
 constants whose values may vary from line to line but  only depend on the parameters in their subscripts, i.e. $n\geq 0$  and $\epsilon>0$, as well as on the parameters of the model; that is, on the drift and diffusion functions.
  We also use the letters
  $c,c_{\epsilon},c_n,c_{n,\epsilon}$ to denote universal constants.
 Importantly these contants do not depend  on the time horizon.
  We also consider the uniform log-norm parameters 
\begin{equation}\label{def-nabla-sigma}
\rho(\nabla\sigma)^2:=
\sum_{1\leq k\leq r}\rho(\nabla\sigma_{k})^2
\quad\mbox{\rm and}\quad \rho_{\star}(\nabla \sigma):=\sup_{1\leq k\leq r}\rho(\nabla \sigma_{k})\quad\mbox{\rm with}\quad \rho(\nabla\sigma_{k}):=\sup_{t,x}\rho(\nabla \sigma_{t,k}(x))
\end{equation}
and the parameters $\rchi(b,\sigma)$ defined  by
\begin{equation}\label{def-chi-b}
\rchi(b,\sigma):=c+\Vert \nabla^2b\Vert+\Vert \nabla^2\sigma\Vert^2
+ \rho_{\star}(\nabla \sigma)^2
\end{equation}

 \subsection{Regularity conditions and some preliminary results}\label{regularity-sec}

We consider two different types of regularity conditions ($\Ma$)$_n$ and $(\Ta)_n$, indexed by some parameter 
$n\in [2,\infty[$, for the diffusion $(b_t,\sigma_t)$.
\\

\begin{description}
\item[$(\Ma)_n$] There exists some parameter $\kappa_n\geq 0$  such that for any $x\in \RR^d$ we have
$$
m_n(x):=\sup_{s\leq t}\EE\left(\Vert X_{s,t}(x)\Vert^{n}\right)^{1/n}\leq \kappa_n~(1\vee \Vert x\Vert)
$$
\item[$(\Ta)_n$] There exists some parameter $\lambda_A>0$ such that 
   \begin{equation}\label{ref-mat-Upsilon}
A_t:=\nabla b_t+(\nabla b_t)^{\prime}+\sum_{1\leq k \leq r}\nabla \sigma_{k,t}(\nabla\sigma_{k,t})^{\prime}\leq -2\lambda_A~I
 \end{equation} 
 where $\sigma_{k,t}$ denotes the $k$-th column of $\sigma_{t}.$
  In addition, the following condition is satisfied
   \begin{equation}\label{ref-lambda-A-n}
  \lambda_A(n):=\lambda_A-\frac{d(n-2)}{2}~ \rho_{\star}(\nabla \sigma)^2>0
 \end{equation} \end{description} 

We now define the corresponding assumptions for the diffusion $(\overline{b}_t,\overline{\sigma}_t)$.
 
 \begin{description}
 \item[$(\overline{\Ma})_n$] The regularity condition defined as in $(\Ma)_n$ for the diffusion $(\overline{b}_t,\overline{\sigma}_t)$.
 
 \item[$(\overline{\Ta})_n$] Let $\overline{A}_t$ be the symmetric matrix defined as $A_t$ in \eqref{ref-mat-Upsilon} when 
  $({b}_t,{\sigma}_t) = (\overline{b}_t,\overline{\sigma}_t)$. Assume there exists some $\lambda_{\overline{A}}>0$ such that $\overline{A}_t \leq -2 \lambda_{\overline{A}}~I $. 
  Furthermore, assume $ \lambda_{\overline{A}}(n)>0$  where $ \lambda_{\overline{A}}(n)$ is defined as $\lambda_A(n)$ when 
  $(\lambda_A,{\sigma}_t)=(\lambda_{\overline{A}},\overline{\sigma}_t)$.   
 
 \item[$(M)_n$] We write $(M)_n$  when both conditions $({\Ma})_n$ and $(\overline{\Ma})_n$ are satisfied.
  
  \item[$(T)_n$] Both conditions $({\Ta})_n$ and $(\overline{\Ta})_n$ are met, and let
$$
 \lambda_{A,\overline{A}}(n):=\lambda_{A}(n)\wedge \lambda_{\overline{A}}(n)
 $$
 \end{description}

In practice, the uniform moment condition $(\Ma)_n$ is often checked using Lyapunov techniques. For example we can use the following 
polynomial growth condition.   \\
 
\begin{description}
\item[$(\Pa)_n$] There exists some parameters $\alpha_i,\beta_i\geq 0$ with $i=0,1,2$ such that 
 for any $t\geq 0$
and any $x\in\RR^d$ we have
\begin{equation}\label{def-alpha-beta}
\Vert \sigma_{t}(x)\Vert_F^2\leq \alpha_0+\alpha_1\Vert x\Vert+\alpha_2\Vert x\Vert^2
\quad\mbox{and}\quad
\langle x,b_t(x)\rangle\leq \beta_0+\beta_1\Vert x\Vert-\beta_2\Vert x\Vert^2
\end{equation}
for some norm $\Vert  \sigma_{t}(x) \Vert$ of the matrix-valued diffusion function.  In addition, we have
  $$
\beta_2(n):=  \beta_2-\frac{(n-1)}{2}~\alpha_2>0
  $$
\end{description}

\begin{lem} For any $n\geq 2$ we have
\begin{equation}\label{moments-intro}
(\Pa)_n\quad\Longrightarrow\quad (\Ma)_n\quad \mbox{\rm with}\quad \kappa_n=1+\frac{(\gamma_1+(n-2)\alpha_1)+(\gamma_0+(n-2)\alpha_0)^{1/2}}{2\beta_2(n)^{1/2}}
\end{equation}
\end{lem}
 The proof of the above assertion follows standard stochastic calculations, thus it is housed in the appendix, on page~\pageref{moments-intro-proof}.

For  one-dimensional geometric Brownian motions the condition 
$(\Pa)_n$ is a sufficient and necessary condition for the existence of uniformly bounded absolute $n$-moments. In this case $(\Ta)_n$ coincides with
$(\Pa)_n$ by setting $$\lambda_A=\beta_2-\alpha_2/2\quad \mbox{\rm and}\quad   \alpha_2= \rho_{\star}(\nabla \sigma)^2$$

 Whenever condition $(M)_n$  is met for some $n\geq 2$, we also check the uniform estimates
  \begin{equation}\label{ref-ui-m-over}
\EE\left(\Vert [X_{u,t}\circ \overline{X}_{s,u}](x)\Vert^{n}\right)^{1/n}\leq \kappa_n~(1+\Vert x\Vert)
 \end{equation}
 with the same parameter $ \kappa_n$ as the one associated with the condition $(M)_n$.
 
Recalling that the functions $(b_t,\overline{b}_t)$ and $(\sigma_t,\overline{\sigma}_t)$ have at most linear growth,
with the $\LL_n$-norms $\vertiii{\point}_n$ introduced in (\ref{ref-vertiii}) we also have that
  \begin{equation}\label{ref-ui-m-over-delta}
\vertiii{\Delta b(x)}_n\leq \kappa_{1,n}(1\vee\Vert x\Vert)\quad \mbox{\rm and}\quad \vertiii{\Delta a(x)}_{n/2}\leq \kappa_{2,n}~(1\vee\Vert x\Vert)^2
 \end{equation}

To give more insight where these assumptions will be used, we now briefly state the stability results that stem from them.  Condition $(\Ta)_n$ ensures that the exponential decays of the absolute and uniform $n$-moments of the tangent and the Hessian processes; that is, when  $(\Ta)_n$ is met for some $n\geq 2$ we have that
  \begin{equation}\label{ref-tan-hess}
 \EE\left(\Vert \nabla X_{s,t}(x)\Vert^{n}\right)^{1/n}\vee\EE\left(\Vert \nabla^2 X_{s,t}(x)\Vert^{n}\right)^{1/n}\leq \kappa_n~e^{-\lambda(n) (t-s)}\quad \mbox{\rm for some}\quad\lambda(n)>0
 \end{equation}
A more precise statement is provided in proposition~\ref{def-4th-prop} and proposition~\ref{eq-prop-nabla-2-estimate}. 
These uniform estimates clearly imply, via a conditioning argument, that for any $n\geq 2$ and $s\leq u\leq t$ we have
  \begin{equation}\label{intro-inq-nabla}
 \EE\left(\Vert (\nabla X_{u,t})(\overline{X}_{s,u}(x))\Vert^{n}\right)^{1/n}\vee\EE\left(\Vert (\nabla^2 X_{u,t})(\overline{X}_{s,u}(x))\Vert^{n}\right)^{1/n}\leq \kappa_n~e^{-\lambda(n) (t-u)}
 \end{equation}
with the same parameters $(\kappa_n,\lambda(n))$ as in (\ref{ref-tan-hess}).

The case $\nabla \sigma=0$ will also serve a useful purpose, for example in analysing the error of a numerical implementation as in proposition \ref{lem:discretize}. For instance whenever $(\Ta)_2$ is met we have
 the almost sure and uniform gradient estimates
  \begin{equation}\label{ref-nablax-estimate-0-ae}
 \Vert\nabla X_{s,t}\Vert_2:=\sup_x\Vert\nabla X_{s,t}(x)\Vert_2\leq e^{-\lambda_A(t-s)}
 \end{equation}
 In addition, we have  the almost sure and uniform Hessian estimates
  \begin{equation}\label{ref-nabla2x-estimate-0-ae}
\Vert\nabla^2 X_{s,t}\Vert_F:=\sup_x\Vert\nabla^2 X_{s,t}(x)\Vert_F\leq \frac{d}{\lambda_A}~\Vert\nabla^2b\Vert_{F}~e^{-\lambda_A(t-s)}
 \end{equation}
 A proof of the above estimates is provided in the beginning of section~\ref{tangent-sec} and section~\ref{hessian-sec}. In this situation, whenever $(\Ta)_{2}$ is met  we have
  \begin{equation}\label{intro-inq-0-0} 
 \EE\left[\Vert   T_{s,t}(\Delta a, \Delta b)(x)\Vert^n\right]^{1/n}   \leq \kappa~\left(\vertiii{\Delta b(x)}_{n}+\vertiii{\Delta a(x)}_{n}\right).    
 \end{equation}
 In the above display, $T_{s,t}(\Delta a, \Delta b)(x)$ stands for the stochastic process discussed in (\ref{def-T-st}), and $\kappa$ stands for some finite constant that doesn't depend on the parameter $n$.
 For instance, for a Langevin diffusion associated with some convex potential function $U$ we have $b=-\nabla U$ and $\nabla \sigma=0$. Then 
 assuming 
   \begin{equation}\label{ref-HA-Langevin}
       \begin{array}{l}
    \nabla^2 U\geq \lambda~I  \quad \Longrightarrow \quad (\Ta)_2 \quad \mbox{\rm is met}\\
   \\
\displaystyle \Longrightarrow \quad  \Vert\nabla X_{s,t}\Vert_2\leq e^{-\lambda(t-s)}\quad\mbox{\rm and}\quad
\Vert\nabla^2 X_{s,t}\Vert_F\leq \frac{d}{\lambda}~\Vert\nabla^3U\Vert_{F}~e^{-\lambda(t-s)}
   \end{array}
 \end{equation} where the almost sure tangent and Hessian bounds follow from (\ref{ref-nablax-estimate-0-ae}) and (\ref{ref-nabla2x-estimate-0-ae}) respectively. 

In practice, it is often easier to work with $a_t(x)=\sigma_t(x)\sigma_t(x)'$ than $\sigma_t(x)$ and we now discuss some ways of estimating $\Delta  \sigma_t(x) = \sigma_t(x) -\overline{\sigma}_t(x)$
in terms of $\Delta  a_t(x) =a_t(x) -\overline{a}_t(x)$ and in the reverse direction. The latter is straightforward: 
 $$
\Vert \Delta  a_t(x) \Vert\leq \Vert  \Delta  \sigma_t (x)\Vert~\left[ \Vert \sigma_t(x)\Vert+ \Vert \overline{\sigma}_t(x)\Vert\right].
 $$
To estimate $\Delta  \sigma_t$ in terms of  $\Delta  a_t$, assume the following
ellipticity condition is satisfied
\begin{equation}\label{elip}
a_t(x)\geq \upsilon~I\quad \mbox{\rm and}\quad \overline{a}_t(x)\geq \upsilon~I
\quad\mbox{\rm
for some parameter $\upsilon>0$.}
\end{equation}
We recall the Ando-Hemmen inequality~\cite{hemmen} for any symmetric positive definite matrices $Q_1,Q_2$ 
\begin{equation}\label{square-root-key-estimate}
\Vert Q_1^{1/2}-Q_2^{1/2}\Vert \leq \left[\lambda^{1/2}_{min}(Q_1)+\lambda^{1/2}_{min}(Q_2)\right]^{-1}~\Vert Q_1- Q_2\Vert
\end{equation}
for any unitary invariant matrix norm $\Vert . \Vert$. In the above display, $\lambda_{\tiny min}(\point)$ stands for the minimal eigenvalue. We also have
the square root inequality
\begin{equation}\label{square-root-inq}
Q_1\geq Q_2\Longrightarrow Q_1^{1/2}\geq Q_2^{1/2}
\end{equation}
See for instance theorem 6.2 on page 135 in~\cite{higham}, as well as proposition 3.2  in~\cite{hemmen}.   A proof of (\ref{square-root-inq}) can be found in~\cite{bellman}.
 In this situation, using (\ref{square-root-key-estimate}) and (\ref{square-root-inq})  we check that
  \begin{equation}\label{elip-ref-est}
 \Vert  \Delta  \sigma_t (x)\Vert \leq \frac{1}{\sqrt{ \upsilon}}~\Vert  \Delta  a_t(x) \Vert\quad \mbox{\rm and}\quad 
\Vert \sigma_t(x)\Vert\leq \Vert \sigma_t(0)\Vert+\frac{1}{\sqrt{ \upsilon}}~\left[\Vert a_t(x)\Vert+\Vert a_t(0)\Vert\right]
\end{equation}
 This provides a way to estimate the growth of $\sigma_t(x)$ in terms of the one of $a_t(x)$. 
  For instance the estimate (\ref{intro-inq-1}) combined with (\ref{elip-ref-est}) implies that
 $$
  \EE\left[\Vert   X_{s,t}(x)-\overline{X}_{s,t}(x)\Vert^{n}\right]^{1/n}\leq  \kappa_{\delta,n}~\left(\vertiii{\Delta b(x)}_{2n/(1+\delta)}+\vertiii{\Delta a(x)}_{2n/\delta}(1\vee\Vert x\Vert)
\right)
 $$

$\bullet$ Assume that $(\overline{\Ma})_n$ is satisfied for some $n\geq 1$. Also let $f_t(x)$ be some multivariate function such that
 $$\Vert f(0)\Vert:=\sup_t\Vert f_t(0)\Vert<\infty\quad\mbox{\rm and} \quad\Vert \nabla f\Vert:=\sup_{t,x}\Vert \nabla f_t(x)\Vert<\infty$$
In this situation, we have the estimates
$$
\vertiii{f(x)}_n\leq \Vert f(0)\Vert+\Vert \nabla f\Vert~\overline{m}_n(x)\quad \mbox{\rm and therefore}\quad
\vertiii{f(x)}_n\leq \kappa_n~(\Vert f(0)\Vert+\Vert \nabla f\Vert)~(1\vee \Vert x\Vert)
$$

\subsection{Some results on anticipating stochastic calculus}\label{sec-malliavin}
In this section we review some results on Malliavin derivatives and Skorohod integration calculus which will be needed below. 
We restrict the presentation to unit time intervals.
Let $(\Omega,\Wa)$ be the canonical space equipped with the Wiener measure $\PP$ associated with the $r$-dimensional Brownian motion $W_t$ discussed in the introduction.

The Malliavin  
derivative $D_t$ is a linear operator from some dense domain  $\DD_{2,1}\subset\LL_2(\Omega)$ into the space $\LL_2(\Omega\times [0,1];\RR^r)$ of $r$-dimensional processes 
with square integrable states on the unit time interval.    For multivariate $d$-column vector random variables $F$ with entries $F^i$, we use the same rules as for the gradient and we set
$$
D_tF=\left[D_tF^1,\ldots,D_tF^d\right]\quad \mbox{\rm with}\quad D_tF^i=\left[\begin{array}{c}
D^1_tF^i\\
\vdots\\
D^rF^i
\end{array}\right]
$$
For  $(p\times q)$-matrices $F$ with entries $F^j_{k}$ we let $D_tF$ be the tensor with entries 
$$
(D_tF)_{i,j,k}=D^i _tF^j_{k}
$$
 It is clearly out of the scope of this article to review the 
  analytical construction of Malliavin differential calculus. For a more thorough discussion we refer the  reader to the seminal book by Nualart~\cite{nualart}, see also
  the more synthetic presentation in the articles~\cite{nualart-pardoux,ocone-pardoux}.

Formally, one can think the Malliavin derivatives $D_{t}^iF$ of some $F\in \DD_{2,1}$ as way to extract from the random variable $F$ the integrand of Brownian increment $dW^i_t$. 
 For instance, when $s\leq t$ we have
 \begin{eqnarray}
 D_{t}^iX_{s,t}(x)&=&\sigma_{t,i}( X_{s,t}(x))\nonumber\\
  (D_{t}\,\nabla X_{s,t}(x))_{i,j,k}&=& D_{t}^i\,(\nabla X_{s,t}(x))_{j,k}:=\left(\nabla X_{s,t}(x)~\nabla\sigma_{t,i}( X_{s,t}(x))\right)_{j,k}~
\label{first-Malliavin}
 \end{eqnarray}
 As conventional differentials, for any smooth function $G$ from $\RR^d$ into $\RR^{p\times q}$, Malliavin derivatives satisfy the chain rule properties 
  $$
 D_t^i(G^j_{k}\circ F)=\sum_{1\leq l\leq d}\left(\partial_{x_l}G_{k}^j\right)(F)\times D_t^iF^l\quad\Longleftrightarrow\quad  D_t(G\circ F)=D_tF~((\nabla G)\circ F)
 $$
 For instance, for any $s\leq u\leq v$ we have
\begin{equation}
 D_u\left(X_{u,t}\circ X_{s,u}\right)=\left(D_u X_{s,u}\right)~\left[\left(\nabla X_{u,t}\right)\circ X_{s,u}\right]~~\mbox{\rm and}~~
D_{u}\left(\varsigma_{t}\circ X_{s,t}\right)=(D_u X_{s,t})~\left[\left(\nabla \varsigma_{t}\right)\circ X_{s,t}\right]\label{ref-chain-r}
 \end{equation}
  In the same vein, we have
   \begin{equation}\label{s-tensor-diff-u-v}
  \begin{array}{l}
D_{u}\left( \nabla X_{s,u}~\left[\left(\nabla X_{u,t}\right)\circ X_{s,u}\right]\right)\\
 \\
 \displaystyle=(D_{u}\nabla X_{s,u})
 \left[\left(\nabla X_{u,t}\right)\circ X_{s,u}\right]+\left(D_{u}X_{s,u}\otimes \nabla X_{s,u}\right)\left[\left(\nabla^2X_{u,t}\right)\circ X_{s,u}\right]
 \end{array}
 \end{equation}
 
 Let $\LL_{2,1}(\RR^r)\subset \LL_2(\Omega\times [0,1];\RR^r)$  be the Hilbert space of $r$-dimensional process $U_t$ with Malliavin differentiable 
 entries $U^i_t\in\DD_{2,1}$ equipped with the norm
 $$
 \vertiii{U}:=\EE\left(\int_{[0,1]}~\Vert U_t\Vert^2~dt \right)^{1/2}+\EE\left(\int_{[0,1]^2}~\Vert D_sU_t\Vert^2~ds\,dt\right)^{1/2}
 $$
 
 The Skorohod integral w.r.t. the Brownian motion $W^i_t$ on the unit interval is defined a linear and continuous mapping from $$V\in \LL_{2,1}(\RR)\mapsto \int_0^1V_t~dW_t^i\in \LL_2(\Omega)$$ characterized by the
 two following properties
 \begin{eqnarray}
 \EE\left(\int_0^1V_t~dW_t^i\right)&=&0\nonumber\\
  \EE\left(\left(\int_0^1V_t~dW_t^i\right)^2\right)&=&\EE\left(\int_{[0,1]}~V_t^2~dt \right)+\EE\left(\int_{[0,1]^2}~D^i_sV_t~D^i_{t}V_s~ds\,dt \right)\label{isometry}
 \end{eqnarray}
 The above formula can be seen as an extended version of the It\^o isometry to Skorohod integrals, for instance~\cite{nualart-z}, as well as chapters 1.3 to 1.5 in the book by Nualart~\cite{nualart}.   
 
 As for the It\^o integral, the Skorohod integral w.r.t. the $r$-dimensional Brownian motion $W_t$ of a matrix valued process with entries
 $V^i_{k}\in  \LL_{2,1}(\RR)$ is defined by the column vector with entries
 $$
\left(  \int_0^1V_t~dW_t\right)^i:=
 \int_0^1V_t^i~dW_t:=\sum_{1\leq k\leq r} \int_0^1V^i_{t,k}~dW_t^k
 $$

   \section{Variational equations}\label{var-sec}

  \subsection{The tangent process}\label{tangent-sec}
     In terms of the tensor product  (\ref{tensor-product-2-2}),   the gradient $\nabla X_{s,t}(x)$ of the diffusion flow $X_{s,t}(x)$ is given by the gradient $(d\times d)$-matrix
$$
d  \,\nabla X_{s,t}(x)=\nabla X_{s,t}(x)~\left[\nabla b_t\left(X_{s,t}(x)\right)~dt+\sum_{1\leq k\leq r}\nabla \sigma_{t,k}\left(X_{s,t}(x)\right)~dW^k_t\right]
$$
where $W^k_t$ is the $k$-th component of the Brownian motion. After some calculations we check that
\begin{equation}\label{def-C}
 \begin{array}{l}
\displaystyle d  \,\left[\nabla X_{s,t}(x) \,\nabla X_{s,t}(x)^{\prime}\right]
=\nabla X_{s,t}(x) ~A_t\left(X_{s,t}(x)\right)~\nabla X_{s,t}(x)^{\prime}~dt+d M_{s,t}(x)
\end{array} 
\end{equation} 
with the matrix function $A_t(x)$ defined in (\ref{ref-mat-Upsilon}) and the symmetric matrix valued martingale 
$$
d M_{s,t}(x):=\sum_{1\leq k\leq r}~\nabla X_{s,t}(x) \left[\nabla\sigma_{t,k}\left(X_{s,t}(x)\right)+\nabla\sigma_{t,k}\left(X_{s,t}(x)\right)^{\prime}\right]\nabla X_{s,t}(x)^{\prime}~dW^k_t
$$
These expansions, when combined with condition $(\Ta)_2$, yield the following estimates of the difference between $X_{s,t}(x)$ and $X_{s,t}(y)$.

\begin{prop}
Assume $(\Ta)_2$ is satisfied. Then
  \begin{equation}\label{ref-nablax-estimate-0-bis}
 \EE\left(\Vert X_{s,t}(x)-X_{s,t}(y)\Vert^2\right)^{1/2}\leq \sqrt{d}~e^{-\lambda_A(t-s)}~~\Vert x-y\Vert.
 \end{equation}
In addition, we have the almost sure estimate
  \begin{equation}\label{ref-nablax-estimate-0-ae-bis}
 \nabla\sigma=0\Longrightarrow \Vert X_{s,t}(x)-X_{s,t}(y)\Vert\leq e^{-\lambda_A(t-s)}~~\Vert x-y\Vert
 \end{equation}
 \label{prop:diff_innit}
 \end{prop}
 
\begin{proof} [Proof of Prop.\ \ref{prop:diff_innit}]
Whenever $(\Ta)_2$ is met, we have the following uniform estimate from (\ref{def-C})  
  \begin{equation}\label{ref-nablax-estimate-0}
(\Ta)_2\Longrightarrow \EE\left(\Vert\nabla X_{s,t}(x)\Vert_2^2\right)^{1/2}\leq \EE\left(\Vert\nabla X_{s,t}(x)\Vert_F^2\right)^{1/2}\leq \sqrt{d}~e^{-\lambda_A(t-s)}
 \end{equation}
 where the $\sqrt{d}$  term arises from imposing the initial condition $\nabla X_{s,s}(x)=I$ on the resulting differential equation for 
 $\partial_t  \EE\left(\Vert\nabla X_{s,t}(x)\Vert_F^2\right)^{1/2}$.
In addition, when $\nabla \sigma=0$ the martingale $M_{s,t}(x)=0$ is null and  as a consequence of (\ref{def-C}) we have the following almost sure estimate
 \begin{equation}\label{ref-nablax-estimate-0-ae-again}
 \Vert\nabla X_{s,t}\Vert_2:=\sup_x\Vert\nabla X_{s,t}(x)\Vert_2\leq e^{-\lambda_A(t-s)}
 \end{equation}
The Taylor expansion
$$
\begin{array}{l}
\displaystyle X_{s,t}(x)-X_{s,t}(y)=\int_0^1~\nabla X_{s,t}(\epsilon x+(1-\epsilon)y)^{\prime}(x-y)~d\epsilon\\
\\
\displaystyle\Longrightarrow \Vert X_{s,t}(x)-X_{s,t}(y)\Vert^2\leq  \left[\int_0^1~\Vert \nabla X_{s,t}(\epsilon x+(1-\epsilon)y)\Vert_2^2~d\epsilon\right]~\Vert x-y\Vert^2
\end{array}
$$
combined with (\ref{ref-nablax-estimate-0}) and (\ref{ref-nablax-estimate-0-ae-again}) completes the proof.
\end{proof}

  These contraction inequalities quantify the stability of the stochastic flow $X_{s,t}(x)$  w.r.t. the initial state $x$.
For instance,  the estimate (\ref{ref-nablax-estimate-0-bis}) ensures that the Markov transition semigroup is exponentially stable; that is, we have that
   \begin{equation}\label{ref-eta-mu-cv}
   \WW_2\left(\mu_0 P_{s,t},\mu_1 P_{s,t}\right)\leq c~\exp{\left[-\lambda_A(t-s)\right]}~  \WW_2\left(\mu_0 ,\mu_1 \right)
 \end{equation}
 For the Langevin diffusions discussed in (\ref{ref-HA-Langevin}) the stochastic flow is time homogeneous; that is we have that $X_{s,t}=X_{t-s}:=X_{0,(t-s)}$
  and $P_{s,t}=P_{t-s}:=P_{0,(t-s)}$.
 In addition when $\sigma(x)=\sigma~I$, the diffusion flow $X_t(x)$ has a single invariant measure on $\RR^d$ given by
 the Boltzmann-Gibbs measure
  \begin{equation}\label{ref-gibbs}
 \pi(dx)=\frac{1}{Z}~\exp{\left(-\frac{2}{\sigma^2}\,U(x)\right)}~dx\quad \mbox{\rm with}\quad Z:=\int~~ e^{-\frac{2}{\sigma^2}U(x)}~dx
 \end{equation}
From (\ref{ref-HA-Langevin}), it follows that
 $$
  \nabla^2 U\geq \lambda~I\quad\Longrightarrow\quad
   \WW_n\left(\mu  P_{s,t},\pi\right)\leq \exp{\left[-\lambda(t-s)\right]}~  \WW_n\left(\mu ,\pi \right)
 $$ for all $n \geq 1$.

 Taking the trace in (\ref{def-C}) we also find that
$$
 \begin{array}{l}
\displaystyle d  \,\Vert \nabla X_{s,t}(x)\Vert^2_F
=\tr\left[\nabla X_{s,t}(x) ~A_t\left(X_{s,t}(x)\right)~\nabla X_{s,t}(x)^{\prime}\right]~dt+d N_{s,t}(x)
\end{array} 
$$
with the martingale
$$
d N_{s,t}(x)=\sum_{1\leq k \leq r} \tr\left(\nabla X_{s,t}(x) \left[\nabla\sigma_{t,k}\left(X_{s,t}(x)\right)+\nabla\sigma_{t,k}\left(X_{s,t}(x)\right)^{\prime}\right]\nabla X_{s,t}(x)^{\prime}\right)~dW^k_t
$$
Observe that
$$
\partial_t\langle N_{s,\point}(x)\rangle_t=\sum_k~\tr\left(\nabla X_{s,t}(x) \left[\nabla\sigma_{t,k}\left(X_{s,t}(x)\right)+\nabla\sigma_{t,k}\left(X_{s,t}(x)\right)^{\prime}\right]\nabla X_{s,t}(x)^{\prime}\right)^2
$$
This implies that
$$
 \begin{array}{l}
\displaystyle \partial_t\EE\left(\Vert \nabla X_{s,t}(x)\Vert^4_F\right)
=2~\EE\left(\Vert \nabla X_{s,t}(x)\Vert^2_F~\tr\left[\nabla X_{s,t}(x) ~A_t\left(X_{s,t}(x)\right)~\nabla X_{s,t}(x)^{\prime}\right]\right)\\\
\\
\hskip3cm\displaystyle+\sum_{1\leq k\leq r}\EE\left(\tr\left(\nabla X_{s,t}(x) \left[\nabla\sigma_{t,k}\left(X_{s,t}(x)\right)+\nabla\sigma_{t,k}\left(X_{s,t}(x)\right)^{\prime}\right]\nabla X_{s,t}(x)^{\prime}\right)^2\right)
\end{array} 
$$
Whenever $(\Ta)_2$ is met, we have the estimate
$$
 \begin{array}{l}
\displaystyle \partial_t\EE\left(\Vert \nabla X_{s,t}(x)\Vert^4_F\right)
\leq -4\left[\lambda_A-\rho(\nabla\sigma)^2\right]~\EE\left(\Vert \nabla X_{s,t}(x)\Vert^4_F\right)
\end{array} 
$$
with the uniform log-norm parameter $\rho(\nabla\sigma)$ defined in (\ref{def-nabla-sigma}).
This yields the estimate
$$
 \partial_t\EE\left(\Vert \nabla X_{s,t}(x)\Vert^4_F\right)^{1/4}
\leq -\left[\lambda_A-\rho(\nabla\sigma)^2\right]~\EE\left(\Vert \nabla X_{s,t}(x)\Vert^4_F\right)^{1/4}
$$
More generally, we readily check the following result.
\begin{prop}\label{def-4th-prop}
When condition $(\Ta)_n$ is met we have the following time-uniform bounds,
\begin{equation}\label{def-4th}
\EE\left(\Vert \nabla X_{s,t}(x)\Vert^{n}_F\right)^{1/n}\leq \sqrt{d}~e^{-\left[\lambda_A-(n-2)\rho(\nabla\sigma)^2/2\right](t-s)}
\end{equation} 
\end{prop}
  \subsection{The Hessian process}\label{hessian-sec}

In terms of the tensor product  (\ref{tensor-notation}), we have the matrix diffusion equation
$$
 \begin{array}{l}
d \, \nabla^2 X_{s,t}(x)\\
\\
=\left[\left[\nabla X_{s,t}(x)\otimes \nabla X_{s,t}(x)\right]\nabla^2b_t(X_{s,t}(x))+\nabla^2 X_{s,t}(x)\nabla b_t(X_{s,t}(x))\right]dt+d \Ma_{s,t}(x)
\end{array}$$
with the null matrix initial condition $\nabla^2 X_{s,s}(x)=0$ and the matrix-valued martingale
$$
d \Ma_{s,t}(x)=\sum_{1\leq k\leq r}\left(\left[\nabla X_{s,t}(x)\otimes \nabla X_{s,t}(x)\right]\nabla^2\sigma_{t,k}(X_{s,t}(x))+\nabla^2 X_{s,t}(x)\nabla \sigma_{t,k}(X_{s,t}(x))\right)~dW^k_t
$$
Consider the tensor functions
\begin{equation}\label{tensor-functions-ref}
\upsilon_t:=
\sum_{1\leq k\leq d}(\nabla^2\sigma_{t,k})~(\nabla^2 \sigma_{t,k})^{\prime}\quad \mbox{\rm and}\quad \tau_t:=\nabla^2b_t+
\sum_{1\leq k\leq d}(\nabla^2\sigma_{t,k})~(\nabla\sigma_{t,k})^{\prime}
\end{equation}

After some computations, we check that
$$
 \begin{array}{l}
\displaystyle d \, \left[\nabla^2 X_{s,t}(x)\nabla^2 X_{s,t}(x)^{\prime}\right]\\
\\
\displaystyle=\left\{\left[\nabla^2 X_{s,t}(x)~A_t(X_{s,t}(x))~\nabla^2 X_{s,t}(x)^{\prime}\right]+2\left[\left[\nabla X_{s,t}(x)\otimes \nabla X_{s,t}(x)\right]~\tau_t(X_{s,t}(x))~\nabla^2 X_{s,t}(x)^{\prime}\right]_{\tiny sym}\right.\\
\\
\displaystyle\hskip3cm\left.+\left[\left[\nabla X_{s,t}(x)\otimes \nabla X_{s,t}(x)\right]\upsilon_t(X_{s,t}(x))\left[\nabla X_{s,t}(x)\otimes \nabla X_{s,t}(x)\right]^{\prime}\right]\right\}dt+d \Na_{s,t}(x)
\end{array}$$
with the matrix function $A_t(x)$ defined in (\ref{ref-mat-Upsilon}) and the tensor-valued martingale 
$$
 \begin{array}{l}
\displaystyle d \Na_{s,t}(x:)=2~\sum_{1\leq k\leq r}
\left\{\left[\nabla X_{s,t}(x)\otimes \nabla X_{s,t}(x)\right]~\nabla^2\sigma_{t,k}(X_{s,t}(x))~\nabla^2 X_{s,t}(x)^{\prime}\right.\\
\\
\displaystyle\hskip6cm\left.+\nabla^2 X_{s,t}(x)~\nabla \sigma_{t,k}(X_{s,t}(x))~\nabla^2 X_{s,t}(x)^{\prime}\right\}_{\tiny sym}~dW^k_t
\end{array}
$$
When $\nabla\sigma=0$ the above equation reduces to
$$
 \begin{array}{l}
\displaystyle \partial_t\, \left[\nabla^2 X_{s,t}(x)\nabla^2 X_{s,t}(x)^{\prime}\right]\\
\\
\displaystyle=\left[\nabla^2 X_{s,t}(x)~A_t(X_{s,t}(x))~\nabla^2 X_{s,t}(x)^{\prime}\right]+2\left[\left[\nabla X_{s,t}(x)\otimes \nabla X_{s,t}(x)\right]~\nabla^2b_t(X_{s,t}(x))~\nabla^2 X_{s,t}(x)^{\prime}\right]_{\tiny sym}
\end{array}$$
Whenever $(\Ta)_2$ is met, taking the trace in the above display we check that
$$
 \partial_t\, \Vert\nabla^2 X_{s,t}(x)\Vert_F^2\leq -2\lambda_A~\Vert\nabla^2 X_{s,t}(x)\Vert_F^2+2\Vert\nabla^2b\Vert_{F}~\Vert\nabla X_{s,t}(x)\Vert_F^2~\Vert\nabla^2 X_{s,t}(x)\Vert_F
$$
This yields the estimate
$$
 \partial_t\, \Vert\nabla^2 X_{s,t}(x)\Vert_F\leq -\lambda_A~\Vert\nabla^2 X_{s,t}(x)\Vert_F+\Vert\nabla^2b\Vert_{F}~\Vert\nabla X_{s,t}(x)\Vert_F^2
$$
Using (\ref{ref-nablax-estimate-0-ae}) this implies that
$$
\Vert\nabla^2 X_{s,t}(x)\Vert_F\leq \Vert\nabla^2b\Vert_{F}~e^{-\lambda_A(t-s)}~\int_s^t~e^{\lambda_A(u-s)}~\Vert\nabla X_{s,u}(x)\Vert_F^2~du\leq 
\frac{d}{\lambda_A}~\Vert\nabla^2b\Vert_{F}~e^{-\lambda_A(t-s)}
$$
 This ends the proof of  the almost sure estimate (\ref{ref-nabla2x-estimate-0-ae}).

For more general models,  we have that
$$
 \begin{array}{l}
\displaystyle d \, \Vert \nabla^2 X_{s,t}(x)\Vert^2_F\\
\\
\displaystyle=\left\{\tr\left[\nabla^2 X_{s,t}(x)~A_t(X_{s,t}(x))~\nabla^2 X_{s,t}(x)^{\prime}\right]+2~\tr\left[\left[\nabla X_{s,t}(x)\otimes \nabla X_{s,t}(x)\right]~\tau_t(X_{s,t}(x))~\nabla^2 X_{s,t}(x)^{\prime}\right]\right.\\
\\
\displaystyle\hskip3cm\left.+\tr\left[\left[\nabla X_{s,t}(x)\otimes \nabla X_{s,t}(x)\right]\upsilon_t(X_{s,t}(x))\left[\nabla X_{s,t}(x)\otimes \nabla X_{s,t}(x)\right]^{\prime}\right]\right\}dt+d M_{s,t}(x)
\end{array}$$
with a continuous martingale $M_{s,t}(x)$ with angle bracket
$$
 \begin{array}{l}
\displaystyle \partial_t\langle M_{s,\point}(x)\rangle_t\\
\\
\displaystyle=4~\sum_{1\leq k\leq r}\tr
\left\{\left[\nabla X_{s,t}(x)\otimes \nabla X_{s,t}(x)\right]~\nabla^2\sigma_{t,k}(X_{s,t}(x))~\nabla^2 X_{s,t}(x)^{\prime}\right.\\
\displaystyle\hskip7cm\left.+\nabla^2 X_{s,t}(x)~\nabla \sigma_{t,k}(X_{s,t}(x))~\nabla^2 X_{s,t}(x)^{\prime}\right\}^2
\end{array}
$$
\begin{prop}\label{prop-nabla-2-estimate}
Assume  $(\Ta)_n$ is met. In this situation, for any  $\epsilon>0$ s.t. $\lambda_A(n)>\epsilon$ we have
\begin{equation}\label{eq-prop-nabla-2-estimate}
\EE\left(\Vert \nabla^2 X_{s,t}(x)\Vert^{n}_F\right)^{1/n}\leq  n~\epsilon^{-1}~\rchi(b,\sigma)~\exp{\left(-\left[\lambda_A(n)-\epsilon\right](t-s)\right)}
\end{equation}
with the parameters $\rchi(b,\sigma)$ and $\lambda_A(n)$ defined in (\ref{def-chi-b}) and (\ref{ref-lambda-A-n}).
\end{prop}
In the above display, $\rho_{\star}(\nabla \sigma)$ is defined in (\ref{def-nabla-sigma}).
The proof of the above estimate is technical and thus housed in the appendix on page~\pageref{prop-nabla-2-estimate-proof}

\subsection{Bismut-Elworthy-Li formulae}
We further assume that ellipticity condition (\ref{elip}) is met.
In this situation, we can extend gradient semigroup formulae to measurable functions using the Bismut-Elworthy-Li formula
    \begin{equation}\label{bismut-omega}
\nabla P_{s,t}(f)(x)=
    \EE\left(f(X_{s,t}(x))~  \tau^{\omega}_{s,t}(x)\right)    \end{equation}
    with the stochastic process
    $$
     \tau^{\omega}_{s,t}(x):=\int_s^t~   \partial_u \omega_{s,t}(u)~
\nabla X_{s,u}(x)~    a_u(X_{s,u}(x))^{-1/2 }~dW_u
    $$
The above formula is valid for any function $\omega_{s,t}:u\in [s,t]\mapsto \omega_{s,t}(u)\in \RR $ of the following form 
    \begin{equation}\label{bismut-omega-varphi}
\omega_{s,t}(u)=\varphi\left((u-s)/(t-s)\right)~\Longrightarrow   \partial_u \omega_{s,t}(u)=\frac{1}{t-s}~\partial\varphi\left((u-s)/(t-s)\right)~
      \end{equation}
   for some non decreasing differentiable function $\varphi$ on $[0,1]$ with bounded continuous derivatives and such that
  $$
  (\varphi(0),\varphi(1))=(0,1)\Longrightarrow \omega_{s,t}(t)-\omega_{s,t}(s)=1
  $$
Whenever $(\Ta)_2$ is met, combining (\ref{ref-nablax-estimate-0}) with (\ref{bismut-omega}), for any $f$ s.t. $\Vert f\Vert\leq 1$   we check that
  \begin{eqnarray*}
 \Vert \nabla P_{s,t}(f)\Vert^2&\leq& 
 \EE\left( \Vert\tau^{\omega}_{s,t}(x)\Vert^2 \right)\\
 &\leq& \kappa_1~ \int_s^t~e^{-2\lambda_A (u-s)}~
 \Vert \partial_u \omega^{s,t}(u)\Vert^2~du=~\frac{ \kappa_1}{t-s}~ \int_0^1~e^{-2\lambda_A (t-s)v}~
 \left( \partial\varphi(v)\right)^2~dv
\end{eqnarray*}
Let $\varphi_{\epsilon}$ with $\epsilon\in ]0,1[$ be some differentiable function on $[0,1]$ null on $[0,1-\epsilon]$ and such that $\vert \partial\varphi_{\epsilon}(u)\vert\leq c/\epsilon$
and
$(\varphi_{\epsilon}(1-\epsilon),\varphi(1))=(0,1)$. For instance we can choose
$$
\varphi_{\epsilon}(u)=\left\{
\begin{array}{ccl}
0&\mbox{\rm if}&u\in [0,1-\epsilon]\\
\displaystyle1+\cos{\left(\left(1+\frac{1-u}{\epsilon}\right)\frac{\pi}{2}\right)}&\mbox{\rm if}&u\in [1-\epsilon,1]
\end{array}
\right.
$$
 In this situation, we check that
$$
 \Vert \nabla P_{s,t}(f)\Vert^2\leq ~\frac{\kappa_2}{\epsilon^2}~\frac{1}{t-s}~ \int_{1-\epsilon}^1~e^{-2\lambda_A (t-s)v}~dv
$$
from which we find the rather crude uniform estimate
\begin{equation}\label{bismut-est}
 \Vert \nabla P_{s,t}(f)\Vert\leq \frac{ \kappa}{\epsilon}~\frac{1}{\sqrt{t-s}}~e^{-\lambda_A(1-\epsilon) (t-s)}
\end{equation}

 In the same vein, for any $s\leq u\leq t$ we have the formulae
      \begin{eqnarray}
\nabla^2 P_{s,t}(f)(x)&=&   \EE\left(f(X_{s,t}(x))~\tau^{[2],\omega}_{s,t}(x)+\nabla X_{s,t}(x)\,\nabla f(X_{s,t}(x))~  \tau^{\omega}_{s,t}(x)^{\prime}\right)\label{bismut-omega-2-grad}\\
&=&    \EE\left(f(X_{s,t}(x))~  \left[\tau^{[2],\omega}_{s,u}(x)+\nabla X_{s,u}(x)~\tau^{\omega}_{u,t}(X_{s,u}(x))\,\tau^{\omega}_{s,u}(x)^{\prime}\right]\right) \label{bismut-omega-2}   \end{eqnarray}
    with the process
    $$
         \begin{array}{l}
     \tau^{[2],\omega}_{s,t}(x)\\
     \\
  \displaystyle   :=\int_s^t~   \partial_u \omega_{s,t}(u)~\left[\nabla^2 X_{s,u}(x)~ a_u(X_{s,u}(x))^{-1/2 }+
     \left(\nabla X_{s,u}(x)\otimes \nabla X_{s,u}(x)\right)~ (\overline{\nabla} a_u^{-1/2})(X_{s,u}(x))\right]~dW_u
   \end{array}$$
 In the above display $\overline{\nabla} a^{-1/2}_u$ stands for the tensor function
  $$
(\overline{\nabla} a^{-1/2}_u(x))_{(i,j),k}:= \partial_{x_i}a^{-1/2}_u(x)_{j,k}=-\left(a_u^{-1/2}(x)\left[ \partial_{x_i}a_u^{1/2}(x)\right]a_u^{-1/2}(x)\right)_{j,k}
  $$ 
  A detailed proof of the formulae (\ref{bismut-omega-2-grad}) and (\ref{bismut-omega-2})
  in the context of nonlinear diffusion flows can be found in the appendix in~\cite{mp-var-19}.
  
  Observe that
  $$
  (\ref{elip})\Longrightarrow
 \sup_i \Vert  \partial_{x_i}a^{-1/2}_u(x)\Vert\leq c~\Vert \nabla\sigma\Vert/ \upsilon 
  $$
  Whenever $(\Ta)_2$ is met, 
using the estimate (\ref{prop-nabla-2-estimate}) for any $\epsilon\in ]0,1[$ 
\begin{equation}\label{bismut-est-P2-grad}
 \Vert \nabla^2 P_{s,t}(f)\Vert\leq\frac{\kappa}{\epsilon}~\frac{1}{\sqrt{t-s}}~e^{-\lambda_A (t-s) (1-\epsilon)}~\left(\Vert f\Vert+\Vert\nabla f\Vert\right)~\end{equation}

In the same vein, using (\ref{bismut-omega-2}) for any $u\in ]s,t[$ and any bounded measurable function $f$ s.t.
$\Vert f\Vert\leq 1$ we also check the rather crude uniform estimate
$$
 \Vert \nabla^2 P_{s,t}(f)\Vert\leq \frac{\kappa_1}{\epsilon}~\frac{1}{\sqrt{u-s}}~e^{-\lambda_A (u-s) (1-\epsilon)}+\frac{\kappa_2}{\epsilon^2}~\frac{1}{\sqrt{(t-u)(u-s)}}~e^{-\lambda_A (u-s) }~e^{-\lambda_A (t-s) (1-\epsilon)}
$$
Choosing $u=s+(1-\epsilon)(t-s)$ in the above display we check that  for any $\epsilon\in ]0,1[$  we obtain the uniform estimate
\begin{equation}\label{bismut-est-P2}
 \Vert \nabla^2 P_{s,t}(f)\Vert\leq \frac{c_1}{\epsilon\sqrt{1-\epsilon}}~\frac{1}{\sqrt{t-s}}~e^{-\lambda_A (1-\epsilon)^2(t-s)}+\frac{c_2}{\epsilon^2}\frac{1}{\sqrt{\epsilon(1-\epsilon)}}~\frac{1}{t-s}~ e^{-2\lambda_A (1-\epsilon)(t-s)}
\end{equation}
The extended versions of the above formulae  in the context of diffusions on differentiable 
  manifolds can be found in the series of articles~\cite{aht-03,bismut,Elworthy,xm-li,thompson}.
 \section{Backward semigroup analysis}\label{proof-theo-al-gr}
 \subsection{The two-sided stochastic integration}\label{two-sided}

For any given time horizon $s\leq t$ we have the rather well known backward stochastic flow equation
\begin{equation}\label{ref-backward-flow}
X_{s,t}(x)=x+\int_s^t\left[\nabla X_{u,t}(x)^{\prime}~b_u(x)+\frac{1}{2}~~ \nabla ^2X_{u,t}(x)^{\prime}~a_u(x)~\right]~du+\int_s^t\nabla X_{u,t}(x)^{\prime}\sigma_u(x)~dW_u
\end{equation}
The right hand side integral is understood as a conventional backward It\^o-integral.
In a more synthetic form, the above backward formula reduces to (\ref{backward-synthetic}).

An elementary proof of the above formula based on Taylor expansions is presented in~\cite{daprato-3}, different approaches can also be found in~\cite{kunita-2} and
~\cite{krylov}. Extensions of the backward It\^o formula (\ref{ref-backward-flow}) to jump type diffusion models as well as  nonlinear diffusion flows can also be found in~\cite{daprato-2} and in the appendix of~\cite{mp-var-18}.

Consider the discrete time interval $[s,t]_h:=\{u_0,\ldots,u_{n-1}\}$ associated with some refining time mesh $u_{i+1}=u_i+h$ from $u_0=s$ to $u_{n}=t$, for some time step $h>0$. In this notation, combining (\ref{Alekseev-grobner-intro-0}) with  (\ref{Alekseev-grobner-intro-ref-2})  for any $u\in [s,t]_h$ we have the Taylor type approximation
$$
\begin{array}{l}
X_{u+h,t}\circ \overline{X}_{s,u+h}-X_{u,t}\circ \overline{X}_{s,u}\\
\\
\displaystyle\simeq -\left(\left(\nabla X_{u+h,t}\right)(\overline{X}_{s,u}(x))^{\prime}~\Delta b_u(\overline{X}_{s,u}(x))
+\frac{1}{2}\,\left(\nabla ^2X_{u+h,t}\right)(\overline{X}_{s,u}(x))^{\prime}~\Delta a_u(\overline{X}_{s,u}(x))~\right)~h\\
\\
\displaystyle\hskip3cm- \left(\nabla X_{u+h,t}\right)(\overline{X}_{s,u}(x))^{\prime}~\Delta\sigma_u(\overline{X}_{s,u}(x))~(W_{u+h}-W_u)
\end{array}
$$
This yields the  interpolating forward-backward telescoping sum formula
\begin{equation}\label{ref-telescoping-sum}
\begin{array}{l}
X_{s,t}(x)-\overline{X}_{s,t}(x)\\
\\
\displaystyle=-\sum_{u\in [s,t]_h}\left[X_{u+h,t}(\overline{X}_{s,u+h}(x))-X_{u,t}\left(\overline{X}_{s,u}(x)\right)\right]\\
\\
\displaystyle\simeq\sum_{u\in [s,t]_h}\left(\left(\nabla X_{u+h,t}\right)(\overline{X}_{s,u}(x))^{\prime}~\Delta b_u(\overline{X}_{s,u}(x))
+\frac{1}{2}\,\left(\nabla ^2X_{u+h,t}\right)(\overline{X}_{s,u}(x))^{\prime}~\Delta a_u(\overline{X}_{s,u}(x))~\right)~h\\
\\
\hskip3cm\displaystyle+\sum_{u\in [s,t]_h}\left(\nabla X_{u+h,t}\right)(\overline{X}_{s,u}(x))^{\prime}~\Delta\sigma_u(\overline{X}_{s,u}(x))~(W_{u+h}-W_u)
\end{array}
\end{equation}

We obtain formally (\ref{Alekseev-grobner}) by summing the above terms and passing to the limit $h\downarrow 0$.

To be more precise, we follow the  two-sided stochastic integration calculus introduced by Pardoux and Protter in~\cite{pardoux-protter}. As mentioned by the authors this methodology can be seen as a variation of It\^o original construction of the stochastic integral.
In this framework, 
 the Skorohod stochastic integral (\ref{def-S-sk}) arising in (\ref{ref-interp-du}) is defined by the $\LL_2$-convergence 
  \begin{equation}\label{sk-integral}
  \begin{array}[b]{l}
S_{s,t}(\varsigma)(x)\\
\\
\displaystyle:=\lim_{h\rightarrow 0}\sum_{u\in [s,t]_h} \left(\nabla X_{u+h,t}\right)(\overline{X}_{s,u}(x))^{\prime}~
\varsigma_{u}(\overline{X}_{s,u}(x))~(W_{u+h}-W_{u})
\end{array}\quad \mbox{\rm with}\quad \varsigma_u=\Delta \sigma_{u}
 \end{equation}
The proof of the above assertion  is based on a slight extension of proposition 3.3 in~\cite{pardoux-protter} to Skorohod integrals  of the form (\ref{def-S-sk}).
For the convenience of the reader, a detailed proof of the above assertion for one dimensional models is provided in section~\ref{sec-extended-2-sided}.
 
Using (\ref{sk-integral}), the complete proof of (\ref{ref-interp-du}) now follows the same line of arguments as the ones used in the proof of It\^o-type change rule formula stated in theorem 6.1 in~\cite{pardoux-protter}, thus it is skipped.

 \subsection{A multivariate stochastic interpolation formulae}\label{ref-rig}
   
 In terms of the tensor product  (\ref{tensor-notation}),  for any $p\geq 1$ and any twice differentiable function $f$ from $\RR^d$ into $\RR^p$ with at most polynomial growth 
 the function $  F_{s,t}:=  \PP_{s,t}(f)$ satisfies the backward formula (\ref{backward-random-fields})
  with the random fields
  $$
 G _{u,t}(x):=\nabla  F_{u,t}(x)^{\prime}~b_u(x) +\frac{1}{2}~\nabla^2  F_{u,t}(x)^{\prime}~a_{u}(x)\quad
 \mbox{\rm and}\quad  H _{u,t}(x):=\nabla F_{u,t}(x)^{\prime}~\sigma_u(x)~
  $$
  Using the quantitative estimates presented in section~\ref{q-sec}, we checked that the regularity conditions $(H_1)$, $(H_2)$ and $(H_3)$ stated in section~\ref{sec-1-biw-intro} are satisfied.
Rewritten in terms of the stochastic semigroups $  \PP_{s,t}$ and $\overline{\PP}_{s,t}$  we obtain the forward-backward multivariate  interpolation formula
    \begin{equation}\label{Alekseev-grobner-sg-ae}
  \PP_{s,t}(f)(x)-  \overline{\PP}_{s,t}(f)(x)=  \TT_{s,t}(f,\Delta a,\Delta b)(x)+\SS_{s,t}(f,\Delta \sigma)(x)
  \end{equation}
with the stochastic integro-differential operator
    \begin{equation}\label{Alekseev-grobner-sg-ae)int}
     \begin{array}{l}
\displaystyle  \TT_{s,t}(f,\Delta a,\Delta b)(x)\\
\\
\displaystyle:=\int_s^t~\left[\nabla  \PP_{u,t}(f)(\overline{X}_{s,u}(x))^{\prime}~\Delta b_u(\overline{X}_{s,u}(x)) +\frac{1}{2}~\nabla^2  \PP_{u,t}(f)(\overline{X}_{s,u}(x))^{\prime}~\Delta a_{u}(\overline{X}_{s,u}(x))\right]~du
   \end{array}
  \end{equation}
and the two-sided stochastic integral term given by
    \begin{equation}\label{Alekseev-grobner-sg-ae-f}
 \SS_{s,t}(f,\Delta \sigma)(x):=\int_s^t~ \nabla \PP_{u,t}(f)(\overline{X}_{s,u}(x))^{\prime}~\Delta\sigma_u(\overline{X}_{s,u}(x))~dW_u
  \end{equation}

Using elementary differential calculus, for twice differentiable (column vector-valued) function $f$ from $\RR^d$ into $\RR^p$ we readily check the gradient and the Hessian formulae
\begin{eqnarray}
\nabla \,\PP_{s,t}(f)(x)&=&\nabla X_{s,t}(x)~\PP_{s,t}(\nabla f)(x)\nonumber\\
\nabla^2 \PP_{s,t}(f)(x)
&=&\left[\nabla X_{s,t}(x)\otimes \nabla X_{s,t}(x)\right]~\PP_{s,t}(\nabla^2 f)(x)+\nabla^2 X_{s,t}(x)~\PP_{s,t}(\nabla f)(x)\label{grad-sg}
\end{eqnarray}

This shows that $\TT_{s,t}(f,\Delta a,\Delta b)$ and $ \SS_{s,t}(f,\Delta \sigma)$ have the same form as the integrals $T_{s,t}(\Delta a,\Delta b)$ and $S_{s,t}(\Delta a,\Delta b)$ defined in (\ref{Alekseev-grobner}) and (\ref{def-T-st}) up to some terms involving the gradient and the Hessian of the function $f$. For instance, we have the two-sided stochastic integral formula
$$
 \SS_{s,t}(f,\Delta \sigma)(x)=\int_s^t~ 
 \PP_{u,t}(\nabla f)(\overline{X}_{s,u}(x))^{\prime}~
 \nabla X_{u,t}(\overline{X}_{s,u}(x))^{\prime}~\Delta\sigma_u(\overline{X}_{s,u}(x))~dW_u
$$
Also observe that (\ref{Alekseev-grobner-sg-ae}) coincides with (\ref{Alekseev-grobner}) for the identity function; that is, we have that
$$
f(x)=x\Longrightarrow  \TT_{s,t}(f,\Delta a,\Delta b)=T_{s,t}(\Delta a,\Delta b)\quad \mbox{\rm and}\quad  \SS_{s,t}(f,\Delta \sigma)=S_{s,t}(\Delta \sigma)
$$

The above discussion shows that the analysis of the differences of the stochastic semigroups $(\PP_{s,t}-  \overline{\PP}_{s,t})$ in terms of the tangent and the Hessian processes is essentially the same as the one of 
the difference of the stochastic flows $(X_{s,t}-\overline{X}_{s,t})$. For instance using the discussion provided  section~\ref{sec-sk-f}, when  the gradient and the Hessian of the function $f$ are uniformly bounded the estimates stated in theorem~\ref{theo-al-gr-2} can be easily extended at the level of the stochastic semigroups.

The $\LL_2$-norm of the  two-sided stochastic integrals in (\ref{Alekseev-grobner}) and (\ref{Alekseev-grobner-sg-ae}) are  uniformly estimated
as soon as the pair of drift and diffusion functions $(b_t,\sigma_t)$ and $(\overline{b},\overline{\sigma}_t)$ satisfy condition $(\Ta)_2$.
For a more thorough discussion we refer to section~\ref{var-skorohod}, see for instance the $\LL_n$-norm estimates presented in theorem~\ref{theo-quantitative-sko} applied to the difference function $\varsigma_t=\Delta\sigma_t$.

\subsection{Semigroup perturbation  formulae}\label{sg-sect}
 Besides the fact that the Skorohod integral in the r.h.s. of (\ref{Alekseev-grobner-sg-ae}) is not a martingale (w.r.t. the Brownian motion filtration) it is centered (see for instance (\ref{isometry}) and the argument provided in the beginning of section~\ref{var-skorohod}). Thus, taking the expectation in the univariate version of (\ref{Alekseev-grobner-sg-ae}) we obtain the following interpolation semigroup decomposition.

\begin{cor}\label{weak-theo-al-gr}
For any twice differentiable function $f$ from $\RR^d$ into $\RR$ with bounded derivatives we have
the forward-backward semigroup interpolation formula
  \begin{equation}\label{Alekseev-grobner-sg}
   \begin{array}{l}
\displaystyle 
P_{s,t}(f)(x)-\overline{P}_{s,t}(f)(x) 
 =\int_s^t~\EE\left(
\langle \nabla P_{u,t}(f)(\overline{X}_{s,u}(x)),\Delta b_u(\overline{X}_{s,u}(x))\rangle
\right)~du\\
\\
\displaystyle\hskip5cm+\frac{1}{2}~\int_s^t~\EE\left(\tr\left[\nabla^2  P_{u,t}(f) (\overline{X}_{s,u}(x))~\Delta a_{u}(\overline{X}_{s,u}(x))\right]\right)~du
   \end{array}
  \end{equation}
  In addition, under some appropriate regularity conditions  for any differentiable function $f$ such that $\Vert f\Vert\leq 1$ and $\Vert\nabla f\Vert\leq 1$ we have
the uniform estimate  
  \begin{equation}\label{ref-P-2}
\vert P_{s,t}(f)(x)-\overline{P}_{s,t}(f)(x) \vert \leq  \kappa~\left[\vertiii{\Delta a(x)}_{1}+\vertiii{\Delta b(x)}_{1}\right] 
     \end{equation}
 \end{cor}

Rewritten in terms of the infinitesimal generators $(L_t,\overline{L}_t)$ of the stochastic flows $(X_{s,t},\overline{X}_{s,t})$ we recover the rather well known semigroup perturbation  formula
$$  
P_{s,t}=\overline{P}_{s,t}+\int_s^t~\overline{P}_{s,u} (L_u-\overline{L}_u)P_{u,t} ~du
\quad \Longleftrightarrow\quad (\ref{Alekseev-grobner-sg})
   $$
The above formula can be readily checked using the interpolating formula given for any $s\leq u<t$ by the evolution equation
$$
\partial_u (\overline{P}_{s,u}P_{u,t})=(\partial_u \overline{P}_{s,u})P_{u,t}+\overline{P}_{s,u}(\partial_uP_{u,t})= \overline{P}_{s,u}\overline{L}_uP_{u,t}-
\overline{P}_{s,u}L_uP_{u,t}$$ 

Now we come to the proof of  (\ref{ref-P-2}).
Whenever $(\Ta)_2$ is met, combining (\ref{bismut-est}) with  (\ref{bismut-est-P2-grad}) for any differentiable function 
$f$ s.t. $\Vert f\Vert\leq 1$ and $\Vert\nabla f\Vert\leq 1$ and for any $\epsilon\in ]0,1[$  we check that
$$
\vert P_{s,s+t}(f)(x)-\overline{P}_{s,s+t}(f)(x) \vert  \leq  \frac{ \kappa}{\epsilon} ~\left[\vertiii{\Delta a(x)}_{1}+\vertiii{\Delta b(x)}_{1}\right]~\int_0^t~
~\frac{1}{\sqrt{u}}~e^{-\lambda_A(1-\epsilon) u}~du
   $$
   This ends the proof of (\ref{ref-P-2}).\cqfd

After some elementary manipulations the  forward-backward interpolation formula (\ref{Alekseev-grobner-sg})  yields the following corollary.

\begin{cor}
Let $X_t$ and $\overline{X}_t$ be some ergodic diffusions associated with some time homogeneous drift and diffusion functions
$(b,\sigma)$ and $(\overline{b},\overline{\sigma})$. 
The invariant probability measures $\pi$ and $\overline{\pi}$ of  $X_t$ and $\overline{X}_t$ 
are connected for any twice differentiable function $f$ from $\RR^d$ into $\RR$ with bounded derivatives by the following interpolation formula

  \begin{equation}\label{diff-pi}
(\pi-\overline{\pi})(f) =\int_0^{\infty}~\EE\left(
\left\langle \nabla P_{t}(f)(\overline{Y}),\Delta b(\overline{Y})\right\rangle
+\frac{1}{2}~\tr\left[\nabla^2  P_{t}(f) (\overline{Y})~\Delta a(\overline{Y})\right]\right)~dt
  \end{equation}
  In the above display $\overline{Y}$ stands for a random variable with distribution $\overline{\pi}$ and $P_t$ stands for the Markov transition semigroup of
  the process $X_t$.
\end{cor}

The formula (\ref{diff-pi}) can be used to estimate the invariant measure of a stochastic flow associated with some perturbations of the drift and the diffusion function.

For instance, for homogeneous Langevin diffusions $X_{t}$ associated with some convex potential function $U$
we have
$$
   b=-\nabla U\quad \mbox{\rm and}\quad\sigma=I\quad \Longrightarrow\quad   \pi(dx)~\propto~\exp{\left(-2\,U(x)\right)}~dx
$$
 In the above display, $dx$ stands for the Lebesgue measure on $\RR^d$. In this situation, using (\ref{diff-pi}), for any ergodic 
 diffusion flow $\overline{X}_{t}$ with some drift $\overline{b}$ and an unit diffusion matrix we have
 $$
 \overline{\pi}(f)=\pi(f)+\int_0^{\infty}~\EE\left(
\left\langle (\overline{b}+\nabla U)(\overline{Y}),\nabla P_{t}(f)(\overline{Y})\right\rangle\right)~dt
  $$
Notice that the above formula is implicit as the r.h.s. term depends on $ \overline{\pi}$. By symmetry arguments, we also have 
the following more explicit perturbation formula 
  $$
   \overline{\pi}(f)=\pi(f)+\int_0^{\infty}~\EE\left(
\left\langle (\overline{b}+\nabla U)(Y),\nabla \overline{P}_{t}(f)(Y)\right\rangle\right)~dt
  $$
   In the above display ${Y}$ stands for a random variable with distribution ${\pi}$ and $\overline{P}_t$ stands for the Markov transition semigroup of
  the process $\overline{X}_t$.

  \subsection{Some extensions}\label{sec-extensions}
  
  Several extensions of the forward-backward stochastic interpolation formula (\ref{Alekseev-grobner}) to more general stochastic perturbation processes can be developed. For instance, suppose we are given some stochastic processes $\overline{Y}_{s,t}(x)\in\RR^d$ and $\overline{Z}_{s,t}(x)\in \RR^{d\times r}$ adapted to the filtration of the Brownian motion $W_t$, and let 
 $\overline{X}_{s,t}(x)$ be the stochastic flow defined by the stochastic differential equation
   \begin{equation}\label{X-Y-Z}
 d\overline{X}_{s,t}(x)=\overline{Y}_{s,t}(x)~dt+\overline{Z}_{s,t}(x)~dW_t
 \end{equation}
In this situation, the interpolation formula (\ref{ref-interp-du}) remains valid when  $\overline{a}_u(\overline{X}_{s,u}(x))$  is replaced by the stochastic matrices
$\overline{Z}_{s,t}(x)\overline{Z}_{s,t}(x)^{\prime}$. This yields without further work the forward-backward stochastic interpolation formula (\ref{Alekseev-grobner}) with the local perturbations
\begin{eqnarray*}
\Delta b_u (\overline{X}_{s,u}(x))&:=&b_u (\overline{X}_{s,u}(x))-\overline{Y}_{s,u}(x)\\
\Delta \sigma_u (\overline{X}_{s,u}(x))&:=&\sigma_u (\overline{X}_{s,u}(x))-\overline{Z}_{s,u}(x)
\quad \mbox{\rm and}\quad
\Delta a_u (\overline{X}_{s,u}(x)):=a_u (\overline{X}_{s,u}(x))-\overline{Z}_{s,u}(x)\overline{Z}_{s,u}(x)^{\prime}
\end{eqnarray*}
The corresponding  interpolation formula should be used with some caution as the $\LL_2$-norm of the  two-sided stochastic integral (\ref{def-S-sk}) depends on the Malliavin differential of the integrand process of the Brownian motion; see for instance the variance formula provided in lemma~\ref{lem-var}.

Assume that  $\sigma=I$ and the regularity condition $(\Ta)_2$ is met.
 Also  suppose 
 $\overline{X}_{s,t}(x)$ is given by a stochastic differential equation of the form (\ref{X-Y-Z}) with $r=d$ and $\overline{Z}_{s,t}(x)=I$.
 Arguing as above, in terms of the tensor product  (\ref{tensor-notation}) we have
   \begin{equation}\label{X-Y-ref-ag}
  X_{s,t}(x)-\overline{X}_{s,t}(x)=\int_s^t~\left(\nabla X_{u,t}\right)(\overline{X}_{s,u}(x))^{\prime}~(b_u(\overline{X}_{s,u}(x))-\overline{Y}_{s,u}(x))~du
 \end{equation}
Combining (\ref{ref-nablax-estimate-0-ae}) with the generalized Minkowski inequality,  we check the following proposition.
\begin{prop}
 Assume that $(\Ta)_2$ is met for some $\lambda_A>0$. In this situation,
for any $1\leq n\leq \infty$ we have the estimates
  \begin{equation}\label{Alekseev-grobner-2-n-XY}
\EE\left[\Vert X_{s,t}(x)-\overline{X}_{s,t}(x)\Vert^{n}\right]^{1/n}\leq \int_s^t~e^{-\lambda_A(t-u)}~\EE\left[\Vert b_u(\overline{X}_{s,u}(x))-\overline{Y}_{s,u}(x)\Vert\right]^{1/n}~du
 \end{equation}
 \end{prop}
 In the same vein,  we have
  \begin{equation}\label{Alekseev-grobner-2-weak}
 P_{s,t}(f)(x)-\overline{P}_{s,t}(f)(x) 
 =\int_s^t~\EE\left(
\langle \nabla P_{u,t}(f)(\overline{X}_{s,u}(x)),b_u(\overline{X}_{s,u}(x))-\overline{Y}_{s,u}(x)\rangle
\right)~du
 \end{equation}
For instance, for the Langevin diffusion discussed in  (\ref{ref-HA-Langevin})  and (\ref{ref-gibbs}) the weak expansion (\ref{Alekseev-grobner-2-weak})
implies that
\begin{equation}\label{ref-discrete-time}
[\pi\overline{P}_{s,t} -
\pi](f)=\int_s^t~\int\pi(dx)~\EE\left(
\langle \nabla P_{t-u}(f)(\overline{X}_{s,u}(x)),\nabla U(\overline{X}_{s,u}(x))+\overline{Y}_{s,u}(x)\rangle
\right)~du
\end{equation}
This yields the $\WW_1$-Wasserstein estimate
$$
\WW_1(\pi\overline{P}_{s,t},
\pi)\vert\leq \int_s^t~e^{-\lambda_A(t-u)}~\int\pi(dx)~\EE\left(
\Vert \nabla U(\overline{X}_{s,u}(x))+\overline{Y}_{s,u}(x)\Vert
\right)~du
$$
Combining (\ref{bismut-est}) with (\ref{ref-discrete-time}), for any $\epsilon\in ]0,1[$  we also have the total variation norm estimate
\begin{equation}\label{bismut-est-overline}
\Vert \pi\overline{P}_{s,t} -
\pi\Vert_{\tiny tv}\leq \frac{c}{\epsilon}~\int_s^t~\frac{1}{\sqrt{t-u}}~e^{-\lambda_A(1-\epsilon) (t-u)}~\left[\int\pi(dx)~\EE\left(
\Vert \nabla U(\overline{X}_{s,u}(x))+\overline{Y}_{s,u}(x)\Vert
\right)\right]~du
\end{equation}
 \section{Skorohod fluctuation processes}\label{sk-section}

 \subsection{A variance formula}\label{var-skorohod}
 
  Let $\varsigma_t(x)$ be some differentiable $(d\times r)$-matrix valued function on $\RR^d$ such that
  \begin{equation}\label{hyp-varsigma}
\Vert\nabla \varsigma\Vert<\infty\quad \mbox{\rm and}\quad
\Vert \varsigma(0)\Vert:=\sup_t\Vert \varsigma_t(0)\Vert<\infty  
\end{equation}
  
 Recalling that $(W_{u+h}-W_{u})$ is independent of the flows $\overline{X}_{s,u}$ and $\nabla X_{u+h,t}$, the discrete time approximation (\ref{sk-integral}) shows that 
 Skorohod stochastic integral is centered; that is, we have that
 $\EE(S_{s,t}(\varsigma)(x))=0$. 
  
Following (\ref{sk-integral}), the variance can be computed using the following approximation formula
 \begin{equation}\label{sk-integral-var-ref}
    \begin{array}{l}
 \displaystyle
\EE\left[\Vert S_{s,t}(\varsigma)(x)\Vert^2\right]
=\lim_{h\rightarrow 0}~\sum_{u,v\,\in [s,t]_h}~ \sum_{1\leq i\leq d}~\sum_{1\leq j,k\leq r} \\
\\
\hskip3cm\displaystyle\EE\left\{
\left[\left(\nabla X_{u+h,t}\right)(\overline{X}_{s,u}(x))^{\prime}~ \varsigma_u(\overline{X}_{s,u}(x))\right]_{i,j}~
\left[\left(\nabla X_{v+h,t}\right)(\overline{X}_{s,v}(x))^{\prime}~ \varsigma_{v}(\overline{X}_{s,v}(x))\right]_{i,k}\right.\\
\\
\left.\hskip7cm
   (W^j_{u+h}-W^j_{u})
   (W^k_{v+h}-W^k_{v})~\right\}
   \end{array}
 \end{equation}
 The proof of the above assertion is provided in section~\ref{sec-extended-2-sided}, see for instance proposition~\ref{k-prop}.
 
Consider  the matrix valued function
 \begin{equation}\label{defi-Sigma}
 \Sigma_{s,u,t}(x):=\left[\left(\nabla X_{u,t}\right)^{\prime}\circ\overline{X}_{s,u}\right](x)~\varsigma_u(\overline{X}_{s,u}(x))
 \end{equation}
In this notation, the limiting diagonal term $u=v$ in the r.h.s. of (\ref{sk-integral-var-ref}) is clearly equal to
 $$
 \int_s^t~ \EE\left[ \sum_{i,j}\Sigma_{s,u,t}(x)_{i,j}~\Sigma_{s,u,t}(x)_{i,j}\right]~du= \int_s^t~ \EE\left[ \Vert\Sigma_{s,u,t}(x)\Vert^2_{\tiny F}\right]~du
 $$
In addition, whenever condition $(\Ta)_2$ is met and $\varsigma$ is bounded, (\ref{ref-nablax-estimate-0}) readily yields the estimate
 \begin{equation}\label{ref-diagonal-Sigma}
\left[ \int_s^t~ \EE\left[ \Vert\Sigma_{s,u,t}(x)\Vert^2_{\tiny F}\right]~du\right]^{1/2}\leq \Vert\varsigma\Vert_2~\sqrt{d/(2\lambda_A)}
 \end{equation}
 More generally, using  (\ref{def-4th}) whenever $(\overline{\Ma})_{2/\delta}$ and $(\Ta)_{2/(1-\delta)}$ are met for some $\delta\in ]0,1[$ we have the estimate
 $$
\EE\left[ \Vert  \Sigma_{s,u,t}(x)\Vert^2\right]
\leq c_{1,\delta}~
 \left[
 \Vert \varsigma(0)\Vert^2+\Vert\nabla \varsigma\Vert^2~(1+\Vert x\Vert)^2\right]~e^{-2\lambda_A(2/(1-\delta))(t-u)}
 $$
This implies that
 \begin{equation}\label{ref-diagonal-Sigma-ref2}
\left[ \int_s^t~ \EE\left[ \Vert\Sigma_{s,u,t}(x)\Vert^2_{\tiny F}\right]~du\right]^{1/2}\leq~c_{2,\delta}~
 \left[
 \Vert \varsigma(0)\Vert+\Vert\nabla \varsigma\Vert~(1+\Vert x\Vert)\right]/\sqrt{\lambda_A}
 \end{equation}

  The non-diagonal term can be computed in a more direct way using Malliavin derivatives of the functions $\Sigma_{s,u,t}$.
  For any $s\leq u\leq v\leq t$ we have
   \begin{equation}\label{s-tensor-diff-u-v-ref}
  D_v\left\{\left[\left(\nabla X_{u,t}\right)^{\prime}\circ\overline{X}_{s,u}\right]~\left[\varsigma_u\circ\overline{X}_{s,u}\right]\right\}=
 \left[\left(D_v\left(\nabla X_{u,t}\right)^{\prime}\right)\circ\overline{X}_{s,u}\right]~\left[\varsigma_u\circ\overline{X}_{s,u}\right]
 \end{equation}
As expected, observe that
$$
\nabla\sigma=0\quad\Longrightarrow\quad D_v \Sigma_{s,u,t}(x)=0
$$
In the reverse angle, whenever $s\leq v\leq  u\leq t$ we have the chain rule formula

   \begin{equation}\label{s-tensor-diff-u-v-2}
  \begin{array}{l}
D_v\left(\left[\varsigma_u\circ\overline{X}_{s,u}\right]~\left[\left(\nabla X_{u,t}\right)\circ\overline{X}_{s,u}\right]\right)\\
 \\
 :=
 \left[D_{v}\left(\varsigma_u\circ \overline{X}_{s,u}\right)\right]\left[\left(\nabla X_{u,t}\right)\circ\overline{X}_{s,u}\right]+
  \left[D_v \overline{X}_{s,u}\otimes (\varsigma_u\circ \overline{X}_{s,u})\right] \left[\left(\nabla^2 X_{u,t}\right)\circ\overline{X}_{s,u}\right] \end{array}
 \end{equation}
As above, Malliavin differentials $D_{v}\left(\varsigma_u\circ \overline{X}_{s,u}\right)$ and $D_v \overline{X}_{s,u}$ can be computed using the chain rule formulae (\ref{ref-chain-r}).

  A more detailed analysis of the chain rules formulae (\ref{ref-chain-r}), (\ref{s-tensor-diff-u-v})  and (\ref{s-tensor-diff-u-v-2}) for one dimensional models is provided in section~\ref{sec-extended-2-sided} (cf. lemma~\ref{lem-tex}).

  Observe that
$$
\nabla  \varsigma=0\quad\Longrightarrow\quad D_v \left[\Sigma_{s,u,t}^{\,\prime}\right]=  \left[D_v \overline{X}_{s,u}\otimes (\varsigma_u\circ \overline{X}_{s,u})\right] \left[\left(\nabla^2 X_{u,t}\right)\circ\overline{X}_{s,u}\right] 
$$

We consider  the inner product
\begin{eqnarray*}
\left\langle D_u\Sigma_{s,v,t}(x),D_v\Sigma_{s,u,t}(x)\right\rangle
&:=& \sum_{i,j,k} \left( D_v\Sigma_{s,u,t}(x)\right)_{k,i,j}~ 
  \left(D_u\Sigma_{s,v,t}(x)\right)_{j,i,k}\end{eqnarray*}
In this notation, an explicit description of the $\LL_2$-norm of the  two-sided stochastic integral in terms of
 Malliavin derivatives is given below. 
\begin{lem}\label{lem-var}
The $\LL_2$-norm of the Skorohod integral
$S_{s,t}(\varsigma)(x)$ introduced in (\ref{sk-integral}) is given for any $x\in \RR^d$ and $s\leq t$ by the formulae
 $$
  \EE\left[\Vert S_{s,t}(\varsigma)(x)\Vert^2\right]=\int_{[s,t]}~ \EE\left[ \Vert\Sigma_{s,u,t}(x)\Vert^2_{\tiny F}\right]~du
 +\int_{[s,t]^2}~
\EE\left[  \left\langle D_v\Sigma_{s,u,t}(x),
D_u\Sigma_{s,v,t}(x)\right\rangle\right]
~du~dv
 $$
 with the random matrix function $\Sigma_{s,u,t}$ defined in (\ref{defi-Sigma}) and
 the Malliavin
 derivative $D_v\Sigma_{s,u,t}$ given in formulae (\ref{s-tensor-diff-u-v-ref}) and (\ref{s-tensor-diff-u-v-2}). In addition, we have
 $$
 \nabla\sigma=0\quad\Longrightarrow\quad  \EE\left[\Vert S_{s,t}(\varsigma)(x)\Vert^2\right]=\int_{[s,t]}~ \EE\left[ \Vert\Sigma_{s,u,t}(x)\Vert^2_{\tiny F}\right]~du
 $$

\end{lem}
The above lemma can be interpreted as a matrix version of the isometry property (\ref{isometry}).
A  proof of the above lemma based on the $\LL_2$-approximation of two-sided stochastic integrals is provided in section~\ref{sec-extended-2-sided} (see for instance proposition~\ref{k-prop}).

\subsection{Quantitative estimates}\label{q-sec}
For any $p>1$ and any tensor norms we also quote the rather well known $\LL_p$-norm estimates
 $$
   \begin{array}{l}
 \displaystyle  \EE\left[\Vert S_{s,t}(\varsigma)(x)\Vert^{p}\right]^{2/p}\\
 \\
  \displaystyle\leq c_{1,p} \int_{[s,t]}~ \EE\left[ \Vert\Sigma_{s,u,t}(x)\Vert^2\right]~du
 +c_{2,p}~\EE\left[ \left(\int_{[s,t]^2}~
 \Vert D_v\Sigma_{s,u,t}(x)\Vert^2
~du~dv\right)^{p/2}\right]^{2/p}
   \end{array}
 $$
 for some finite constants $c_{i,p}$ whose values only depend on $p$.
A proof of these estimates can be found in~\cite{nualart-pardoux,watanabe}, see also \cite{nualart-z} for multiple Skorohod integrals. By the generalized Minkowski inequality, for any $n\geq 2$ we also have the estimate
  \begin{equation}\label{pre-theo-fluctuation}
   \begin{array}{l}
 \displaystyle  \EE\left[\Vert S_{s,t}(\varsigma)(x)\Vert^{n}\right]^{2/n}\\
 \\
  \displaystyle\leq c_{1,n} \int_{[s,t]}~ \EE\left[ \Vert\Sigma_{s,u,t}(x)\Vert^2\right]~du
 +c_{2,n}~\int_{[s,t]^2}~\EE\left[
 \Vert D_v\Sigma_{s,u,t}(x)\Vert^n\right]^{2/n}
~du~dv
   \end{array}
 \end{equation}
Observe that for any $n\geq 2$ we have
$$
(\overline{\Ma})_n\Longrightarrow  \vertiii{\varsigma(x)}_n\leq \kappa_n~\left(\Vert\varsigma(0)\Vert+\Vert\nabla\varsigma\Vert\right) (1\vee\Vert x\Vert)
$$
   The main objective of this section is to prove the following theorem.
\begin{theo}\label{theo-quantitative-sko}

Assume that  $(M)_{2n/\delta}$ and  $(T)_{2n/(1-\delta)}$ are satisfied for some parameter $n\geq 2$ and some $\delta\in ]0,1[$. In this situation, we have the uniform estimate
  \begin{equation}\label{intro-inq-s}
 \displaystyle  \EE\left[\Vert S_{s,t}(\varsigma)(x)\Vert^{n}\right]^{1/n}
\leq  \kappa_{\delta,n}~\vertiii{\varsigma(x)}_{2n/\delta}~(1\vee \Vert x\Vert)
 \end{equation}
For uniformly bounded  diffusion functions $(\varsigma,\sigma,\overline{\sigma})$
whenever
 $(T)_{2n}$ is met for some $n\geq 2$ we have
  \begin{equation}\label{intro-inq-s-2}
 \displaystyle  \EE\left[\Vert S_{s,t}(\varsigma)(x)\Vert^{n}\right]^{1/n}
\leq \kappa_{n}~\left(\Vert \varsigma\Vert+ \Vert\nabla\varsigma \Vert\right)
 \end{equation}
In addition, for constant diffusion functions $(\varsigma,\sigma,\overline{\sigma})$ whenever  $(T)_{2}$ is met, for any $n\geq 2$ we have the uniform  estimate
  \begin{equation}\label{intro-inq-s-2-2}
\EE\left[\Vert S_{s,t}(\varsigma)(x)\Vert^{n}\right]^{1/n}
\leq \kappa_{n}~\Vert \varsigma\Vert
 \end{equation}
   \end{theo}
The proof of the above theorem, including a more detailed description of the parameters $\kappa_{\delta,n}$ and $\kappa_{n}$ is provided below.

Next, we estimate the $\LL_n$-norm of the Malliavin differential $D_v\Sigma_{s,u,t}(x)$ in the two cases $(s\leq u\leq v\leq t)$ and $(s\leq v\leq u\leq t)$.

\subsubsection*{Case $(s\leq u\leq v\leq t)$:} 
Using  (\ref{s-tensor-diff-u-v-ref}) we have
$$
 \Vert D_v\Sigma_{s,u,t}(x)\Vert\leq c~\Vert \varsigma_u(\overline{X}_{s,u}(x))\Vert~\Vert
(D_v\nabla X_{u,t})(\overline{X}_{s,u}(x))\Vert
 $$
Using (\ref{ref-chain-r}) and (\ref{s-tensor-diff-u-v}) this yields the estimate
$$
 \Vert D_v\Sigma_{s,u,t}(x)\Vert\leq c_1~ \II_{s,u,t}(x)+c_2~ \JJ_{s,u,t}(x)
 $$
with the functions
$$
    \begin{array}{l}
\displaystyle \II_{s,u,t}(x):=  \Vert \nabla \sigma\Vert~\Vert \varsigma_u(\overline{X}_{s,u}(x))\Vert~
\Vert (\nabla X_{u,v})(\overline{X}_{s,u}(x))\Vert~\Vert(\nabla X_{v,t})(Z^{s,v}_u(x))\Vert\\
\\
 \JJ_{s,u,t}(x):=\Vert \sigma_v(Z^{s,v}_u(x))\Vert~  \Vert \varsigma_u(\overline{X}_{s,u}(x))\Vert~\Vert(\nabla X_{u,v})(\overline{X}_{s,u}(x))\Vert~
 \Vert(\nabla^2 X_{v,t})(Z^{s,v}_u(x))\Vert 
    \end{array}
 $$
 In the above display, $Z^{s,v}_u(x)$ stands for the interpolating flow defined in (\ref{interpolating-flow}).
\begin{itemize}
\item Firstly assume that $\Vert \varsigma\Vert\vee \Vert \sigma\Vert<\infty$ and  $(\Ta)_{2n}$ is satisfied for some parameter $n\geq 1$. In this situation, 
applying proposition~\ref{def-4th-prop} and proposition~\ref{prop-nabla-2-estimate},
for any  $\epsilon\in ]0,1[$
we have the uniform estimates
$$
\EE\left( \Vert D_v\Sigma_{s,u,t}(x)\Vert^{n}\right)^{1/n}\leq  \Vert \varsigma\Vert~\rchi_{n,\epsilon}(b,\sigma)~\exp{\left(-(1-\epsilon)\lambda_A(2n)(t-u)\right)}
$$
with the parameter $  \rchi_{n,\epsilon}(b,\sigma)$ given by
$$
   \rchi_{n,\epsilon}(b,\sigma):= c~\left[ \Vert \sigma\Vert\vee \Vert \nabla \sigma\Vert\right]   
  \left[1+\frac{1}{\epsilon}~\frac{n}{\lambda_A(2n)}~\rchi(b,\sigma)\right]~\quad\mbox{\rm with $\rchi(b,\sigma)$ given in (\ref{def-chi-b}).}
$$

\item More generally, when $\Vert \nabla\varsigma\Vert\vee \Vert \nabla\sigma\Vert<\infty$ the functions $\varsigma_t(x)$ and $\sigma_t(x)$ may grow at the most linearly with respect to $\Vert x\Vert$. Assume that conditions $(M)_{2n/\delta}$ and condition $(\Ta)_{2n/(1-\delta)}$ are satisfied for some parameters $n\geq 1$ and $\delta\in ]0,1[$. In this situation, applying H\"older inequality we check that
$$
    \begin{array}{l}
\displaystyle
\EE\left( \Vert  \II_{s,u,t}(x)\Vert^{n}\right)^{1/n}\leq c~ \Vert \nabla \sigma\Vert~
\EE\left( \Vert \varsigma_u(\overline{X}_{s,u}(x))\Vert^{n/\delta}\right)^{\delta/n}\\
\\
\hskip.3cm\displaystyle\times\EE\left( \Vert  (\nabla X_{u,v})(\overline{X}_{s,u}(x))\Vert^{2n/(1-\delta)}\right)^{(1-\delta)/(2n)}
\EE\left( \Vert (\nabla X_{v,t})(Z^{s,v}_u(x))\Vert^{2n/(1-\delta)}\right)^{(1-\delta)/(2n)}
    \end{array}$$
    Applying proposition~\ref{def-4th-prop} we check that
   $$
    \begin{array}{l}
\displaystyle
\EE\left( \Vert  \II_{s,u,t}(x)\Vert^{n}\right)^{1/n}\leq c_{n,\delta}~ \Vert \nabla \sigma\Vert ~\vertiii{\varsigma(x)}_{n/\delta}~e^{-\lambda_A(2n/(1-\delta))(t-u)}
    \end{array}$$ 
    In the same vein, combining proposition~\ref{def-4th-prop} and proposition~\ref{prop-nabla-2-estimate} with the uniform moment estimates (\ref{ref-ui-m-over}) we check that
       $$
    \begin{array}{l}
\displaystyle
\EE\left( \Vert  \JJ_{s,u,t}(x)\Vert^{n}\right)^{1/n}\leq    c_{n,\delta}~\left[\Vert \sigma(0)\Vert+\Vert\nabla \sigma\Vert\right]~\frac{1}{\epsilon}~\frac{\rchi(b,\sigma)}{\lambda_A(2n/(1-\delta))}\\
\\
\hskip3cm\displaystyle\times
~\vertiii{\varsigma(x)}_{2n/\delta}~~\left[ 1+\Vert x\Vert\right]~~e^{-(1-\epsilon)\lambda_A(2n/(1-\delta))(t-u)}
    \end{array}$$ 
    We conclude that
    $$
    \EE\left( \Vert D_v\Sigma_{s,u,t}(x)\Vert^{n}\right)^{1/n}\leq 
         \rchi_{n,\delta,\epsilon}(b,\sigma)
~\vertiii{\varsigma(x)}_{2n/\delta}~\left[ 1+\Vert x\Vert\right]~~e^{-(1-\epsilon)\lambda_A(2n/(1-\delta))(t-u)}
    $$
    with the parameter
    $$
    \rchi_{n,\delta,\epsilon}(b,\sigma):=c_{n,\delta}~\left[\Vert \sigma(0)\Vert+\Vert\nabla \sigma\Vert\right]~\left(1+\frac{1}{\epsilon}~\frac{\rchi(b,\sigma)}{\lambda_A(2n/(1-\delta))}\right)
    $$
    \end{itemize}
   
   \subsubsection*{Case $(s\leq v\leq u\leq t)$:} 
 We use (\ref{s-tensor-diff-u-v-2}) to check that
$$
  \begin{array}{l}
\Vert D_v\Sigma_{s,u,t}(x)\Vert
 \leq 
 \Vert [D_{v}\left(\varsigma_u\circ \overline{X}_{s,u}\right)](x)\Vert~\Vert \left(\nabla X_{u,t}\right)(\overline{X}_{s,u}(x))\Vert\\
 \\
 \hskip3cm+
\Vert [D_v \overline{X}_{s,u}](x)\otimes \varsigma_u(\overline{X}_{s,u})(x)\Vert ~\Vert \left(\nabla^2 X_{u,t}\right)(\overline{X}_{s,u} (x))\Vert
  \end{array}
  $$
  On the other hand, using the chain rules (\ref{ref-chain-r}) we have
    \begin{eqnarray*}
 D_v\overline{X}_{s,u}&:=&\left(D_v \overline{X}_{s,v}\right)~\left[\left(\nabla \overline{X}_{v,u}\right)\circ \overline{X}_{s,v}\right]\nonumber\\
D_{v}\left(\varsigma_{u}\circ \overline{X}_{s,u}\right)&=&(D_v \overline{X}_{s,u})~\left[\left(\nabla \varsigma_{u}\right)\circ \overline{X}_{s,u}\right]
 \end{eqnarray*}
  This yields the estimate
  $$
  \begin{array}{l}
\Vert D_v\Sigma_{s,u,t}(x)\Vert
 \leq 
 c_1~\Vert\overline{\sigma}_v(\overline{X}_{s,v}(x))\Vert~\Vert\nabla \varsigma \Vert~\Vert (\nabla \overline{X}_{v,u})(\overline{X}_{s,v}(x))\Vert~\Vert \left(\nabla X_{u,t}\right)(\overline{X}_{s,u}(x))\Vert\\
 \\
 \hskip.3cm+c_2~
\Vert\overline{\sigma}_v(\overline{X}_{s,v}(x))\Vert~\Vert\varsigma_{u}(\overline{X}_{s,u}(x))  \Vert~\Vert (\nabla \overline{X}_{v,u})(\overline{X}_{s,v}(x))\Vert~\Vert \left(\nabla^2 X_{u,t}\right)(\overline{X}_{s,u} (x))\Vert
  \end{array}
  $$

\begin{itemize}   
\item Firstly assume that $\Vert \varsigma\Vert\vee \Vert \overline{\sigma}\Vert<\infty$ and condition $(T)_{2n}$ is satisfied for some $n\geq 1$.
In this situation, arguing as above for  any $\epsilon\in ]0,1[$ 
we have the uniform estimates
$$
\EE\left( \Vert D_v\Sigma_{s,u,t}(x)\Vert^{n}\right)^{1/n}\leq  \left(\Vert \varsigma\Vert+ \Vert\nabla\varsigma \Vert\right)~\overline{\rchi}_{n,\epsilon}(b,\sigma)~\exp{\left(-(1-\epsilon)\lambda_{A,\overline{A}}(2n)(t-v)\right)}
$$
for some universal constant $c$ and the parameter $   \overline{\rchi}_{n,\epsilon}(b,\sigma)$ given by
$$
   \overline{\rchi}_{n,\epsilon}(b,\sigma):=c~\Vert\overline{\sigma}\Vert~\left[1+\frac{1}{\epsilon}~\frac{n}{\lambda_{A,\overline{A}}(2n)}~\rchi(b,\sigma)\right]~\quad\mbox{\rm with $\rchi(b,\sigma)$ given in (\ref{def-chi-b}).}
$$

\item More generally assume that $\Vert \nabla\varsigma\Vert\vee \Vert \nabla\overline{\sigma}\Vert<\infty$. Also assume that conditions $(M)_{2n/\delta}$ and  $(T)_{2n/(1-\delta)}$ are satisfied for some parameters $n\geq 1$ and $\delta\in ]0,1[$. In this situation, we have
   $$
       \begin{array}{l}
\displaystyle
    \EE\left( \Vert D_v\Sigma_{s,u,t}(x)\Vert^{n}\right)^{1/n}\\
    \\
\displaystyle    \leq 
         \rchi_{n,\delta,\epsilon}(b,\sigma,\overline{\sigma})
~\vertiii{\varsigma(x)}_{2n/\delta}~\left[ 1+\Vert x\Vert\right]~~e^{-(1-\epsilon)\lambda_{A,\overline{A}}(2n/(1-\delta))(t-v)}
    \end{array}    $$
    with the parameter
    $$
    \rchi_{n,\delta,\epsilon}(b,\sigma,\overline{\sigma}):=c_{n,\delta}~\left[\Vert \overline{\sigma}(0)\Vert+\Vert\nabla \overline{\sigma}\Vert\right]~\left(1+\frac{1}{\epsilon}~\frac{\rchi(b,\sigma)}{\lambda_{A,\overline{A}}(2n/(1-\delta))}\right)
    $$
        \end{itemize}   
        The end of the proof of theorem~\ref{theo-quantitative-sko} is  a direct consequence of the estimates discussed above combined with (\ref{pre-theo-fluctuation}) and the diagonal estimates presented in (\ref{ref-diagonal-Sigma}).  \cqfd

\subsection{Some extensions}\label{sec-sk-f}        
This section is concerned with the  two-sided stochastic integral (\ref{Alekseev-grobner-sg-ae-f}). Using the gradient formula in (\ref{grad-sg})
the Skorohod stochastic integral in (\ref{Alekseev-grobner-sg-ae-f}) takes the form
$$
 \SS_{s,t}(f,\Delta \sigma)(x)=\int_s^t~ \Sigma_{s,u,t}(f)(x)~dW_u        
$$
with the   integrands
$$ 
  \Sigma_{s,u,t}(f)(x):=\nabla f(Z^{s,t}_u(x))^{\prime}~ \Sigma_{s,u,t}(x)\quad \mbox{\rm and}\quad
  \Sigma_{s,u,t}(x):=\left[\left(\nabla X_{u,t}\right)^{\prime}\circ\overline{X}_{s,u}\right]~\left[\Delta \sigma_u\circ\overline{X}_{s,u}\right]
 $$   
As in (\ref{s-tensor-diff-u-v}), using  the chain  rule properties of Malliavin derivatives  we check that
$$
D^i_v \Sigma_{s,u,t}(f)=\left(D^i_v \nabla f(Z^{s,t}_u)^{\prime}\right)~ \Sigma_{s,u,t}+\nabla f(Z^{s,t}_u)^{\prime}~D^i_v  \Sigma_{s,u,t}
$$
as well as
$$
D^i_v \nabla f(Z^{s,t}_u)^{\prime}=\nabla^2 f(Z^{s,t}_u)^{\prime}~D^i_v Z^{s,t}_u
$$
This yields the differential formula
$$
D^i_v \Sigma_{s,u,t}(f)=\nabla f(Z^{s,t}_u)^{\prime}~D^i_v  \Sigma_{s,u,t}+\nabla^2 f(Z^{s,t}_u)^{\prime}~(D^i_v Z^{s,t}_u)~ \Sigma_{s,u,t}
$$
The Malliavin derivatives $D^i_v  \Sigma_{s,u,t}$ are computed using formulae (\ref{s-tensor-diff-u-v-ref}) and (\ref{s-tensor-diff-u-v-2}); thus, it remains to 
compute the Malliavin derivatives $D_v Z^{s,t}_u$ of the interpolating path.

$\bullet$ When $u\leq v$ we have
$$
 Z^{s,t}_u=(X_{v,t}\circ X_{u,v})\circ\overline{X}_{s,u}=X_{v,t}\circ  Z^{s,v}_u
$$
In this situation, as in (\ref{ref-chain-r}) using  the chain  rule properties of Malliavin derivatives  we check that
$$
D_v Z^{s,t}_u=D_v Z^{s,v}_u~((\nabla X_{v,t})\circ Z^{s,v}_u )=((D_vX_{u,v})\circ \overline{X}_{s,u})~((\nabla X_{v,t})\circ Z^{s,v}_u )
$$
By (\ref{first-Malliavin}) we conclude that
$$
D_v Z^{s,t}_u=(\sigma_v \circ Z^{s,v}_u)~((\nabla X_{v,t})\circ Z^{s,v}_u )
$$

$\bullet$ When $v\leq u$ we have
$$
 Z^{s,t}_u=X_{u,t}\circ (\overline{X}_{v,u}\circ\overline{X}_{s,v})=Z^{v,t}_u\circ \overline{X}_{s,v}
$$
In this situation, arguing as above we check that
$$
D_v Z^{s,t}_u=D_v \overline{X}_{s,v}~((\nabla Z^{v,t}_u)\circ\overline{X}_{s,v})=D_v \overline{X}_{s,v}~
((\nabla \overline{X}_{v,u})\circ\overline{X}_{s,v})~((\nabla X_{u,t})\circ \overline{X}_{s,u})
$$
By (\ref{first-Malliavin}) we conclude that
$$
D_v Z^{s,t}_u=(\overline{\sigma}_v \circ \overline{X}_{s,v})~((\nabla \overline{X}_{v,u})\circ\overline{X}_{s,v})~((\nabla X_{u,t})\circ \overline{X}_{s,u})
$$

%
%

\section{Some anticipative calculus}\label{lem-var-proof}

 For clarity and to avoid unnecessary sophisticated multi-index notation, we only consider one dimensional model. The proof of the results presented in this section in the general
   case  can be reproduced word-for-word for multidimensional models.

 To simplify the presentation, 
we write $\partial^n f$ the derivative of order $n\geq 1$ of a smooth function $f$. We also  set $Y_{s,t}(x):=\overline{X}_{s,t}(x)$.
 We also  reduce the analysis to the unit interval.   
In this context, for any $t\in [0,1]$ 
we set
 \begin{equation}\label{def-XYt}
 Y_{t}:=Y_{0,t}\quad \mbox{\rm and}\quad
 X^t:=X_{t,1}
 \end{equation}
 
\subsection{Extended two-sided stochastic integrals}\label{sec-extended-2-sided}
The aim of this section is to extend the two-sided stochastic integration introduced in~\cite{pardoux-protter} to Skorohod integrals of the form
(\ref{sk-integral}), for some time homogeneous function $\varsigma_u=\varsigma$   satisfying (\ref{hyp-varsigma}).
For any $t\in [0,1]$ 
we set
 \begin{equation}\label{def-Phi}
\Phi( X^{t},Y_t(x)):=\partial X^{t}(Y_{t}(x))~ \varsigma(Y_{t}(x))
 \end{equation}
 In this notation the limiting integral in (\ref{sk-integral}) takes formally the following form
 $$
S_{0,1}(\varsigma)(x):= \int_0^1~\Phi( X^{t},Y_t(x))~dW_t
$$
The existence of this two-sided stochastic integral is discussed below in (\ref{ref-2-sided-Phi}).

To simplify the presentation, we fix the state variable $x$ and we write $Y_{t}$ and $\Phi( X^{t},Y_t)$ instead of $Y_{t}(x)$ and $\Phi( X^{t},Y_t(x))$. Next technical lemma provided a more explicit description of the Malliavin derivatives of the processes $\Phi(X^t,Y_t)$.
\begin{lem}\label{lem-tex}
For any $s<t$ we have
$$
 D_s\,\Phi(X^t,Y_t)=\left[
\partial((\partial X^{t})\circ Y_{s,t})(Y_s)~(\varsigma\circ Y_{s,t})(Y_{s})+((\partial X^{t})\circ Y_{s,t})(Y_{s})\times  \partial(\varsigma\circ Y_{s,t})(Y_s)\right]~\overline{\sigma}(Y_{s})
$$
In addition, we have
$$
 D_t\,\Phi(X^s,Y_s)=\left[\partial((\partial X^{t})\circ X_{s,t})(Y_s)~(\sigma\circ X_{s,t})(Y_s)+\partial(X^{t}\circ X_{s,t})(Y_s)~\partial\sigma(X_{s,t}(Y_s))~\right]\varsigma(Y_s)
$$
\end{lem}

\proof

Using the chain rules properties, for any $s<t$ we have
$$
\begin{array}{l}
 \displaystyle
 D_s\,\Phi(X^t,Y_t)=D_s\left((\partial X^t)(Y_{s,t}(Y_s))~(\varsigma\circ Y_{s,t})(Y_s)\right)\\
\\
=D_s\left((\partial X^t)\circ Y_{s,t})(Y_s)\right)~(\varsigma\circ Y_{s,t})(Y_s)+(\partial X^t)\circ Y_{s,t})(Y_s)~D_s(\varsigma\circ Y_{s,t})(Y_s)\\
 \end{array}$$
 The end of the proof of the first assertion comes from the fact that
 $$
 D_s\left((\partial X^t)\circ Y_{s,t})(Y_s)\right) =\partial((\partial X^{t})\circ Y_{s,t})(Y_s)~D_sY_s\quad
 \mbox{\rm with}\quad  D_sY_s=\overline{\sigma}(Y_{s})
 $$
 In the same vein, we have
 $$
 D_s(\varsigma\circ Y_{s,t})(Y_s)= \partial(\varsigma\circ Y_{s,t})(Y_s)~\overline{\sigma}(Y_{s})
 $$

We also have that
 $$
\begin{array}{l}
 \displaystyle
 D_t\,\Phi(X^s,Y_s)=D_t\left((\partial X^s)(Y_s)~\varsigma(Y_s)\right)\\
\\
=D_t\left(\partial (X^{t}\circ X_{s,t})(Y_s)\right)~\varsigma(Y_s)
=D_t\left(\left((\partial X^{t})\circ X_{s,t}\right)(Y_s)~(\partial X _{s,t})(Y_s) \right)~\varsigma(Y_s)
 \end{array}$$
 
 The last assertion comes from the fact that
  $$
\begin{array}{l}
 \displaystyle
D_t\,\left(\left((\partial X^{t})\circ X_{s,t}\right)(Y_s)~(\partial X _{s,t})(Y_s) \right)~\\
\\
=D_t\left((\partial X^{t})\circ X_{s,t}\right)(Y_s)~(\partial X _{s,t})(Y_s) +\left((\partial X^{t})\circ X_{s,t}\right)(Y_s)~D_t(\partial X _{s,t})(Y_s) 
 \end{array}$$
The r.h.s. term in the above display can be rewritten as follows
 $$
 \begin{array}{l}
 \displaystyle
D_t(\partial X _{s,t})(Y_s)=\partial\sigma(X_{s,t}(Y_s))~(\partial X _{s,t})(Y_s)\\
\\
\Longrightarrow ((\partial X^{t})\circ X_{s,t})(Y_s)~D_t(\partial X _{s,t})(Y_s)=\partial(X^{t}\circ X_{s,t})(Y_s)~\partial\sigma(X_{s,t}(Y_s))~
 \end{array} $$
 In the same vein, we have
  $$
 \begin{array}{l}
 \displaystyle D_t\left((\partial X^{t})\circ X_{s,t}\right)(Y_s)=((\partial^2 X^{t})\circ X_{s,t})(Y_s)~D_tX_{s,t}(Y_s)=((\partial^2 X^{t})\circ X_{s,t})(Y_s)~\sigma(X_{s,t}(Y_s))\\
\\
\Longrightarrow D_t\left((\partial X^{t})\circ X_{s,t}\right)(Y_s)~(\partial X _{s,t})(Y_s) =\partial((\partial X^{t})\circ X_{s,t})(Y_s)~\sigma(X_{s,t}(Y_s))
 \end{array} $$
 
This ends the proof of the second assertion. The proof of the lemma is now completed.
\cqfd

 From the above lemma, we also check that
  all the $n$-absolute moments of the Malliavin derivatives $ D_s\,\Phi(X^t,Y_t)$
  are finite with at most quadratic growth w.r.t. the initial values.

Next proposition extends proposition 3.3 in~\cite{pardoux-protter} to stochastic processes of the form (\ref{def-Phi}).

 \begin{prop}\label{k-prop}
 Let $ [0,1]_h$ be any refining sequence of partitions of the unit interval. For any $h>0$ we define
 $$
 S^{h}(\Phi):=\sum_{t\in[0,1]_h} 
\Phi( X^{t+h},Y_t)~(W_{t+h}-W_t)
 $$
 Then $ S^{h}(\Phi)$ is a Cauchy sequence in $\LL_2(\Omega)$. In addition, for any decreasing sequence of time steps
 $h_1>h_2$ 
 we have the formula
  \begin{equation}\label{cov-cauchy}
 \displaystyle\lim_{h_1\rightarrow 0}\EE\left( S^{h_1}(\Phi)~S^{h_2}(\Phi)\right) =\EE\left(
\int_0^1~\Phi( X^{t},Y_t)^2~dt+\int_{[0,1]^2}~ D_s\,\Phi(X^t,Y_t)~ D_t\,\Phi(X^s,Y_s)~ds~dt
\right)
\end{equation}
 \end{prop}
 
 Before entering into the details of the proof of the proposition, we give a couple of comments.
 The hypothesis that $[0,1]_h$ is a refining sequence indexed by $h$ is not essential but it simplifies the proof of the proposition, see for instance lemma 3.1.1 in~\cite{nualart}.
Arguing as in the proof of theorem 3.3 and theorem 7.1 in~\cite{pardoux-protter} the above proposition ensures that the two-sided integral
defined by the $\LL_2(\Omega)$-limit
coincides with the  two-sided stochastic integral of the process $\Phi( X^{t},Y_t)$ over the unit interval; that is, we have that
\begin{equation}\label{ref-2-sided-Phi}
\int_0^1 \Phi( X^{t},Y_t)~dW_t:=\LL_2-\lim_{h\rightarrow 0}\sum_{t\in[0,1]_h} 
\Phi( X^{t+h},Y_t)~(W_{t+h}-W_t)
\end{equation}
In this context, proposition~\ref {k-prop} can be interpreted as a version of the isometry property (\ref{isometry}) for the generalized two-sided integral defined above.

{\bf Proof of proposition~\ref{k-prop}:}

We fix $h_1>h_2$ and we assume that $[0,1]_{h_2} $ is a refinement of $\in [0,1]_{h_1}$. For any $(s,t)\in ([0,1]_{h_1}\times [0,1]_{h_2})$ we also set
$$
\Pi_{s,t}^{h_1,h_2}:=\Phi( X^{s+h_1},Y_s)~\Phi( X^{t+h_2},Y_t)~(W_{s+h_1}-W_s)~(W_{t+h_2}-W_t)
$$
With a slight abuse of notation we set
$$
\Delta W_s:=(W_{s+h_1}-W_s)\quad \mbox{\rm and}\quad \Delta W_t:=(W_{t+h_2}-W_t)
$$

$\bullet$ For any overlapping pair $s<t<t+h_2<s+h_1$  using the decomposition
  $$
 \Delta W_s= (W_{s+h_1}-W_{t+h_2})+ \Delta W_t+(W_t-W_s)
  $$
  we have
     $$
 \begin{array}{l}
\displaystyle\EE\left(~\Phi( X^{s+h_1},Y_s)~\Phi( X^{t+h_2},Y_t)~\Delta W_t~ \Delta W_s~|~\Wa_t\vee\Wa^{t+h_2}\right)=\Phi( X^{s+h_1},Y_s)~\Phi( X^{t+h_2},Y_t)~h_2
\end{array}$$ 
It follows from the continuity properties of the processes that
$$
\EE\left(\sum_{s<t<t+h_2<s+h_1}~\Pi_{s,t}^{h_1,h_2}\right)\longrightarrow_{h_1\rightarrow 0}~\EE\left(\int_0^1~\Phi( X^{t},Y_t)^2~~dt\right)
$$

$\bullet$ When $s+h_1<t$ we have
\begin{eqnarray*}
 \partial X^{s+h_1}&=& \partial (X^t\circ X_{s+h_1,t})\\
 &=&\partial (X^{t+h_2}\circ X_{s+h_1,t})+\left(\left(
 \partial X^t-\partial X^{t+h_2}\right)\circ X_{s+h_1,t}\right)~\times~\partial X_{s+h_1,t}
\end{eqnarray*}
On the other hand, we have the decomposition
\begin{eqnarray*}
 \partial X^t-\partial X^{t+h_2}&=&\left(( \partial X^{t+h_2})\circ X_{t,t+h_2}\right)\times \partial X_{t,t+h_2}-\partial X^{t+h_2}\\
 &=&\left( (\partial X^{t+h_2})\circ (I+\Delta X_t)-\partial X^{t+h_2}\right)+\partial X^{t+h_2}\times \Delta X^{\prime}_t\\
 &&\hskip3cm+\left(( \partial X^{t+h_2})\circ (I+\Delta X_t)-\partial X^{t+h_2}\right)\times \Delta X^{\prime}_t
\end{eqnarray*}
with the increment functions
$$
\Delta X^{\prime}_t:=\partial X_{t,t+h_2}-1\quad \mbox{\rm and}\quad \Delta X_t:=X_{t,t+h_2}-I
$$
With a slight abuse of notation, we shall
 denote by $\mbox{\rm O}(h^p)$ some possible random variable with any $n$-absolute moment of order $h^p$, for some $p>0$ with 
 $0<h<1$. 
In this notation, we have
\begin{eqnarray*}
\Delta X^{\prime}_t(x)&=&\int_t^{t+h_2}~\partial\sigma (X_{t,u}(x))~\partial X_{t,u}(x)~dW_u+\mbox{\rm O}(h_2)=\mbox{\rm O}(h^{1/2}_2)\\
\Delta X_t(x)&=&\int_t^{t+h_2}~\sigma (X_{t,u}(x))~dW_u+\mbox{\rm O}(h_2)=\mbox{\rm O}(h^{1/2}_2)
\end{eqnarray*}
Given 
a smooth function $\theta$  we set
$$
\partial^n\theta(x,y):=\int_0^1~\frac{(1-\epsilon)^{n-1}}{(n-1)!}~\theta^{\prime\prime}(x+\epsilon y)~d\epsilon
$$
In this notation, we have the first and second order decompositions
$$
\begin{array}{l}
\left(( \partial X^{t+h_2})\circ (I+\Delta X_t)-\partial X^{t+h_2}\right)(x)\\
\\
=(\partial^2 X^{t+h_2})(x,\Delta X_t(x))~\Delta X_t(x)
=(\partial^2 X^{t+h_2})(x)~\Delta X_t(x)+(\partial^3 X^{t+h_2})(x,\Delta X_t(x))~\Delta X_t(x)^2
\end{array}
$$
This implies that
$$
\begin{array}{l}
 (\partial X^t-\partial X^{t+h_2})(x)\\
 \\
 =(\partial^2 X^{t+h_2})(x)~\Delta X_t(x)+\partial X^{t+h_2}(x)\times \Delta X^{\prime}_t(x)\\
 \\
 \hskip.3cm+(\partial^3 X^{t+h_2})(x,\Delta X_t(x))~\Delta X_t(x)^2+(\partial^2 X^{t+h_2})(x,\Delta X_t(x))~\Delta X_t(x)\times \Delta X^{\prime}_t(x)
\end{array}
$$
from which we conclude that
$$
\begin{array}{l}
 \displaystyle\partial X^{s+h_1}=
\partial (X^{t+h_2}\circ X_{s+h_1,t})\\
 \\
 \displaystyle +\left[\partial((\partial X^{t+h_2})\circ X_{s+h_1,t})~\times~((\Delta X_t)\circ X_{s+h_1,t})+\partial (X^{t+h_2}\circ X_{s+h_1,t})~\times~( (\Delta X^{\prime}_t)\circ X_{s+h_1,t})\right]+\mbox{\rm O}(h_2)
\end{array}
$$
This yields the first order decomposition
$$
\begin{array}{l}
 \displaystyle\Phi(X^{s+h_1},Y_s)\\
 \\
 \displaystyle=
\psi^{0}_{s,t}(Y_s) +\psi^{1}_{s,t}(Y_s)~((\Delta X_t)\circ X_{s+h_1,t})(Y_s)+\psi^{2}_{s,t}(Y_s)~( (\Delta X^{\prime}_t)\circ X_{s+h_1,t})(Y_s)+\mbox{\rm O}(h_2)
\end{array}
$$
with the functions
$$
\begin{array}{l}
\psi^{0}_{s,t}(Y_s):=
\partial (X^{t+h_2}\circ X_{s+h_1,t})(Y_s)~ \varsigma(Y_s)\\
\\
\psi^{1}_{s,t}(Y_s):=\partial((\partial X^{t+h_2})\circ X_{s+h_1,t})(Y_s)~ \varsigma(Y_s)~\quad \mbox{\rm and}\quad
\psi^{2}_{s,t}(Y_s):=\partial (X^{t+h_2}\circ X_{s+h_1,t})(Y_s)~ \varsigma(Y_s)
\end{array}
$$
Notice that none of the functions but the increment functions $(\Delta X_t)$ and $(\Delta X^{\prime}_t)$
 depend on $\Wa_{t,t+h_2}$, nor on $\Wa_{s,s+h_1}$.
 
 In the reverse angle, we have
 $$
 \begin{array}{l}
 \displaystyle(\partial X^{t+h_2})\circ Y_{t}\\
 \\
  \displaystyle=(\partial X^{t+h_2})\circ(Y_{s+h_1,t}\circ Y_{s})+\left[((\partial X^{t+h_2})\circ Y_{s+h_1,t})\circ (I+\Delta Y_s)-
((\partial X^{t+h_2})\circ Y_{s+h_1,t})\right]\circ Y_{s}
\end{array} $$
with
$$
\Delta Y_s:=(Y_{s,s+h_1}-I)\Longrightarrow Y_{s+h_1}=(I+\Delta Y_s)\circ Y_s
$$
Arguing as above, we have
 $$
 \begin{array}{l}
 \displaystyle\left[((\partial X^{t+h_2})\circ Y_{s+h_1,t})\circ (y+\Delta Y_s(y))-
((\partial X^{t+h_2})\circ Y_{s+h_1,t})(y)\right]\\
\\
=\partial((\partial X^{t+h_2})\circ Y_{s+h_1,t})(y)~\Delta Y_s(y)+\partial^2((\partial X^{t+h_2})\circ Y_{s+h_1,t})(y,\Delta Y_s(y))~(\Delta Y_s(y))^2
\end{array} $$
We conclude that
 $$
 \begin{array}{l}
 \displaystyle(\partial X^{t+h_2})\circ Y_{t}\\
 \\
  \displaystyle=(\partial X^{t+h_2})\circ(Y_{s+h_1,t}\circ Y_{s})+\partial((\partial X^{t+h_2})\circ Y_{s+h_1,t})(Y_s)~((\Delta Y_s)\circ Y_s)+\mbox{\rm O}(h_1)\end{array} $$
In the same vein, we have
 $$
 \displaystyle\varsigma\circ Y_{t}=(\varsigma\circ Y_{s+h_1,t}\circ Y_{s})+\partial(\varsigma\circ Y_{s+h_1,t})(Y_s)~((\Delta Y_s)\circ Y_s)+\mbox{\rm O}(h_1)
 $$
 Multiplying these terms, we check that
  $$
\Phi(X^{t+h_2},Y_t)=\Psi^{0}_{s,t}(Y_s)
  +\Psi^{1}_{s,t}(Y_s)~((\Delta Y_s)\circ Y_s)+\mbox{\rm O}(h_1)
  $$
  with the functions
$$
\begin{array}{l}
\Psi^{0}_{s,t}(Y_s):=((\partial X^{t+h_2})\circ Y_{s+h_1,t})(Y_{s})\times (\varsigma\circ Y_{s+h_1,t})(Y_{s})\\
\\
\Psi^{1}_{s,t}(Y_s):=\left[\partial((\partial X^{t+h_2})\circ Y_{s+h_1,t})(Y_s)\times(\varsigma\circ Y_{s+h_1,t})(Y_{s})\right.\\
  \\
\hskip3cm+\left.((\partial X^{t+h_2})\circ Y_{s+h_1,t})(Y_{s})\times  \partial(\varsigma\circ Y_{s+h_1,t})(Y_s)\right]
\end{array}
$$
None of the functions but the increment  $\Delta Y_s$  depend on $\Wa_{s,s+h_1}$, nor on $\Wa_{t,t+h_2}$.

Recall that the functions $\Phi(X^{t+h_2},Y_t)$ and $\psi^{0}_{s,t}(Y_s)$ don't depend on $\Delta W_t$. In addition, 
the functions $\Phi(X^{s+h_1},Y_s)$ and $\Psi^{0}_{s,t}(Y_s)$ don't depend on $\Delta W_s$. This yields the formula
$$
\begin{array}{l}
 \displaystyle\EE\left(\Phi(X^{s+h_1},Y_s)~\Phi(X^{t+h_2},Y_t)~~\Delta W_s~\Delta W_t\right)\\
 \\
 = \displaystyle\EE\left(\left[\Phi(X^{s+h_1},Y_s)-\psi^{0}_{s,t}(Y_s)\right]~\left[\Phi(X^{t+h_2},Y_t)-\Psi^{0}_{s,t}(Y_s)\right]~\Delta W_s~\Delta W_t\right)\\
 \\
 = \displaystyle\EE\left(\Psi^{1}_{s,t}(Y_s)~ \psi^{1}_{s,t}(Y_s)~ 
 \left[((\Delta Y_s)\circ Y_s)~~\Delta W_s\right]~\left[((\Delta X_t)\circ X_{s+h_1,t})(Y_s)~\Delta W_t\right]\right)\\
 \\
 + ~\displaystyle\EE\left( \Psi^{1}_{s,t}(Y_s)~\psi^{2}_{s,t}(Y_s)~\left[((\Delta Y_s)\circ Y_s)~\Delta W_s\right]~\left[
 ( (\Delta X^{\prime}_t)\circ X_{s+h_1,t})(Y_s)~\Delta W_t\right]\right)+\mbox{\rm O}\left(h_1^{2+1/2}\right)
\end{array}
$$
To take the final step, observe that
$$
\begin{array}{l}
 \displaystyle
\EE\left(\Delta Y_s(y)~\Delta W_s\right)\\
\\
 \displaystyle=\EE\left(\int_s^{s+h_1}~\overline{b}(Y_{s,u}(y))~(W_{s+h_1}-W_u)~du\right)+\EE\left(\int_s^{s+h_1}~\overline{\sigma}(Y_{s,u}(y))~du\right)\\
 \\
  \displaystyle=\EE\left(\int_s^{s+h_1}~\overline{\sigma}(Y_{s,u}(y))~du\right)+\mbox{\rm O}\left(h_1^{1+1/2}\right)
\end{array}$$
In the same vein, we have
$$
\begin{array}{l}
 \displaystyle
\EE\left( (\Delta X_t) (X_{s+h_1,t}(y))~\Delta W_t~|~\Wa_{s+h_1,t}\right)\\
\\
  \displaystyle=\EE\left(\int_t^{t+h_2}~\sigma(X_{s+h_1,u}(y))~du~|~\Wa_{s+h_1,t}\right)+\mbox{\rm O}\left(h_2^{1+1/2}\right)
\end{array}$$
and
$$
\begin{array}{l}
 \displaystyle
\EE\left( (\Delta X^{\prime}_t)(X_{s+h_1,t}(y))~\Delta W_t~|~\Wa_{s+h_1,t}\right)\\
\\
  \displaystyle=\EE\left(\int_t^{t+h_2}~\partial\sigma(X_{s+h_1,u}(y))~(\partial X_{t,u})(X_{s+h_1,t}(y))~du\right)+\mbox{\rm O}\left(h_2^{1+1/2}\right)\\
  \\
    \displaystyle=\EE\left(\int_t^{t+h_2}~\partial\left(\sigma\circ X_{t,u}\right)(X_{s+h_1,t}(y))~du\right)+\mbox{\rm O}\left(h_2^{1+1/2}\right)
\end{array}$$
This shows that
$$
\begin{array}{l}
 \displaystyle h_1^{-1}~h_2^{-1}\EE\left(\Phi(X^{s+h_1},Y_s)~\Phi(X^{t+h_2},Y_t)~~\Delta W_s~\Delta W_t\right)\\
 \\
 = \displaystyle\EE\left(\Psi^{1}_{s,t}(Y_s)~ \psi^{1}_{s,t}(Y_s)~ 
h_1^{-1} \left[\int_s^{s+h_1}~\overline{\sigma}(Y_{u})~du\right]~h_2^{-1}\left[\int_t^{t+h_2}~\sigma(X_{s+h_1,u}(Y_s))~du\right]\right)\\
 \\
 + ~\displaystyle\EE\left( \Psi^{1}_{s,t}(Y_s)~\psi^{2}_{s,t}(Y_s)~h_1^{-1}\left[\int_s^{s+h_1}~\overline{\sigma}(Y_{u})~du\right]~h_2^{-1}\left[
\int_t^{t+h_2}~\partial\left(\sigma\circ X_{t,u}\right)(X_{s+h_1,t}(Y_s))~du\right]\right)+\mbox{\rm O}\left(h_1^{1/2}\right)
\end{array}
$$
It follows that
$$
\begin{array}{l}
 \displaystyle\lim_{h_1\rightarrow 0}\EE\left(\sum_{s+h_1<t}~\Pi_{s,t}^{h_1,h_2}\right)\\
 \\
\displaystyle=\EE\left(
\int_0^1\int_0^t 
\left[
\partial((\partial X^{t})\circ Y_{s,t})(Y_s)~(\varsigma\circ Y_{s,t})(Y_{s})+((\partial X^{t})\circ Y_{s,t})(Y_{s})\times  \partial(\varsigma\circ Y_{s,t})(Y_s)\right]~\overline{\sigma}(Y_{s})~\right.
\\
\\
\left.\hskip3cm\times\left[
\partial((\partial X^{t})\circ X_{s,t})(Y_s)~
\sigma(X_{s,t}(Y_s))+\partial (X^{t}\circ X_{s,t})(Y_s)~\partial\sigma(X_{s,t}(Y_s))\right] \varsigma(Y_s)~~ds~dt
\right)
 \end{array}
$$
We end the proof of (\ref{cov-cauchy}) 
using lemma~\ref{lem-tex} and symmetry arguments. This ends the proof of the proposition.
\cqfd

\subsection{Generalized backward It\^o-Ventzell formula}\label{biv-proof}

This section is mainly concerned with the proof of theorem~\ref{biv}. Before entering into the details of the proof we discuss how it applies
to the process $(X^t,Y_t)$  introduced in (\ref{def-XYt}).

Consider the random fields
\begin{eqnarray}
F_{t}(x)&:=&X^t(x)\Longrightarrow
\partial F_t=\partial X^t\quad \mbox{\rm and}\quad
\partial^2 F_{t}=\partial^2 X^t\nonumber\\
G_t(x)&:=&
\partial X^t(x)~b(x)+\frac{1}{2}~\partial^2 X^t(x)~a(x)~
\quad \mbox{\rm and}\quad H_t(x):=\partial X^t(x)~\sigma(x)\label{def-rf}
\end{eqnarray}
In this notation, the backward random field formula (\ref{ref-backward-flow}) with $t\in [0,1]$ takes the form
\begin{equation}\label{back-random-field-def}
F_{t}(x):=F_{1}(x)+\int_t^1~G_s(x)~ds+\int_{t}^1~H_s(x)~dW_s\quad\mbox{\rm with}\quad F_1(x)=x
\end{equation}
We fix some given $Y_0=y\in\RR$ and we write $Y_t$ instead of $Y_t(y)$ and set
$$
\left(A_u,B_u,\Sigma_u\right):=\left(\overline{a}(Y_u),\overline{b}(Y_u),\overline{\sigma}(Y_u)\right)
$$ In this notation, we have
\begin{equation}\label{Y0-ref}
Y_t=y+\int_0^tB_u~du+\int_0^t\Sigma_u~dW_u\end{equation}
Observe that $B_{u},\Sigma_{u}$ as well as the Malliavin derivatives $D_v\Sigma_{u}=\partial \overline{\sigma}(Y_u)~D_vY_u$ have moments of any order. 
Consider the processes
\begin{eqnarray*}
U_t&:=&\partial F_t(Y_t)~B_t+\frac{1}{2}~
\partial^2 F_t(Y_t)~A_t-G_t(Y_t)=\partial X^t (Y_t)~(\overline{b}-b)(Y_t)+\frac{1}{2}~
\partial^2 X^t(Y_t)~(\overline{a}-a)(Y_t)
\\
V_t&:=&\partial F_t(Y_t)~\Sigma_t-H_t(Y_t)=\partial X^t(Y_t)~(\overline{\sigma}-\sigma)(Y_t)\quad\mbox{\rm with}\quad A_t:=\Sigma_t^2
\end{eqnarray*}
In this notation, up to a change of sign and replacing $x$ by $Y_0$ in (\ref{Alekseev-grobner})  the stochastic interpolation formula stated in theorem~\ref{theo-al-gr} on the unit interval  takes  the following form
$$
F_1(Y_1)-F_0(Y_0)=\int_0^1 U_s~ds+\int_0^1 V_s~dW_s
$$

More generally, suppose we are given a  forward real valued continuous semi-martingale $Y_{t}$ 
of the form (\ref{Y0-ref}) for some  $\Wa_{0,t}$-adapted functions
$B_{t}$ and $\Sigma_{t}$,
and a backward random
field models of the form (\ref{back-random-field-def}) for some $\Wa_{t,1}$-adapted functions
$F_t(x),G_t(x),H_t(x)$. \\

We consider  the following conditions:\\

{\em
$(H_1)^{\prime}$: The functions $F_t(x)$, $G_t(x)$ and $H_t(x)$ as well as the differentials $\partial H_t(x)$ and  $\partial^2F_t(x)$ are  continuous w.r.t. $(t,x)$   for any given $\omega\in\Omega$. In addition, for any $n\geq 1$ we have
\begin{equation}\label{integral-H1}
\begin{array}{rclc}
\displaystyle\sup_{\vert y\vert\leq n}\left(\vert F_t(Y_t+y)\vert\vee \vert H_t(Y_t+y)\vert\vee \vert G_t(Y_t+y)\vert\right)& \leq& g_n(t)&
\\
\displaystyle\sup_{\vert y\vert\leq n}\left(\vert \partial H_t(Y_t+y)\vert\vee\vert \partial F_t(Y_t+y)\vert\vee\vert \partial^2 F_t(Y_t+y)\vert\right)& \leq& g_n(t)\quad \mbox{\it with}&\displaystyle\EE\left(\int_0^1\,g^4_n(t)\,dt\right)<\infty
\end{array}
\end{equation}

$(H_2)^{\prime}$: The Malliavin derivatives $ D_s\partial F_t(x)$ and $D_sH_t(x)$ are  continuous w.r.t. $x$ and $(s,t)$ for any given  $\omega\in\Omega$. In addition,  for any $n\geq 1$ we have 
\begin{equation}\label{integral-H2}
\begin{array}{rclc}
\displaystyle\sup_{\vert y\vert\leq n}\left( \vert (D_sF_t)(Y_t+y)\vert\vee \vert (D_sH_t)(Y_t+y)\vert\right) &\leq& h_n(s,t)&\\
\sup_{\vert y\vert\leq n}\left( \vert (D_s\partial F_t)(Y_t+y)\vert\right) &\leq& h_n(s,t)\quad \mbox{\rm with}&\displaystyle\EE\left(\int_{[0,1]^2}\,h^4_n(s,t)\,dsdt\right)<\infty
\end{array}
\end{equation}
$(H_3)$: The random processes   $B_{u},\Sigma_{u}$  as well as  $D_v\Sigma_{u}$ are continuous w.r.t. the time parameter and they have moments of any order. }\\

The next theorem is a slight extension of theorem~\ref{biv} applied to the semi-martingale and the random fields models discussed in
(\ref{Y0-ref}) and (\ref{def-rf}).

\begin{theo}\label{theo-iv}
Consider a backward random
field models of the form (\ref{back-random-field-def}) for some functions
$F_t(x),G_t(x),H_t(x)$ satisfying $(H_1)^{\prime}$ and $(H_2)^{\prime}$. Also let $Y_{t}$ be a continuous semi-martingale 
of the form (\ref{Y0-ref})  functions
$B_{t}$ and $\Sigma_{t}$ satisfying $(H_3)$.
In this situation, for any $t\in [0,1]$ we have the generalized backward It\^o-Ventzell formula 
\begin{equation}\label{random-fields-version}
\begin{array}{l}
F_t(Y_t)-F_0(Y_0)\\
\\
\displaystyle=\int_0^t~\left(\partial F_s(Y_s)~B_s+\frac{1}{2}~
\partial^2 F_s(Y_s)~A_s-G_s(Y_s)\right)~ds+\int_0^t~\left(\partial F_s(Y_s)~\Sigma_s-H_s(Y_s)\right)~dW_s
\end{array}
\end{equation}
The r.h.s. term in the above display is understood as a Skorohod integral.
\end{theo}

{\bf Proof:} We use the same approximation technique as in~\cite{bismut-2,ocone-pardoux} and~\cite{pardoux-90} (see also the proof of theorem 3.2.11 in~\cite{nualart}).
Consider a mollifier type approximation of the identify given for any $\epsilon>0$ by the function
$$
\varphi_{\epsilon}(x):=\varphi(x/\epsilon)/\epsilon
\quad\mbox{\rm for some smooth compactly supported function $\varphi$ s.t. $\int_{-\infty}^{\infty}\varphi(x)dx=1$. }
$$ 

For any $x$, applying the It\^o-type change rule formula stated in  proposition 8.2 in~\cite{nualart-pardoux}  to the product function
$$\Gamma(X^t(x),\varphi_{\epsilon}(Y_t-x)):=X^t(x)~\varphi_{\epsilon}(Y_t-x)
$$
we check that
\begin{equation}\label{ref-two-sided-ito}
\displaystyle(F_t(x)~\varphi_{\epsilon}(Y_t-x))-(F_0(x)~\varphi_{\epsilon}(Y_0-x))
\displaystyle=\int_0^t~u_s^{\epsilon}(x)~ds+\int_0^t~v_s^{\epsilon}(x)~dW_s
\end{equation}
with
\begin{eqnarray*}
u_s^{\epsilon}(x)&:=&F_s(x)~\partial \varphi_{\epsilon}(Y_s-x)~B_s+\frac{1}{2}~
F_s(x)~\partial^2 \varphi_{\epsilon}(Y_s-x) ~A_s-\varphi_{\epsilon}(Y_s-x)~G_s(x)\\
v_s^{\epsilon}(x)&:=&F_s(x)~\partial \varphi_{\epsilon}(Y_s-x)~ \Sigma_s-\varphi_{\epsilon}(Y_s-x)~H_s(x)
\end{eqnarray*}
The stochastic integral in the r.h.s. of (\ref{ref-two-sided-ito}) can be interpreted as a two-sided stochastic integral.
Recalling that
$$
D_t\varphi_{\epsilon}(Y_s-x)=D_tY_s~ \partial\varphi_{\epsilon}(Y_s-x)
$$
we check that
\begin{eqnarray*} D_tv_s^{\epsilon}(x)
&=&
D_tF_s(x)~\partial \varphi_{\epsilon}(Y_s-x)~ \Sigma_s+F_s(x)~D_tY_s~ \partial^2\varphi_{\epsilon}(Y_s-x)~ \Sigma_s\\
&&+F_s(x)~\partial \varphi_{\epsilon}(Y_s-x)~ D_t\Sigma_s-D_tY_s~ \partial\varphi_{\epsilon}(Y_s-x)~H_s(x)-\varphi_{\epsilon}(Y_s-x)~D_tH_s(x)
\end{eqnarray*}
Condition $(H_3)$ ensures that the processes $Y_t$ and $D_tY_s$ have moments of any order. In addition, under the regularity conditions  $(H_1)^{\prime}$ and $(H_2)^{\prime}$ we check that
$$
\int~\EE\left(\int_0^t~u_s^{\epsilon}(x)^2~ds\right)~dx<\infty\quad\mbox{\rm and}\quad
\int~\EE\left(\left[\int_0^t~v_s^{\epsilon}(x)~dW_s\right]^2\right)~dx<\infty
$$
Applying the Fubini theorem for Skorohod and measure theory integrals (see for instance~\cite{kruk,nualart,purtu} and the work by Leon~\cite{leon}) we check that
$$
F_t^{\epsilon}(Y_t):=\int~F_t(x)~\varphi_{\epsilon}(Y_t-x)~dx=\int~F_0(x)~\varphi_{\epsilon}(Y_0-x)~dx
+\int_0^t~U^{\epsilon}_s~ds=\int_0^t~V^{\epsilon}_s~dW_s
$$ 
with
$$
U^{\epsilon}_s:=\int u_s^{\epsilon}(x)~dx\quad \mbox{\rm and}\quad
V^{\epsilon}_s:=\int v_s^{\epsilon}(x)~dx
$$
Integrating by parts where derivatives of $\varphi_{\epsilon}$ appear we check that
\begin{eqnarray*}
U^{\epsilon}_s&:=&\int~\left(\partial F_s(x)~B_s+\frac{1}{2}~
\partial^2 F_s(x)~A_s-G_s(x)\right)~\varphi_{\epsilon}(Y_s-x)~dx\\
V^{\epsilon}_s&:=&\int~\left(\partial F_s(x)~ \Sigma_s-~H_s(x)\right)~ \varphi_{\epsilon}(Y_s-x)~dx
\end{eqnarray*}
From the a.s. continuity of $F_t(x)$ in $x$ for each $t\geq 0$, we have
$$
F_t^{\epsilon}(Y_t)-F_t(Y_t)=\int~(F_t(Y_t-\epsilon~x)-F_t(Y_t))~\varphi(x)~dx~\longrightarrow_{\epsilon\rightarrow 0}~0
$$
The functions $\partial F_t(x)$, $\partial^2 F_t(x)$  and $G_t(x)$ are almost surely continuous w.r.t. $x$ and uniformly locally bounded. In addition,
the random variables  $A_t$ and $B_t$ are integrable at any order. Moreover, under $(H_1)^{\prime}$ there exists some parameter $n\geq 0$ depending on the support of $\varphi$ such that for any $\epsilon>0$ we have the estimate
\begin{eqnarray*}
\vert U^{\epsilon}_s\vert
&\leq& \sup_{\vert y\vert\leq n}{\vert \partial F_s(Y_s+y)\vert}~\vert B_s\vert+\frac{1}{2}~
 \sup_{\vert y\vert\leq n}{\vert\partial^2 F_s(Y_s+y)\vert}~\vert A_s\vert+ \sup_{\vert y\vert\leq n}{\vert G_s(Y_s+y)\vert}\\&\leq& g_n(t)~(1+\vert A_s\vert+\vert B_s\vert)
\end{eqnarray*}
Thus, by the  dominated convergence theorem on $( \Omega\times [0,1])$ equipped with the measure $(\PP(d\omega)\otimes dt)$ we have
 $$
 \int_0^t~U^{\epsilon}_s~ds\longrightarrow_{\epsilon\rightarrow 0}~
\int_0^t~U_s~ds\quad \mbox{\rm as well as}\quad
 F_t^{\epsilon}(Y_t)\longrightarrow_{\epsilon\rightarrow 0}F_t(Y_t)
$$
 It remains to check that
\begin{equation}\label{V-cv}
\EE(\int_0^t (V_s^{\epsilon}-V_s)^2~ds)+\EE\left(\int_{[0,t]^2} (D_rV^{\epsilon}_s-D_rV_s)~(D_sV^{\epsilon}_r-D_sV_r)~dr~ds\right)\longrightarrow_{\epsilon\rightarrow 0}~0
\end{equation}
Observe that
$$
\begin{array}{l}
\displaystyle\int_0^t (V_s^{\epsilon}-V_s)^2~ds\\
\\
\displaystyle\leq 2~\int_0^t \int~(\partial F_s(x)-\partial F_s(Y_s))^2~\Sigma_s^2~\varphi_{\epsilon}(Y_s-x)~dx~ds+2~\int_0^t~(H_s(x)-H_s(Y_s))^2~\varphi_{\epsilon}(Y_s-x)~dx~ds
\end{array}$$

Using the chain rule property we have
$$
\begin{array}{l}
D_tV^{\epsilon}_s\\
\\
\displaystyle:=\int~D_t\left(\partial F_s(x)~ \Sigma_s-H_s(x)\right)~ \varphi_{\epsilon}(Y_s-x)~dx+
\int~\left(\partial F_s(x)~ \Sigma_s-H_s(x)\right)~ D_t\varphi_{\epsilon}(Y_s-x)~dx
\end{array}
$$
Integrating by parts, we check that
$$
D_tV^{\epsilon}_s=\int~\left[D_t\left(\partial F_s(x)~ \Sigma_s-H_s(x)\right)~ +
\left(\partial^2 F_s(x)~ \Sigma_s-\partial H_s(x)\right)~D_tY_s\right]~ \varphi_{\epsilon}(Y_s-x)~dx
$$
Observe that 
$$
\begin{array}{l}
\displaystyle D_t\left(\partial F_s(x)~ \Sigma_s-H_s(x)\right)~ +
\left(\partial^2 F_s(x)~\Sigma_s-\partial H_s(x)\right)~D_tY_s\\
\\
\displaystyle =((D_t\partial F_s)(x)+\partial^2 F_s(x)~D_tY_s)~\Sigma_s+
\partial F_s(x)~D_t\Sigma_s-((D_tH_s)(x)+\partial H_s(x)~D_tY_s)
\end{array}
$$
On the other hand, we have
$$
\begin{array}{l}
D_tV_s
=D_t(\partial F_s(Y_s))~ \Sigma_s+\partial F_s(Y_s)~ D_t\Sigma_s-D_t(H_s(Y_s))\\
\\
~~~=\left((D_t\partial F_s)(Y_s)+\partial^2F_s(Y_s)~D_tY_s\right) \Sigma_s+\partial F_s(Y_s)~ D_t\Sigma_s
-\left((D_tH_s)(Y_s)+\partial H_s(Y_s)~D_tY_s\right)
\end{array}
$$
Arguing as above, we have the estimate
$$
\begin{array}{l}
\displaystyle\EE\left(\int_{[0,1]^2} (D_rV^{\epsilon}_s-D_rV_s)~(D_sV^{\epsilon}_r-D_sV_r)~dr~ds\right)
\\
\\
\displaystyle\leq 2~\EE\left(\int_{[0,1]^2} (D_tV^{\epsilon}_s-D_tV_s)^2~ds~dt\right)\leq 2^4~\sum_{1\leq i\leq 5}~J_i(\epsilon)
\end{array}
$$
In the above display, $J_i(\epsilon)$ stands for the sequences
\begin{eqnarray*}
 J_1(\epsilon)&:=&\EE\left(\int_{[0,1]^2\times\RR}~\left(\partial F_s(x)-\partial F_s(Y_s) \right)^2~\left(D_t\Sigma_s\right)^2~ \varphi_{\epsilon}(Y_s-x)~ds~dt~dx\right)\\
 J_2(\epsilon)&:=&\EE\left(\int_{[0,1]^2\times\RR}~\left(\partial H_s(x)-\partial H_s(Y_s)\right)^2 (D_tY_s)^2~ \varphi_{\epsilon}(Y_s-x)~ds~dt~dx\right)\\
 J_3(\epsilon)&:=&\EE\left(\int_{[0,1]^2\times\RR}~\left(\partial^2 F_s(x)-\partial^2 F_s(Y_s)\right)^2~ (D_tY_s)^2~A_s~ \varphi_{\epsilon}(Y_s-x)~ds~dt~dx\right)
\end{eqnarray*}
The last two terms depend on the Malliavin derivatives of  $\partial F_s$ and $H_s$ are they are given by
\begin{eqnarray*}
 J_4(\epsilon)&:=&\EE\left(\int_{[0,1]^2\times\RR}~\left((D_t\partial F_s)(x)-(D_t\partial F_s)(Y_s)\right)^2~A_s~ \varphi_{\epsilon}(Y_s-x)~ds~dt~dx\right)\\
 J_5(\epsilon)&:=&\EE\left(\int_{[0,1]^2\times\RR}~\left((D_tH_s)(x)-(D_tH_s)(Y_s)\right)^2~ \varphi_{\epsilon}(Y_s-x)~ds~dt~dx\right)
\end{eqnarray*}

Arguing as above,  by the  dominated convergence theorem we conclude that the Skorohod integral
$$
\int_0^t~V^{\epsilon}_s~dW_s\quad \mbox{\rm converges in $\LL_2(\Omega)$ as $\epsilon\rightarrow 0$ to the Skorohod integral}~\int_0^t~V_s~dW_s
$$
This ends the proof of (\ref{V-cv}), and the proof of the theorem is now easily completed.
\cqfd

We end this section with some comments.
\begin{rmk}\label{f-rmk}

Recalling that the diffusion flow $Y_t$ introduced in (\ref{def-XYt}) has finite absolute moments of any order, the integrability conditions stated in (\ref{integral-H1}) and (\ref{integral-H2}) are satisfied 
as soon as the functions $F_t,G_t,H_t$, the differentials $\partial F_t, \partial^2 F_t, \partial H_t$, and the Malliavin derivatives
$D_sH_t,D_s\partial F_t$ have at most polynomial growth w.r.t. the state variable.

It is now readily check that $(H_1)^{\prime}$ and $(H_2)^{\prime}$ are  met for the random fields introduced in (\ref{def-rf}).

The proof  can be also be extended without difficulties to multivariate models. Following the proof of proposition 3.1 in ~\cite{ocone-pardoux}, an alternative proof of theorem~\ref{theo-iv} based on It\^o formula for Hilbert space valued processes can be developed. This elegant functional approach requires to introduce a custom Hilbert-space valued processes framework but this approach avoids to do explicitly the interchange of integration using the Fubini theorem for Skorohod and measure theory integrals. As the statement of proposition 3.1 in ~\cite{ocone-pardoux}, the assumptions of theorem~\ref{theo-iv} can also be weaken when expressed  in terms of this generalized stochastic calculus for Hilbert-space valued processes.

\end{rmk}

 \section{Illustrations}\label{sec-illustrations}
 
 \subsection{Perturbation analysis}\label{sec-perturbation}
 Assume that $\overline{\sigma}=\sigma$ and the drift function $ \overline{b}_t$ is given by a first order 
 expansion 
 $$
 \overline{b}_t(x)= b_{\delta,t}(x):= b_t(x)+\delta~  b^{(1)}_{\delta,t}(x)\quad \mbox{\rm with}\quad
  b^{(1)}_{\delta,t}(x)= b^{(1)}_t(x)+\frac{\delta}{2}~ b^{(2)}_{\delta,t}(x)
  $$
  for some perturbation parameter $\delta\in [0,1]$ and some functions $ b^{(i)}_{\delta,t}(x)$ with $i=1,2$. 
  
  In this context, the stochastic flow
 $
 \overline{X}_{s,t}(x):=X^{\delta}_{s,t}(x)
 $ can be seen as a $\delta$-perturbation of ${X}_{s,t}(x):=X^{0}_{s,t}(x)$. 
 
 We further assume that the unperturbed diffusion satisfies condition $(\Ta)_2$.
 
 To avoid unnecessary technical discussions on the existence of absolute moments of the flows we also assume that  $ b^{(i)}_{\delta,t}(x)$
 are uniformly bounded w.r.t. the parameters $(\delta,t,x)$. In addition, $b^{(1)}_t(x)$ is differentiable w.r.t.  the coordinate $x$ and it has uniformly bounded gradients.  In this situation, we set
 $$
 \Vert b^{(i)}\Vert:= \sup_{\delta,t,x}\Vert b^{(i)}_{\delta,t}(x)\Vert\quad \mbox{\rm and}\quad  \Vert \nabla 
 b^{(1)}\Vert:= \sup_{t,x}\Vert \nabla b^{(1)}_{t}(x)\Vert
 $$
 With some additional work to estimate the absolute moments of the flows,  the perturbation analysis presented below allows to handle more general models.
 The methodology described in this section can also be extended to expand the flow $X^{\delta}_{s,t}(x)$ at any order as soon as $\delta\mapsto 
 b_{\delta,t}(x)$ is sufficiently smooth.
 
The first order approximation is given by the following theorem.
\begin{theo}
For any $s\leq t$, $x\in \RR^d$ and $\delta\geq 0$ we have the first order expansion
 \begin{equation}\label{first-taylor-perturbation}
 \begin{array}{l}
\displaystyle X^{\delta}_{s,t}(x)={X}_{s,t}(x)+\delta~\partial {X}_{s,t}(x)+\frac{\delta^2}{2}~ \partial^2_{\delta} {X}_{s,t}(x)\end{array} 
 \end{equation}
 with the first order stochastic flow
 $$
\partial {X}_{s,t}(x):= \int_s^t~\left(\nabla X_{u,t}\right)(X_{s,u}(x))^{\prime}~
b^{(1)}_u(X_{s,u}(x))~du
 $$
 The remainder second order term $ \partial^2_{\delta} {X}_{s,t}(x)$ in the above display is such that 
for any $n\geq 2$ s.t. $\lambda_A(n)>0$   we have the uniform estimate
$$
\sup_{s,t,x} \EE[\Vert \partial^{2}_{\delta} {X}_{s,t}(x)\Vert^n]^{1/n}
\leq c_n
$$
\end{theo}

\proof
 
Using 
(\ref{X-Y-ref-ag}) we readily check that
\begin{eqnarray*}
DX^{\delta}_{s,t}(x)&:=&
\delta^{-1}[X^{\delta}_{s,t}(x)-{X}_{s,t}(x)]=\int_s^t~\left(\nabla X_{u,t}\right)(X^{\delta}_{s,u}(x))^{\prime}~
b^{(1)}_{\delta,u}(X^{\delta}_{s,u}(x))~du
 \end{eqnarray*}
By proposition~\ref{def-4th-prop}
 for any $n\geq 2$ we have
\begin{equation}\label{ref-lambda-plus}
\lambda_A^+(n):=\lambda_A-(n-2)\rho(\nabla\sigma)^2/2>0\Longrightarrow
\EE\left(\Vert DX^{\delta}_{s,t}(x)\Vert^{n}\right)^{1/n}\leq c~  \Vert b^{(1)}\Vert/\lambda_A^+(n)
\end{equation}
 This yields the first order Taylor expansion (\ref{first-taylor-perturbation})
 with $$\partial^2_{\delta} {X}_{s,t}(x):= \partial^{(2,1)}_{\delta} {X}_{s,t}(x)+ \partial^{(2,2)}_{\delta} {X}_{s,t}(x)$$ and the second order remainder terms
\begin{eqnarray*}
 \partial^{(2,2)}_{\delta} {X}_{s,t}(x)&:=&\int_s^t~\left(\nabla X_{u,t}\right)(X^{\delta}_{s,u}(x))^{\prime}~b^{(2)}_{\delta,t}(X^{\delta}_{s,u}(x))~du\\
 \partial^{(2,1)}_{\delta} {X}_{s,t}(x)&:=&2\delta^{-1}\int_s^t~\left[\left(\nabla X_{u,t}\right)(X^{\delta}_{s,u}(x))-\left(\nabla X_{u,t}\right)(X_{s,u}(x))\right]^{\prime}~
b^{(1)}_u(X^{\delta}_{s,u}(x))~du\\
&&\hskip2cm+2\delta^{-1}\int_s^t~\left(\nabla X_{u,t}\right)(X_{s,u}(x))^{\prime}~
[b^{(1)}_u(X^{\delta}_{s,u}(x))-b^{(1)}_u(X_{s,u}(x))]~du
 \end{eqnarray*}
 Arguing as above,  for any $n\geq 2$ s.t. $\lambda_A^+(n)>0$ we have the uniform estimate
$$
\EE\left(\Vert  \partial^{(2,2)}_{\delta} {X}_{s,t}(x)\Vert^{n}\right)^{1/n}\leq c~   \Vert b^{(2)}\Vert/\lambda_A^+(n)
$$
To estimate $ \partial^{(2,1)}_{\delta} {X}_{s,t}(x)$ we need to consider the second order decompositions
$$
\begin{array}{l}
\displaystyle 2^{-1}~ \partial^{(2,1)}_{\delta} {X}_{s,t}(x)\\
\\
\displaystyle
=~\int_0^1~\int_s^t
\left[\nabla^2 X_{u,t}\right]\left(X_{s,u}(x)+\epsilon(X^{\delta}_{s,u}(y)-X_{s,u}(x))\right)^{\prime}~\left[b^{(1)}_u(X^{\delta}_{s,u}(x))\otimes
DX^{\delta}_{s,u}(x)\right]~~du~d\epsilon\\
\\
\displaystyle+\int_0^1\int_s^t~
\left(\nabla X_{u,t}\right)(X_{s,u}(x))^{\prime}~\nabla b_u^{(1)}\left(X_{s,u}(x)+\epsilon(X^{\delta}_{s,u}(x)-X_{s,u}(x)),y\right)^{\prime}~
DX^{\delta}_{s,u}(x)~du~d\epsilon
\end{array}$$
Combining proposition~\ref{prop-nabla-2-estimate} with the estimate (\ref{ref-lambda-plus})   for any $n\geq 2$  s.t. $\lambda_A(n)>0$ we check that
$$
 \EE[\Vert \partial^{(2,1)}_{\delta} {X}_{s,t}(x)\Vert^n]^{1/n}
\leq c~\left(1+n~\rchi(b,\sigma)/\lambda_A(n)\right)~\left(\Vert b^{(1)}\Vert/\lambda_A(n)\right)^2
$$
for some universal constant $c<\infty$ and the parameter $\rchi(b,\sigma)$ introduced in (\ref{def-chi-b}). This ends  the proof of (\ref{first-taylor-perturbation}). The proof of the theorem is completed.
\cqfd

\subsection{Interacting diffusions}
Consider a system of $N$ interacting and $\RR^d$-valued diffusion flows $X^{i}_{s,t}(x)$, with $1\leq i\leq N$ given by a stochastic differential equation of the form
$$
dX^{i}_{s,t}(x)=B_t\left(X^{i}_{s,t}(x),\frac{1}{N}\sum_{1\leq i\leq N}X^j_{s,t}(x)\right)~dt+\sigma_t\left(\frac{1}{N}\sum_{1\leq i\leq N}X^j_{s,t}(x)\right)~dW^i_t
$$
for some Lipschitz functions $B_t(x,y)$ and $\sigma_{t}(y)$ with appropriate dimensions. In the above display, $W^i_t$ stands for a collection of independent copies of  $d$-dimensional Brownian motion  $W_t$. 
Assume that $B_t(x,y)$ linear w.r.t. the first coordinate. 

In this situation, up to a change of probability space, the empirical mean of the process 
$$\overline{X}_{s,t}(x):=\frac{1}{N}\sum_{1\leq i\leq N}X^j_{s,t}(x)$$ satisfies the stochastic differential equation
$$
d\overline{X}_{s,t}(x)=b_t\left(\overline{X}_{s,t}(x)\right)~dt+\frac{1}{\sqrt{N}}~{\sigma}_t\left(\overline{X}_{s,t}(x)\right)~dW_t\quad\mbox{\rm with}\quad b_t(x):=B_t(x,x)
$$
Formally, the above diffusion converges as $N\rightarrow\infty$ to the flow ${X}_{s,t}(x)$ of the dynamical system defined by
$$
\partial_t{X}_{s,t}(x):=b_t\left({X}_{s,t}(x)\right)
$$
More rigorously and without further work, the forward-backward interpolation formula (\ref{Alekseev-grobner}) yields directly the bias-variance error decomposition
$$
  \begin{array}{l}
\displaystyle \overline{X}_{s,t}(x)-X_{s,t}(x) =\frac{1}{2N}~\int_s^t\left(\nabla ^2X_{u,t}\right)(\overline{X}_{s,u}(x))^{\prime}~{a}_u(\overline{X}_{s,u}(x))~du\\
 \\
\hskip5cm\displaystyle+\frac{1}{\sqrt{N}}~\int_s^t~\left(\nabla X_{u,t}\right)(\overline{X}_{s,u}(x))^{\prime}~{\sigma}_u(\overline{X}_{s,u}(x))~dW_u
  \end{array}
  $$
  This readily implies the a.s. convergence
  $$
  \overline{X}_{s,t}(x)\longrightarrow_{N\rightarrow\infty}X_{s,t}(x) 
  $$ 
  After some elementary manipulations we check the bias formula
  $$
  \lim_{N\rightarrow\infty}~N~\left[\EE(\overline{X}_{s,t}(x))-X_{s,t}(x)\right]=\frac{1}{2}~\int_s^t\left(\nabla ^2X_{u,t}\right)(X_{s,u}(x))^{\prime}~{a}_u(X_{s,u}(x))~du
  $$
  We also have the almost sure fluctuation theorem
  $$
    \lim_{N\rightarrow\infty}~\sqrt{N}~\left[\overline{X}_{s,t}(x)-X_{s,t}(x) \right]=\int_s^t~\left(\nabla X_{u,t}\right)(X_{s,u}(x))^{\prime}~{\sigma}_u(X_{s,u}(x))~dW_u
  $$

\subsection{Time discretization schemes}\label{subsec:lemdiscretizeproof}

This section is mainly concerned with the proof of proposition~\ref{lem:discretize}.
We fix some parameter $h>0$ and some $s\geq 0$ and for any $ t\in [s+kh,s+(k+1)h[$ we set
$$
  d X_{s,t}^h(x)=Y^h_{s,t}(x)~dt+  \sigma~  dW_t\quad \mbox{\rm with}\quad Y^h_{s,t}(x):=b\left(X^h_{s,s+kh}(x)\right)
$$
for some fluctuation parameter $\sigma\geq 0$.
For any $ s+kh\leq u<s+(k+1)h$ we  have
$$
X^h_{s,u}(x)-X^h_{s,s+kh}(x)
\displaystyle=Y^h_{s,u}(x)~(u-(s+kh))+\sigma~(W_u-W_{s+kh})
$$
Using (\ref{X-Y-ref-ag}), in terms of the tensor product  (\ref{tensor-notation}) we 
readily check that
$$
X^h_{s,t}(x)-X_{s,t}(x)=\int_s^t~\left(\nabla X_{u,t}\right)(X^h_{s,u}(x))^{\prime}~
~\left[Y^h_{s,u}(x)-b(X^h_{s,u}(x))\right]~
du
$$
Combining (\ref{ref-nablax-estimate-0-ae-again}) with the Minkowski integral inequality we check that
\begin{eqnarray*}
\EE \left( \Vert X^h_{s,t}(x)-X_{s,t}(x)\Vert^n \right)^{1/n}& =&\int_s^t~\EE \left( \Vert \left(\nabla X_{u,t}\right)(X^h_{s,u}(x))^{\prime}~
~\left[Y^h_{s,u}(x)-b(X^h_{s,u}(x)) \right] \Vert^n  \right)^{1/n}~du\\
& = &\int_s^t e^{-\lambda(t-u)}~\EE \left( \Vert Y^h_{s,u}(x)-b(X^h_{s,u}(x))  \Vert^n  \right)^{1/n}~du
\end{eqnarray*}
where the second line follows from the exponential estimate of the tangent process from proposition \ref{prop:diff_innit}. The integrand will be bounded as follows:  for any $ s+kh\leq u<s+(k+1)h$ and any $n\geq 1$ we have 
$$
\EE\left(\Vert b(X^h_{s,u}(x))-Y^h_{s,u}(x))\Vert^n\right)^{1/n}\leq \Vert \nabla b\Vert~\left(\left[\Vert  b(0)\Vert+\widehat{m}_n(x)~\Vert  \nabla b\Vert \right]~h+\sigma~\sqrt{h}\right)
$$
which then yields the stated result of the proposition.
We now prove the stated bound on the difference of the drift processes.  For any $ s+kh\leq u<s+(k+1)h$  we have
 \begin{eqnarray}
&& b(X^h_{s,u}(x))-Y^h_{s,u}(x)\nonumber
\\
&& =\left[\int_{0}^1
\nabla b\left(X^h_{s,s+kh}(x)+\epsilon(X^h_{s,u}(x)-X^h_{s,s+kh}(x))\right)^{\prime}~b\left(X^h_{s,s+kh}(x)\right)~d\epsilon\right]~(u-(s+kh))\nonumber
\\
&&\qquad\hskip1cm+\left[\int_{0}^1
\nabla b\left(X^h_{s,s+kh}(x)+\epsilon(X^h_{s,u}(x)-X^h_{s,s+kh}(x))\right)^{\prime}~
~d\epsilon\right]~\sigma~\left(W_{u}-W_{s+kh}\right)
\label{eq:dicrete_proof_eqn_bias}
 \end{eqnarray}
 The $\LL_n$-norm  of the second integral term is bounded by $\Vert \nabla b\Vert \sigma \sqrt{h}$.
 
The assumption  $\langle x,b(x)\rangle\leq -\beta~\Vert x\Vert^2$, for some $\beta>0$,
implies the stochastic flows $X_{s,t}(x)$ has uniform absolute moments of any order $n\geq 1$ w.r.t. the time horizon, that is,   we have that
$$
m_n(x)\leq \kappa_n~(1+\Vert x\Vert)\quad \mbox{\rm with $m_n(x)$ defined in (\ref{moments-intro})}.
$$
The  stochastic flows $X_{s,t}^h(x)$ also obey a similar moment bound: observe that
for any $ t\in [s+kh,s+(k+1)h[$ we have
$$
\begin{array}{l}
  d \Vert X_{s,t}^h(x)\Vert^2\\
  \\
  \leq \left[-2\lambda_0~\Vert X_{s,t}^h(x)\Vert^2+2~\langle X_{s,t}^h(x), b(X^h_{s,s+kh}(x))-b(X^h_{s,t}(x))\rangle+\sigma^2d\right]~dt+  2\sigma~X_{s,t}^h(x)^{\prime}  dW_t
\end{array}$$
Thus, for any $\epsilon>0$ we have
$$
  d \Vert X_{s,t}^h(x)\Vert^2
  \leq \left[(-2\lambda_0+\epsilon)\Vert X_{s,t}^h(x)\Vert^2+\epsilon^{-1}\Vert\nabla b\Vert+\sigma^2d\right]~dt+  2\sigma~X_{s,t}^h(x)^{\prime}  dW_t
$$
We can check that the stochastic flows $X_{s,t}^h(x)$ also have uniform moments w.r.t. the time horizon; that is,  for any $n\geq 1$ we have that
$$
\widehat{m}_n(x):=\sup_{h\geq 0}\sup_{t\geq s}\EE\left[\Vert X_{s,t}^h(x)\Vert^n\right]^{1/n}\leq c_n~(1+\Vert x\Vert)
$$ Using this bounds, we check that
$$
\EE (\Vert b (X^h_{s,s+kh}(x))\Vert^n)^{1/n} = \Vert b(0) \Vert + \hat{m}_n(x) \Vert \nabla b \Vert 
$$ 
The end of the proof now follows elementary manipulations, thus it is skipped.
The  proof of proposition~\ref{lem:discretize} is now completed.\cqfd

\section*{Appendix}

In this appendix we prove 
 the estimates  (\ref{intro-inq-1}) and (\ref{moments-intro}) and proposition~\ref{prop-nabla-2-estimate}.

\subsection*{Proof of (\ref{moments-intro})}\label{moments-intro-proof}
Whenever $(\Ma)_n$ is satisfied, we have
$$
2\langle x,b_t(x)\rangle+\Vert \sigma_{t}(x)\Vert_F^2\leq  \gamma_0+\gamma_1\Vert x\Vert-\gamma_2\Vert x\Vert^2
$$
with the parameters
$$
\gamma_0=\alpha_0+2\beta_0\qquad \gamma_1=\alpha_1+2\beta_1\quad\mbox{and}\quad\gamma_2=2\beta_2-\alpha_2
$$
Observe that
$$
\begin{array}{l}
d \Vert X_{s,t}(x)\Vert^2\\
\\
=\left[2\,\langle X_{s,t}(x),b_t(X_{s,t}(x))\rangle+\Vert \sigma_{t}(X_{s,t}(x))\Vert_F^2\right]~dt+2\sum_k\langle X_{s,t}(x),\sigma_{k,t}(X_{s,t}(x))\rangle~dW^k_t
\end{array}$$
After some elementary computations, for any $n\geq 1$ we check that
$$
\begin{array}{l}
n^{-1}\partial_t\EE\left[\Vert X_{s,t}(x)\Vert^{2n}\right]
\leq -\left[\gamma_2-2(n-1)\alpha_2\right]~\EE\left[\Vert X_{s,t}(x)\Vert^{2n}\right]\\
\\
\hskip3cm+\left[\gamma_1+2(n-1)\alpha_1\right]~\EE\left[\Vert X_{s,t}(x)\Vert^{2n-1}\right]+\left[\gamma_0+2(n-1)\alpha_0\right]~\EE\left[\Vert X_{s,t}(x)\Vert^{2(n-1)}\right]
\end{array}$$
This implies that
$$
\begin{array}{l}
\partial_t\EE\left[\Vert X_{s,t}(x)\Vert^{2n}\right]^{1/n}
\leq -\left[\gamma_2-2(n-1)\alpha_2\right]~\EE\left[\Vert X_{s,t}(x)\Vert^{2n}\right]^{1/n}\\
\\
\hskip3cm\displaystyle+\left[\gamma_1+2(n-1)\alpha_1\right]~\EE\left[\Vert X_{s,t}(x)\Vert^{2n}\right]^{1/(2n)}+\left[\gamma_0+2(n-1)\alpha_0\right]~\end{array}$$
from which we check that for any $\epsilon>0$ we have
$$
\begin{array}{l}
\partial_t\EE\left[\Vert X_{s,t}(x)\Vert^{2n}\right]^{1/n}\\
\\
\displaystyle\leq -\left[\gamma_2-2(n-1)\alpha_2-2\epsilon\right]~\EE\left[\Vert X_{s,t}(x)\Vert^{2n}\right]^{1/n}+\frac{1}{8\epsilon}\left[\gamma_1+2(n-1)\alpha_1\right]^2+\left[\gamma_0+2(n-1)\alpha_0\right]~\end{array}$$
This implies that
$$
\begin{array}{l}
\partial_t\EE\left[\Vert X_{s,t}(x)\Vert^{2n}\right]^{1/n}\\
\\
\displaystyle\leq -2\left[\beta_2-(n-1/2)\alpha_2-\epsilon\right]~\EE\left[\Vert X_{s,t}(x)\Vert^{2n}\right]^{1/n}+\frac{1}{8\epsilon}\left[\gamma_1+2(n-1)\alpha_1\right]^2+\left[\gamma_0+2(n-1)\alpha_0\right]~\end{array}$$
from which we check that
$$
\EE\left[\Vert X_{s,t}(x)\Vert^{2n}\right]^{1/n}\leq e^{-2\left[\beta_2-(n-1/2)\alpha_2-\epsilon\right](t-s)}~\Vert x\Vert^2+
\frac{1}{8\epsilon}~\frac{\left[\gamma_1+2(n-1)\alpha_1\right]^2+\left[\gamma_0+2(n-1)\alpha_0\right]}{2\left[\beta_2-(n-1/2)\alpha_2-\epsilon\right]}
$$
as soon as $\epsilon<\beta_2-(n-1/2)\alpha_2$ and $n\geq 1$. Replacing $\epsilon$ by $\epsilon(\beta_2-(n-1/2)\alpha_2)$ and then $(2n)$ by $n$ 
we check that
$$
  \begin{array}{l}
\displaystyle\EE\left[\Vert X_{s,t}(x)\Vert^{n}\right]^{1/n}\\
\\
\displaystyle\leq e^{-(1-\epsilon)\beta_2(n)(t-s)}~\Vert x\Vert+
\frac{1}{4\sqrt{\epsilon(1-\epsilon)}}~\frac{\gamma_1(n)+\gamma_0(n)^{1/2}}{\beta_2(n)^{1/2}}
\quad\mbox{\rm with}\quad
\gamma_i(n):=\gamma_i+(n-2)\alpha_i
\end{array}
$$
This ends the proof of (\ref{moments-intro}).\cqfd

\subsection*{Proof of proposition~\ref{prop-nabla-2-estimate}}\label{prop-nabla-2-estimate-proof}
The proof of the estimate (\ref{eq-prop-nabla-2-estimate}) is mainly based on the following technical lemma of its own interest.
\begin{lem}
Let $Z_t$ be a non negative diffusion process satisfying in integral sense an inequality of the following form
$$
dZ_t\leq (-\lambda Z_t+\alpha_t~\sqrt{Z_t}+\beta_t)~dt+dM_t\quad \mbox{\rm with}\quad \partial_t\langle M\rangle_t\leq (u_t \sqrt{Z_t}+v_tZ_t)^2
$$
for some parameters $\lambda>0$ and $v_t\geq 0$, and some non negative processes $(\alpha_t,\beta_t,u_t)$. In this situation, for any $\epsilon>0$ we have
\begin{equation}\label{key-est-y_n}
\EE(Z^n_t)^{1/n}\leq e^{\int_0^t \lambda_{n,s}(\epsilon)ds}~\EE(Z^n_0)^{1/n}+\int_0^te^{\int_s^t \lambda_{n,u}(\epsilon)du}~z^n_s(\epsilon)~ds
\end{equation}
with the parameters
\begin{eqnarray*}
\lambda_{n,t}(\epsilon)&:=&-\lambda+\frac{n-1}{2}~v^2_t+\frac{\epsilon}{2}
\\
\displaystyle z^n_t(\epsilon)&:=&
\EE\left[\beta_t^n\right]^{1/n}+\frac{n-1}{2}~\EE\left[ u^{2n}_t\right]^{1/n}+\frac{1}{\epsilon}~
\left(\EE\left[ \alpha_t^{2n}\right]^{1/n}+(n-1)^2~\EE\left[(u_tv_t)^{2n}\right]^{1/n}\right)
\end{eqnarray*}
\end{lem}
\proof
Applying It\^o's formula, for any $n\geq 2$, we have
$$
\begin{array}{l}
\displaystyle n^{-1}\partial_t\EE(Z^n_t)\\
\\
\displaystyle \leq \EE\left[Z^{n-1}_t(-\lambda Z_t+\alpha_t~\sqrt{Z_t}+\beta_t)+\frac{n-1}{2}~(u_t \sqrt{Z_t}+v_tZ_t)^2~Z^{n-2}_t\right]\\
\\
\displaystyle=\left(-\lambda+\frac{n-1}{2}~v^2_t\right)~\EE(Z^n_t)+\EE\left[\left(\beta_t+\frac{n-1}{2}~u^2_t\right)Z^{n-1}_t\right]+ \EE\left(\left[\alpha_t+(n-1)u_tv_t\right]~Z^{n-1/2}_t\right)
\end{array}
$$
On the other hand, for any $\epsilon>0$ we have the almost sure inequality
$$
\left[\alpha_t+(n-1)u_tv_t\right]~Z^{(n-1)/2}_t~Z^{n/2}_t\leq \frac{1}{2\epsilon}~\left[\alpha_t+(n-1)u_tv_t\right]^2~Z^{n-1}_t+\frac{\epsilon}{2}~Z^{n}_t
$$
This implies that
$$
\begin{array}{l}
\displaystyle n^{-1}\partial_t\EE(Z^n_t)\\
\\
\displaystyle \leq \lambda_{n,t}(\epsilon)~\EE(Z^n_t)+\EE\left[\left(\beta_t+\frac{n-1}{2}~u^2_t+\frac{1}{2\epsilon}~\left[\alpha_t+(n-1)u_tv_t\right]^2\right)Z^{n-1}_t\right]\end{array}
$$
Applying H\"older inequality we check that
$$
\begin{array}{l}
\displaystyle\EE\left[\left(\beta_t+\frac{n-1}{2}~u^2_t+\frac{1}{2\epsilon}~\left[\alpha_t+(n-1)u_tv_t\right]^2\right)~Z^{n-1}_t\right]\\
\\
\displaystyle\leq 
\EE\left[\left(\beta_t+\frac{n-1}{2}~u^2_t+\frac{1}{2\epsilon}~\left[\alpha_t+(n-1)u_tv_t\right]^2\right)^n\right]^{1/n}
~\EE(Z^n_t)^{1-1/n}\leq 
 z^n_t
~\EE(Z^n_t)^{1-1/n}
\end{array}
$$
This yields the estimate
$$
 \partial_t\EE(Z^n_t)^{1/n}=\EE(Z^n_t)^{-(1-1/n)}~n^{-1}\partial_t\EE(Z^n_t)\leq   \lambda_{n,t}(\epsilon)~\EE(Z^n_t)^{1/n}+ z^n_t
$$
This ends the proof of the lemma.
\cqfd

We set
$$
Y_{s,t}(x):=\Vert \nabla^2 X_{s,t}(x)\Vert^2_F\quad \mbox{\rm and}\quad T_{s,t}(x):=\Vert \nabla X_{s,t}(x)\Vert_F
$$
and we also consider the collection of parameters
$$
\begin{array}{rclcrcl}
\Vert \tau\Vert_F&:=&\sup_{t,x}\Vert \tau_t(x)\Vert_F&&
\rho( \upsilon)&:=&\sup_{t,x}\lambda_{\tiny max}(\upsilon_{t}(x))
\end{array}
$$
with the tensor functions $(\tau_t,\upsilon_t)$ introduced in  (\ref{tensor-functions-ref}). Observe that
$$
\Vert \tau\Vert_F\leq \Vert \nabla^2b\Vert_F+d~ \Vert \nabla^2\sigma\Vert_{F}^2\quad \mbox{\rm and}\quad
\rho( \upsilon)\leq d~\Vert \nabla^2\sigma\Vert^2_{2}
$$
Whenever $(\Ta)_2$ is met we have
$$
\tr\left[\nabla^2 X_{s,t}(x)~A_t(X_{s,t}(x))~\nabla^2 X_{s,t}(x)^{\prime}\right]\leq -2\lambda_A~Y_{s,t}(x)
$$
Also observe that
$$
\vert\tr\left[\left[\nabla X_{s,t}(x)\otimes \nabla X_{s,t}(x)\right]~\tau_t(X_{s,t}(x))~\nabla^2 X_{s,t}(x)^{\prime}\right]\vert\leq \Vert \tau\Vert_F~
Y_{s,t}(x)^{1/2}~T_{s,t}(x)^2
$$
and
$$
\tr\left[\left[\nabla X_{s,t}(x)\otimes \nabla X_{s,t}(x)\right]\upsilon_t(X_{s,t}(x))\left[\nabla X_{s,t}(x)\otimes \nabla X_{s,t}(x)\right]^{\prime}\right]\leq \rho( \upsilon)~T_{s,t}(x)^4
$$

In the same vein, we have
$$
 \begin{array}{l}
\vert \tr
\left\{\left[\nabla X_{s,t}(x)\otimes \nabla X_{s,t}(x)\right]~\nabla^2\sigma_{t,k}(X_{s,t}(x))~\nabla^2 X_{s,t}(x)^{\prime}\right.\\
\displaystyle\hskip7cm\left.+\nabla^2 X_{s,t}(x)~\nabla \sigma_{t,k}(X_{s,t}(x))~\nabla^2 X_{s,t}(x)^{\prime}\right\}\vert\\
\\
\leq \Vert \nabla^2\sigma_k\Vert_F~T_{s,t}(x)^2~Y_{s,t}(x)^{1/2}+\rho(\nabla \sigma_k)~Y_{s,t}(x)
\end{array}
$$

We are now in position to prove proposition~\ref{prop-nabla-2-estimate}.

{\bf Proof of proposition~\ref{prop-nabla-2-estimate}:} 

Applying the above lemma to the processes
$$
Z_t=Y_{s,t}(x)\qquad \lambda=2\lambda_A\qquad \alpha_t=2\Vert \tau\Vert_F~T_{s,t}(x)^2\qquad \beta_t=\rho( \upsilon)~T_{s,t}(x)^4
$$
and the parameters
$$
u_t=2\sqrt{d}~ \Vert \nabla^2\sigma\Vert_{F}~T_{s,t}(x)^2\quad\mbox{\rm and}\quad v_t=2\sqrt{d}~ \rho_{\star}(\nabla \sigma)
$$
we obtain the estimate (\ref{key-est-y_n}) with the parameters
\begin{eqnarray*}
\lambda_{n,t}(\epsilon)&:=&-2\left[\lambda_A-d(n-1) \rho_{\star}(\nabla \sigma)^2-\frac{\epsilon}{4}\right]
\\
\displaystyle z^n_t(\epsilon)&:=&\left\{\rho( \upsilon)~
+2d(n-1)~\Vert \nabla^2\sigma\Vert_{F}^2~\right.\left.+\frac{4}{\epsilon}~
\left(
\Vert \tau\Vert_F^2~
+4~d^2(n-1)^2~ \rho_{\star}(\nabla \sigma)^2~\Vert \nabla^2\sigma\Vert_{F}^2~\right)\right\}\\
&&\hskip3cm\times \EE\left[\Vert \nabla X_{s,t}(x)\Vert_F^{4n}\right]^{1/n}
\end{eqnarray*}
Observe that
$$
\begin{array}{l}
\displaystyle
z^n_t(\epsilon) 
\leq c n^2~(1\vee\epsilon^{-1})~\rchi(b,\sigma)^2~\EE\left[\Vert \nabla X_{s,t}(x)\Vert_F^{4n}\right]^{1/n}\\
\\
\end{array}
$$
for some universal constant $c<\infty$ and the parameter $\rchi(b,\sigma)$ defined in (\ref{def-chi-b}).
Using (\ref{def-4th}) we check that
$$
\begin{array}{l}
\displaystyle\EE\left(\Vert \nabla^2 X_{s,t}(x)\Vert^{2n}_F\right)^{1/n}\\
\\
\displaystyle\leq  c n^2~(1\vee\epsilon^{-1})~\rchi(b,\sigma)^2~ \int_s^te^{-2\left[\lambda_A-d(n-1) \rho_{\star}(\nabla \sigma)^2-\frac{\epsilon}{4}\right](t-u)}~e^{-4\left[\lambda_A-(n-1)\rho(\nabla\sigma)^2\right](u-s)}~du\\
\\
\displaystyle= c  n^2~(1\vee\epsilon^{-1})~\rchi(b,\sigma)^2~e^{-2\left[\lambda_A-d(n-1) \rho_{\star}(\nabla \sigma)^2-\frac{\epsilon}{4}\right](t-s)}~\\
\\
\hskip3cm\displaystyle \int_s^te^{-2\left[
\lambda_A-(n-1)\rho(\nabla\sigma)^2+(n-1) [d\rho_{\star}(\nabla \sigma)^2-\rho(\nabla\sigma)^2]+\frac{\epsilon}{4}\right](u-s)}~du
\end{array}$$
Assume that
$$
\lambda_A>d(n-1) \rho_{\star}(\nabla \sigma)^2
$$
In this case there exists some $0<\epsilon_n\leq 1$ such that for any $0<\epsilon\leq \epsilon_n$ we have
$$
\lambda_A-d(n-1) \rho_{\star}(\nabla \sigma)^2>\epsilon
$$
and therefore
$$
\displaystyle\EE\left(\Vert \nabla^2 X_{s,t}(x)\Vert^{2n}_F\right)^{1/(2n)}\leq  c ~n~\epsilon^{-1}~\rchi(b,\sigma)~\exp{\left(-\left[\lambda_A-d(n-1) \rho_{\star}(\nabla \sigma)^2-\epsilon\right](t-s)\right)}
$$
This ends the proof of the proposition.\cqfd

\subsection*{Proof of (\ref{intro-inq-1})}\label{intro-inq-1-proof}
Using (\ref{intro-inq-nabla}), the generalized Minkowski inequality applied to (\ref{Alekseev-grobner})  whenever $(\Ta)_{n/\delta}$ is met for some $\delta\in ]0,1[$ and $n\geq 2$ gives
  \begin{equation}\label{intro-inq-0}
     \begin{array}{l}
\displaystyle 
 \EE\left[\Vert    T_{s,t}(\Delta a,\Delta b)(x)\Vert^n\right]^{1/n}\\
 \\
\displaystyle   \leq \frac{\kappa_{n/\delta}}{\lambda(n/\delta)}~\left(\vertiii{\Delta b(x)}_{n/(1-\delta)}+\vertiii{\Delta a(x)}_{n/(1-\delta)}\right)\quad \mbox{\rm 
with  $(\kappa_n,\lambda(n))$ given in (\ref{ref-tan-hess}). }
     \end{array}
 \end{equation}
 The Skorohod integral $S_{s,t}(\Delta \sigma)(x)$ is estimated using theorem~\ref{theo-quantitative-sko}.
Using (\ref{intro-inq-0}) and (\ref{intro-inq-s}) we check that
$$
   \begin{array}{l}
\displaystyle
  \EE\left[\Vert   X_{s,t}(x)-\overline{X}_{s,t}(x)\Vert^{n}\right]^{1/n}\\
  \\
\displaystyle\leq  \kappa_{(\delta_1,\delta_2),n}~\left(\vertiii{\Delta a(x)}_{n/(1-\delta_1)}+\vertiii{\Delta b(x)}_{n/(1-\delta_1)}
+\vertiii{\Delta\sigma(x)}_{2n/\delta_2}~(1\vee\Vert x\Vert)\right)\end{array} 
$$
as soon as the regularity conditions 
$(\Ta)_{n/\delta_1}$,
$(M)_{2n/\delta_2}$ and  $(T)_{2n/(1-\delta_2)}$ are satisfied for some parameter $n\geq 2$ and some $\delta_1,\delta_2\in ]0,1[$. Choosing $\delta_1=(1-\delta_2)/2$ and setting $\delta=\delta_2$ we check that
$$
   \begin{array}{l}
\displaystyle
  \EE\left[\Vert   X_{s,t}(x)-\overline{X}_{s,t}(x)\Vert^{n}\right]^{1/n}\\
  \\
\displaystyle\leq  \kappa_{\delta,n}~\left(\vertiii{\Delta a(x)}_{2n/(1+\delta)}+\vertiii{\Delta b(x)}_{2n/(1+\delta)}
+\vertiii{\Delta\sigma(x)}_{2n/\delta}~(1\vee\Vert x\Vert)\right)\end{array} 
$$
as soon as $(M)_{2n/\delta}$ and  $(T)_{2n/(1-\delta)}$ are satisfied for some parameter $n\geq 2$ and some $\delta\in ]0,1[$. 
For instance, $(\Ma)_{2n/\delta}$ and  $(\Ta)_{2n/(1-\delta)}$ are satisfied as soon as
$$
 \beta_2-\alpha_2/2>(n/\delta-1)~\alpha_2\quad\mbox{\rm and}\quad \lambda_A>d(n/(1-\delta)-1)~ \rho_{\star}(\nabla \sigma)^2
$$
This ends the proof of (\ref{intro-inq-1}).\cqfd

\section*{Acknowledgments}
P. Del Moral is supported in part from the Chair Stress Test, RISK Management and Financial Steering, led by the French Ecole polytechnique and its Foundation and sponsored by BNP Paribas, and by the ANR Quamprocs on quantitative analysis of metastable processes.

We also thank the anonymous reviewers for their excellent suggestions for improving the paper.
Their detailed comments greatly improved the 
presentation of the article.


\begin{thebibliography}{99}
 
 \bibitem{Alekseev}
 V. Alekseev.  An estimate for the perturbations of the solution of ordinary differential equations. Vestn. Mosk.Univ., Ser. I, Math. Meh. vol. 2, (1961).
 
 \bibitem{hemmen}
T. Ando and J. L. van Hemmen. An inequality for trace ideals. Commun. Math. Phys., vol.
76, pp. 143--148 (1980).

 
   \bibitem{mp-var-18}
M.  Arnaudon, P. Del Moral. A variational approach to nonlinear and interacting diffusions.
	\href{https://arxiv.org/pdf/1812.04269.pdf}{ArXiv:1812.04269  (2018).}
Stochastic Analysis and Applications DOI: 10.1080/07362994.2019.1609985 (2019).

   \bibitem{mp-dualtiy}
M.  Arnaudon, P. Del Moral. A duality formula and a particle Gibbs sampler for continuous time Feynman-Kac measures on path spaces. ArXiv 1805.05044 (2018). Electronic Journal of Probability 25 (2020).

   \bibitem{mp-var-19}
M. Arnaudon, P. Del Moral. A second order analysis of McKean-Vlasov semigroups. \href{https://arxiv.org/abs/1906.05140}{\tt ArXiv:1906.05140  (2019}), Annals of Applied Probability, vol. 30, no. 6, pp. 2613--2664. (2020).

  \bibitem{aht-03}
M. Arnaudon, H. Plank, A. Thalmaier.
A Bismut type formula for the Hessian of heat semigroups.
C. R. Math. Acad. Sci. Paris, vol. 336, no. 8, pp. 661--666 (2003).

\bibitem{bellman}
R. Bellman.  Some inequalities for the square Root of a Positive Definite Matrix. Linear Algebra and its applications, vol. 1, no. 3, pp. 321--324 (1968).

\bibitem{bellman-2}
R. Bellman, Stability Theory of Differential Equations, McGraw Hill, New York, (1953).

\bibitem{bishop-stab}
 A. N. Bishop, P. Del Moral.
On the Stability of Matrix-Valued Riccati Diffusions. \href{https://arxiv.org/abs/1808.00235}{ArXiv:1808.00235 (2018)}.Electron. J. Probab., vol. 24, no. 84, 40 pp. (2019). 

\bibitem{bishop-18}
A.N. Bishop,  P. Del Moral, S.D. Pathiraja. Perturbations and projections of Kalman-Bucy semigroups. Stochastic Processes and their Applications, vol. 9, no.128, pp. 2857--2904  (2018).

\bibitem{bishop-19}
A. N. Bishop, P. Del Moral, A. Niclas.
A perturbation analysis of stochastic matrix Riccati diffusions. \href{https://arxiv.org/abs/1709.05071}{\tt Arxiv 1709.05071 (2017)}.  Ann. Inst. H. Poincar\'e Probab. Statist., vol. 56, no. 2, pp. 884--916.(2020).

\bibitem{bismut-2}
J.M.
Bismut.  A generalized formula of It\^o and some other properties of stochastic
flows, Z. Wahrschein. Werw. Geb., vol. 55,  pp. 331--350 (1981).


\bibitem{bismut}
J.M.
Bismut. Large deviations and the Malliavin calculus. Birkhauser Prog. Math. 45 (1984).


\bibitem{carverhill}
A.P. Carverhill and K.D. Elworthy. Flows of Stochastic Dynamical Systems: The Functional Analytic Approach. Z. Wahrs 65, pp. 245--267  (1983).

\bibitem{coppel1978stability}
W.A. Coppel. Dichotomies in Stability Theory. Springer (1978).

  \bibitem{daprato-2}
 G. Da Prato, 
J.L.  Menaldi, L.  Tubaro. Some results of backward It\^o formula. Stochastic analysis and applications, vol. 25, no. 3, pp. 679--703 (2007).

  \bibitem{daprato-3}
 G. Da Prato. Some remarks about backward It\^o formula and applications. Stochastic analysis and applications, vol. 16, no. 6, pp. 993--1003 (1998).

 \bibitem{d-2004}
P. {Del Moral}.
\newblock Feynman-{K}ac formulae.
\newblock  Genealogical and interacting particle systems with applications.
\newblock  Probability and its Applications (New York). (573p.) Springer-Verlag, New
  York (2004).
  
   \bibitem{d-2013}
  Del Moral, Pierre. Mean field simulation for Monte Carlo integration. Chapman and Hall, CRC press. Monographs on Statistics and Applied Probability (2013).
  
  \bibitem{dm-g-99}
P. Del Moral, A.  Guionnet. On the stability of measure valued processes with applications to filtering. Comptes Rendus de l'Acad\'emie des Sciences-Series I-Mathematics, vol. 329, no. 5, pp. 429--434 (1999).

\bibitem{guionnet}
P. Del Moral and A. Guionnet. On the stability of interacting processes with applications to filtering and genetic algorithms. Ann. Inst. Henri Poincar\'e, vol. 37, no. 2, pp. 155--194 (2001).

 \bibitem{dm-2000}
P.~Del~Moral and L.~Miclo.
\newblock Branching and interacting particle systems ap\-pro\-ximations of
  {F}eynman-{K}ac formulae with applications to non-linear filtering.
\newblock In {\em S\'eminaire de {P}robabilit\'es, {XXXIV}}, volume 1729,
  {\em Lecture Notes in Math.}, pages 1--145. Springer, Berlin (2000).
  
    \bibitem{Elworthy}
K.D. Elworthy, X.M. Li. Formulae for the Derivative of Heat Eemigroups. Journal of Functional Analysis 125, pp. 252--286 (1994).


\bibitem{grobner}
 Gr\"obner, W. Die Lie-Reihen und ihre Anwendungen. VEB Deutscher Verlag der Wiss., Berlin (1960).
 
 \bibitem{gronwall}
 T.H. Gronwall, Note on the derivatives with respect to a parameter of the solutions of a system of differential equations, Ann. Math., vol. 20, no. 2 , pp. 
 293--296 (1919).
 
 \bibitem{higham}
N. J. Higham. Functions of Matrices : Theory and Computation, SIAM, Philadelphia, PA (2008).


\bibitem{hudde}
A. Hudde, M. Hutzenthaler, A. Jentzen, S. Mazzonetto.  On the It\^ o-Alekseev-Gr\" obner formula for stochastic differential equations. arXiv preprint arXiv:1812.09857 (2018).

\bibitem{hutz-14}
M. Hutzenthaler and A. Jentzen. On a perturbation theory and on strong convergence rates for stochastic ordinary and partial differential equations with non-globally monotone coefficients. \href{http://arxiv.org/abs/1401.0295}{\tt Arxiv 1401.0295 (2014)}. To appear in the Annals of Probability (2019)

\bibitem{iserles}
A. Iserles, G. S\"oderlind. Global bounds on numerical error for ordinary differential equations. Journal of Complexity, vol. 9, no. 1, pp. 97--112 (1993).


\bibitem{jentzen}
A. Jentzen, F. Lindner, P. Pusnik. On the Alekseev-Gr\"obner formula in Banach spaces. arXiv preprint arXiv:1810.10030 (2018).	


\bibitem{kunita-2}
H. Kunita. First order stochastic partial differential equations. Stochastic Analysis (K. It\o ed.). Kinokunniya. Tokyo, pp. 249-269 (1984).


\bibitem{kunita}
H. Kunita, and M. K. Ghosh. Lectures on stochastic flows and applications. Bombay: Tata Institute of Fundamental Research (1986).

\bibitem{krylov}
N.V. Krylov. B.L. Rozowskii. On the first integrals and Liouville equations for diffusion processes. Stochastic Differential Systems. 
Proc. 3rd SFSP-WG 7/1, Visegrad, Hungary 1980. Lecture Notes in Control and Information Sciences, vol. 36, pp. 117-125 (1981).

\bibitem{kruk}
I. Kruk, F. Russo, C. A. Tudor. Wiener integrals, Malliavin calculus and covariance measure structure. Journal of Functional Analysis, vol. 249, no. 1, pp. 92--142 (2007).

\bibitem{leon}
J.A. Leon. Fubini theorem for anticipating stochastic integrals in Hilbert space. Applied Mathematics and Optimization, vol. 27, no. 3, pp. 313--327 (1993).

\bibitem{xm-li}
X. M.
Li. Doubly Damped Stochastic Parallel Translations and Hessian Formulas. International Conference on Stochastic Partial Differential Equations and Related Fields. Springer, Cham (2016).

\bibitem{nualart}
D. Nualart. The Malliavin calculus and related topics. Vol. 1995. Berlin: Springer (2006).

\bibitem{nualart-pardoux}
D. Nualart, E. Pardoux. Stochastic calculus with anticipating integrands. Probability Theory and Related Fields, vol. 78, no. 4, pp. 535--581 (1988).

\bibitem{nualart-z}
D. Nualart, M. Zakai.  Generalized multiple stochastic integrals and the representation of wiener functionals. Stochastics, vol. 23, no. 3, pp. 
311--330 (1988).

\bibitem{norris}
J. R. Norris.
Simplified Malliavin calculus.
S\'eminaire de probabilit\'es (Strasbourg), tome 20, p. 101--130 (1986).


  \bibitem{ocone-pardoux}
D. Ocone, E. Pardoux. A generalized It\^o-Ventzell formula. Application to a class of anticipating stochastic differential equations. In Annales de l'IHP Probabilit\'es et statistiques, vol. 25, no. 1, pp. 39--71 (1989).

  \bibitem{pardoux-90}
E. Pardoux. Applications of anticipating stochastic calculus to stochastic differential equations. In Stochastic Analysis and Related Topics II (pp. 63-105). Springer, Berlin, Heidelberg (1990).

 \bibitem{pardoux-protter}
 E. Pardoux, P. Protter. A two-sided stochastic integral and its calculus.
 Probability Theory and Related Fields, vol. 76, no. 1, pp. 15--49 (1987).
 
 
 \bibitem{purtu}
 O. Purtukhia. Fubini type theorems for ordinary and stochastic integrals. Proceedings of A. Razmadze Mathematical Institute, 
vol. 130, pp.  101--114 (2002).
 
  \bibitem{scheutzow}
 M. Scheutzow.  A stochastic Gronwall lemma. Infinite Dimensional Analysis, Quantum Probability and Related Topics, vol. 16, no. 2, p. 1350019 (2013).
 
 \bibitem{thompson}
J. Thompson. Derivatives of Feynman-Kac semigroups. Journal of Theoretical Probability, vol. 32, no. 2, pp. 950--973 (2019).


\bibitem{tim}
J. Timmer, S. Haussler, M. Lauk, and C.-H. Lucking, Pathological tremors: Deterministic chaos or nonlinear stochastic oscillators. Chaos: An Interdisciplinary Journal of Nonlinear Science, vol. 10 no. 1, pp. 278--288  (2000).
 
  \bibitem{watanabe}
S. Watanabe. Lectures on stochastic differential equations and Malliavin Calculus. Tata Institute of Fundamental Research. Springer-Verlag (1984).

\bibitem{wu1974note}
M. Wu. A note on stability of linear time-varying systems. IEEE Transactions on Automatic Control. vol. 19, no. 2. pp. 162--162 (1974).

\end{thebibliography}
\end{document}